\documentclass{amsart}
\RequirePackage{amsmath,amssymb,amsthm, amscd, comment,mathtools}

\usepackage{kantlipsum}
\usepackage[all]{xy}
\usepackage{mathabx}
\usepackage{graphicx}

\usepackage[colorlinks,linkcolor=red,anchorcolor=red,citecolor=blue]{hyperref}

\calclayout

\newcommand{\F}{\mathcal{F}}
\newcommand{\G}{\mathcal{G}}
\newcommand{\pr}{\mathrm{pr}}
\newcommand{\ord}{\mathrm{ord}}

\newcommand{\rank}{\mathrm{rank}}
\newcommand{\Sw}{\mathrm{Sw}}
\newcommand{\dimtot}{\mathrm{dimtot}}

\newcommand{\bQl}{\overline{\mathbb Q}_\ell}
\newcommand{\bZl}{\overline{\mathbb Z}_\ell}
\newcommand{\bFl}{\overline{\mathbb F}_\ell}

\providecommand{\abs}[1]{\lvert#1\rvert}

\newtheorem{definition}[subsection]{Definition}
\newtheorem{theorem}[subsection]{Theorem}
\newtheorem{proposition}[subsection]{Proposition}
\newtheorem{lemma}[subsection]{Lemma}
\newtheorem{remark}[subsection]{Remark}
\newtheorem{corollary}[subsection]{Corollary}
\newtheorem{conjecture}[subsection]{Conjecture}
\numberwithin{equation}{subsection}

\begin{document}
\title[Characteristic class and $\varepsilon$-factor]{Characteristic class and the $\varepsilon$-factor of an \'etale sheaf}
\date{\today}
\author{Naoya Umezaki}

\address{Graduate School of Mathematical Sciences, The University of Tokyo, 3-8-1 Komaba, Meguro-ku, Tokyo, 153-8914, Japan}
\email{umezaki@ms.u-tokyo.ac.jp, umezakinaoya@gmail.com}

\author{Enlin Yang}
\address{Fakult\"at f\"ur Mathematik, Universit\"at Regensburg, 93040 Regensburg, Germany}
\email{enlin.yang@mathematik.uni-regensburg.de, yangenlin0727@126.com}

\author{Yigeng Zhao}
\address{Fakult\"at f\"ur Mathematik, Universit\"at Regensburg, 93040 Regensburg, Germany}
\email{yigeng.zhao@mathematik.uni-regensburg.de}

\begin{abstract}
We prove a twist formula for the $\varepsilon$-factor of a constructible sheaf on a projective smooth variety over a finite field in terms of characteristic class of the sheaf. 
This formula is a modified version of the formula conjectured by Kato and Saito in \cite[Ann. Math., 168 (2008):33-96, Conjecture 4.3.11]{Kato_Saito}.

We give two applications of the twist formula. Firstly, we prove that the characteristic classes of  constructible \'etale  sheaves on projective smooth varieties over a finite field are compatible with proper push-forward.
Secondly, we show that the two Swan classes in the literature are the same on proper smooth surfaces over a finite field.
\end{abstract}

\maketitle

\section{Introduction}
\subsection{}Let $k$ be a finite field of characteristic $p$ and let $f\colon X\rightarrow {\rm Spec}k$ be a smooth projective  scheme purely of dimension $d$.
Let $\Lambda$ be a finite field of characteristic $\ell\neq p$ or $\Lambda=\overline{\mathbb Q}_\ell$.
For  $\F\in D_c^b(X,\Lambda)$, let $D(\F)$ be the dual  $R\mathcal Hom(\F,\mathcal K_X)$ of $\F$, where $\mathcal K_X=Rf^!\Lambda$ is the dualizing complex. Let $\chi(X_{\bar{k}},\F)$ be the Euler characteristic of $\F$.
The $L$-function $L(X,\F, t)$ satisfies the following functional equation
\begin{equation}\label{eqYZ:fe}
L(X,\F, t)=\varepsilon(X,\F)\cdot t^{-\chi(X_{\bar{k}},\F)}\cdot L(X, D(\F),t^{-1}),
\end{equation}
where 
\begin{equation}\label{eq:introep00}
\varepsilon(X,\F)=\det(-\mathrm{Frob}_k;R\Gamma(X_{\bar{k}},\F))^{-1}
\end{equation}
is the epsilon factor (the constant term of the functional equation (\ref{eqYZ:fe})) and $\mathrm{Frob}_k$ is the geometric Frobenius (the inverse of the Frobenius substitution).

\subsection{}In (\ref{eqYZ:fe}), both $\chi(X_{\bar{k}},\F)$ and $\varepsilon(X,\F)$ are related to ramification theory. 
Let $cc_X\F\in CH_0(X)$ be the characteristic class of $\F$ defined in \cite[Definition 5.7]{Saito} using the characteristic cycle $CC\F$ of $\F$.
Then $\chi(X_{\bar{k}},\F)={\rm deg}(cc_X\mathcal F)$ by the index formula \cite[Theorem 7.13]{Saito}. 
However, the relation between  $\varepsilon(X,\F)$ and $cc_X\mathcal F$ is more subtle.

\subsection{}
Let $\bar{x}$ be a geometric point of $X$.
The reciprocity map $CH_0(X)\to\pi_1^{\rm ab}(X,\bar x)$ is defined by sending the class $[s]$ of a closed point $s\in X$ to the geometric Frobenius $\mathrm{Frob}_s$.
Let $\rho$ be a continuous representation of $\pi_1(X, \bar{x})$ over $\Lambda$ of finite dimension.
We also denote by $\det\rho\colon CH_0(X)\to\Lambda^\times$  the composition of $\det\rho$ and the reciprocity map $CH_0(X)\to\pi_1^{\rm ab}(X,\bar x)$.

\subsection{}
In this paper, we prove the following Theorem \ref{main}, which is a modified version of the formula conjectured by Kato and T. Saito in \cite[Conjecture 4.3.11]{Kato_Saito}.
We note that their formula is written in terms of the Swan class of $\F$. In Section \ref{sec:sc}, we prove that the Swan class can be written in terms of characteristic class on proper smooth surfaces over a finite field (cf.  Corollary \ref{cor:2surf}).
\begin{theorem}[see Theorem \ref{thm:2}]\label{main}
We have
\begin{equation}\label{eqYZ:ep}
\det\rho(-cc_X\F)=\frac{\varepsilon(X,\F\otimes\rho)}{\quad\varepsilon(X,\F)^{\dim\rho}~}\qquad{\rm in}~\Lambda^\times.
\end{equation}
\end{theorem}
\subsection{}We call (\ref{eqYZ:ep}) the twist formula  for the $\varepsilon$-factor. This is a global analogue of a twist formula for local $\varepsilon$-factors proved by Deligne and Henniart \cite{DH}.
When $\F$ is the constant sheaf $\Lambda$, this is proved by S. Saito \cite{Saito:1984}. If $\mathcal F$ is a  smooth sheaf  on an open dense subscheme $U$ of $X$ such that the complement $D=X- U$ is a simple normal crossing divisor and the sheaf $\mathcal F$ is tamely ramified along $D$,  then Theorem \ref{main} is a consequence of \cite[Theorem 1]{Sai93ep}.
If $\dim X=1$, the formula (\ref{eqYZ:ep}) follows from the product formula of Deligne and Laumon, cf. \cite[7.11]{Deligne} and \cite[3.2.1.1]{Laumon}.
In \cite{vi1,vi2}, Vidal proved a similar result on a proper smooth surface over a finite field of characteristic $p>2$ under one of the following two assumptions:
\begin{enumerate}
\item the sheaf $\mathcal F$ is smooth of rank 1 on an open dense subscheme $U\subseteq X$ such that the corresponding character is of order $np$ with $(n,p)=1$ and $D=X- U$ is a simple normal crossing divisor;
\item the wild ramification of $\mathcal F$ is totally non-fierce, cf. \cite[Th\'eor\`eme 2.2]{vi2}.
\end{enumerate}
In our paper, by using the recent development of ramification theory made by Beilinson \cite{Beilinson} and T. Saito \cite{Saito}, we proved Kato-Saito's conjectural formula (\ref{eqYZ:ep}) without any assumption on the ramification of $\mathcal F$ or the dimension of $X$.

\subsection{}
In \cite{beilinson07},  Beilinson developed the theory of topological epsilon factors using $K$-theory spectrum. Let $R$ be a commutative ring. Let $\mathcal F$ be a perfect constructible complex
of sheaves of $R$-modules on a compact real analytic manifold $M$. In loc.cit., he gave a Dubson-Kashiwara-style description of $\det R\Gamma(M,\mathcal F)$, and he asked  whether his construction  admits a motivic ($\ell$-adic or de Rham) counterpart.
For de Rham cohomology, such a construction is given by  Patel in \cite{Pat12,Pat17b}.
Based on \cite{Pat12,Pat17b},  Abe and Patel proved a similar twist formula in \cite{Pat17a} for global de Rham epsilon factors in the classical setting of $\mathcal D_X$-modules on smooth projective varieties over a field of characteristic zero.  

\subsection{}
As a corollary of Theorem \ref{main}, we prove the following compatibility of characteristic classes with proper push-forward, which is a consequence of the conjectures of T.~Saito (\cite[7.2]{Saito} and \cite[Conjecture 1]{Saito17}).
\begin{corollary}[see Corollary \ref{corYZ:2}]\label{corYZ:1}
Let $f:X\to Y$ be a proper map between smooth projective connected schemes over { a finite field} $k$ and  $\F\in D_c^b(X,\Lambda)$.
Then we have an equality in $CH_0(Y)$:
\begin{equation}\label{eqYZ:corep}
f_*(cc_X\F)=cc_Y Rf_*\F.
\end{equation}
\end{corollary}
In particular, if $Y={\rm Spec}(k)$, then we have $\chi(X_{\bar k},\mathcal F)={\rm deg}(cc_X\mathcal F)$ by (\ref{eqYZ:corep}). 
This gives another proof of the index formula \cite[Theorem 7.13]{Saito} when $k$ is a  finite field. 
If $X$ is a projective smooth curve and $Y={\rm Spec}(k)$, the formula (\ref{eqYZ:corep}) is nothing other than
 the well-known Grothendieck-Ogg-Shafarevich formula for curves \cite{SGA5}.

\subsection{}
For any smooth scheme $X$ over a perfect field $k$, 
let $K(X,\Lambda)$ be the Grothendieck group of the triangulated category $D_c^b(X,\Lambda)$. We have the following linear morphism (cf. \cite[Definition 6.7]{Saito17})
\[cc_{X,\bullet}\colon K(X,\Lambda)\rightarrow CH_\bullet(X)=\bigoplus_{i=0}^{\dim X}CH_i(X)\]
which sends the class of $\mathcal F$ to the (total) characteristic class of $\mathcal F$.  By  \cite[Lemma 6.9]{Saito}, $cc_X(\mathcal F)$ is the dimension 0-part of the total characteristic class $cc_{X,\bullet}(\mathcal F)$, i.e., $cc_X(\mathcal F)=cc_{X,0}(\mathcal F)$. Following Grothendieck \cite{SGA5}, it's natural to ask the following question:
is the diagram
\begin{align}\label{eq:fccf}
\begin{gathered}
\xymatrix{
K(X,\Lambda)\ar[d]_{f_\ast}\ar[r]^{cc_{X,\bullet}}&CH_\bullet(X)\ar[d]^{f_*}\\
K(Y,\Lambda)\ar[r]^{cc_{Y,\bullet}}&CH_\bullet(Y)
}
\end{gathered}
\end{align}
commutative for any proper map $f\colon X\rightarrow Y$ between smooth schemes over $k$? 
By the philosophy of Grothendieck \cite[Note $87_1$]{Grothendieck} and \cite[Example 6.10]{Saito}, the answer is no in general if ${\rm char}(k)>0$.  
However, Corollary \ref{corYZ:1} says that the degree 0-part of the diagram (\ref{eq:fccf}) is commutative, i.e., 
if $f\colon X\rightarrow Y$ is a proper map between smooth projective schemes over a finite field $k$, then we have  the following commutative diagram
\begin{align}\label{eq:fccf2}
\begin{gathered}
\xymatrix{
K(X,\Lambda)\ar[d]_{f_\ast}\ar[r]^{cc_{X}}&CH_0(X)\ar[d]^{f_*}\\
K(Y,\Lambda)\ar[r]^{cc_{Y}}&CH_0(Y).
}
\end{gathered}
\end{align}
If $k=\mathbb C$, the diagram (\ref{eq:fccf}) is commutative by \cite[Theorem A.6]{Ginsburg}.

\subsection{}
Just after writing this paper, T.~Saito  \cite{Saito17b}
announced a  proof for proper push-forward of  characteristic cycles along projective 
morphisms $f\colon X\rightarrow Y$ between smooth projective schemes over a perfect field, 
under the assumption that the direct image of the singular support of $\mathcal F$ has  dimension at most $\dim Y$.
Under the same assumption, T.~Saito's result implies (\ref{eqYZ:corep}). 
In Corollary \ref{corYZ:1}, we don't need any assumption on the dimension of the direct image of the singular support of $\mathcal F$, 
but we have to assume that $k$ is a finite field.

\subsection{}
As another application of Theorem \ref{main}, we show that the two Swan classes in the literature are the same on proper smooth surfaces over a finite field.
Let $X$ be a proper smooth and connected scheme  over a perfect field $k$. Let $j\colon U\rightarrow X$ be an open dense sub-scheme of $X$ and $\mathcal F$ a smooth sheaf of $\Lambda$-modules on $U$. 
In \cite{Kato_Saito}, Kato and T. Saito defined the Swan class $\Sw_X^{\rm ks}(\mathcal F)\in {\rm CH}_0(X-U)\otimes_{\mathbb Z}\mathbb Q$ by using logarithmic product and alteration. Later  in \cite[Definition 6.7.3]{Saito}, T.~Saito defined another Swan class $\Sw_X^{\rm cc}(\mathcal F)$ by using characteristic cycle, cf. (\ref{eq:sc:cc}).
In  \cite[Conjecture 6.8.2]{Saito}, T.~Saito conjectures that the two Swan classes are the same in ${\rm CH}_0(X-U)$. In this paper, we prove the following weak version of his conjecture for smooth proper surfaces over finite fields:
\begin{corollary}[see Corollary \ref{cor:2surf}]\label{cor:intr2}
Let $X$ be a proper smooth and connected surface  over a finite field $k$. 
Let $U$ be an open dense sub-scheme of $X$ and $\mathcal F$ a smooth sheaf of $\Lambda$-modules on $U$. 
Then
we have an equality in $CH_0(X)$:
\begin{equation}\label{eq:sc:2surf}
 \Sw_X^{\rm ks}(\mathcal F)=- \Sw_X^{\rm cc}(\mathcal F).
\end{equation}
\end{corollary}
The proof of Corollary \ref{cor:intr2} also works for higher dimensional smooth schemes over any perfect field if we assume resolution of singularities and a special case of proper push-forward of characteristic class. For the details, see Theorem \ref{thm:sc:surf}.

\subsection{}
Here is a rough idea of the proof of Theorem \ref{main}.
We follow a similar strategy as S. Saito's proof in \cite{Saito:1984}.
By taking a good pencil, we prove Theorem \ref{main} by induction on the dimension of $X$.
By using the product formula of Deligne and Laumon and Proposition  \ref{prop:cc2}, it is reduced to computing local contributions on both sides. In order to calculate local contributions, we use the Milnor formula (\ref{eqYZ:milnor}) for characteristic cycle.

For the proof of Corollary \ref{cor:intr2}, by using Brauer induction theorem and Corollary \ref{corYZ:1},
we first reduce the proof to the case where $\mathcal F$ is of rank 1 and $\mathcal F$ is trivialized by a cyclic covering of Galois group $G$.
The essential case is  when $G$ is a cyclic group of order $p$. In that case, for two characters $\chi$ and $\chi^\prime$ of $G$, if they have the same order, then they have the same characteristic cycle by a result of T. Saito and Yatagawa, cf. \cite[Theorem 0.1]{Saito_Yatagawa}. Using their theorem, we can reduce the proof to a very special case where $\mathcal F=f_\ast\Lambda$ and $f\colon V\rightarrow U$ is an \'etale covering. At last, a direct calculation proves Corollary \ref{cor:intr2}.

\subsection{}
The content of each section is as follows. In Section \ref{sec:lc}, we review the theory of local constants and the product formula for epsilon factors. 
In Section \ref{sec:poc}, we prove a blow up formula and a fibration formula for the Gysin pull-back of cycles  by the zero section  of the cotangent bundle. The proof of Proposition \ref{prop:cc2} is based on that section. 
In Section \ref{sec:eogp}, we start with a review of Saito and Yatagawa's result on the existence of pre-good pencils. Then we prove the existence of good pencils in Lemma \ref{pencil}.
In  Section \ref{sec:cc}, we first recall the definitions and properties of singular support and characteristic cycle. Then we prove Theorem \ref{main} and Corollary \ref{corYZ:1}.
In  Section \ref{sec:sc}, we prove Corollary \ref{cor:intr2}.
\subsection*{Acknowledgements}
The authors would like to express their gratitude to Takeshi Saito for inspiring discussion and carefully reading our manuscript.
The authors also thank the referee for a number of helpful suggestions and corrections.
The second author thanks Ahmed Abbes for pointing out Vidal's work \cite{vi1, vi2} to him during a visit at IHES.
The second author is partially supported by Alexander von Humboldt Foundation for his research at Universit\"at Regensburg and Freie Universit\"at Berlin. 
Both the second and the third authors are partially supported by the DFG through CRC 1085 \emph{Higher Invariants} (Universit\"at Regensburg).
The authors are grateful to these institutions.

\subsection*{Notation and Conventions}

\begin{enumerate}
\item Let $p$ be a prime number and $\Lambda$ be a finite field of characteristic $\ell\neq p$ or $\Lambda=\overline{\mathbb Q}_\ell$.
\item We say that a complex $\mathcal F$ of \'etale sheaves of $\Lambda$-modules on a scheme $X$ is {\it constructible} (respectively {\it smooth}) if the cohomology sheaf $\mathcal H^q(\mathcal F)$ is constructible  for every $q$ and if $\mathcal H^q(\mathcal F)=0$ except finitely many $q$ (respectively moreover $\mathcal H^q(\mathcal F)$ is  locally constant for all $q$).
\item For a scheme $S$ over $\mathbb Z[1/{\ell}]$,
let $D^b_c(S,\Lambda)$ be the triangulated category of bounded complexes of $\Lambda$-modules with constructible cohomology groups on $S$
and let $K(S,\Lambda)$ be the Grothendieck group of $D^b_c(S,\Lambda)$. We call elements of $K(S,\Lambda)$ virtual $\Lambda$-sheaves on $S$.
\item For a pro-finite group $G$, we denote by $K(G,\Lambda)$  the Grothendieck group of the category of continuous representations of $G$ over $\Lambda$. Elements of $K(G,\Lambda)$ will be called virtual $\Lambda$-representations of $G$.
\item For a scheme $X$, we denote by  $\abs{X}$  the set of closed points of  $X$.
\item For a finite extension $E$ of $\mathbb Q_\ell$, we denote by $\mathcal O_E$ the ring of integers of $E$, by $\mathfrak m_E$ the maximal ideal of $\mathcal O_E$ and $\pi_E$ a uniformizer of $\mathcal O_E$. The residue field of $\mathcal O_E$ will be denoted by $k_E$. 
\item For any smooth scheme $X$  over a field $k$, we denote by $T^\ast_XX\subseteq T^\ast X$ the zero section of the cotangent bundle $T^\ast X$ of $X$.
\end{enumerate}

\section{Local  $\varepsilon$-factor and product formula}\label{sec:lc}
\subsection{}In this section, we review the theory of local  constants and the product formula for epsilon factors. In the case where $\Lambda=\bQl$, we use \cite[Th\'eor\` eme 3.1.5.4, Th\'eor\`eme 3.2.1.1]{Laumon}. In the case where $\Lambda$ is a finite field, we use \cite[Th\'eor\`eme 6.5, Th\'eor\`eme 7.11]{Deligne}.
\subsection{}\label{sub:defineTriple}
We first recall the results for local epsilon factors with $\bQl$-coefficients.  
Let $k$ be a finite field of characteristic $p$. A {\it triple $(T,\mathcal F,\omega)$ over $k$} means the following data:
\begin{itemize}
\item $T$ is a henselian trait of equicharacteristic $p$ with generic point $\eta$ and closed point $s$, such that the residue field $k(s)$ of $s$ is a finite extension of $k$.
\item $\mathcal F\in D_c^b(X,\bQl)$.
\item A non-zero meromorphic 1-form $\omega\in \Omega_{k(\eta)}^1\setminus\{0\}$ on $T$. 
\end{itemize}
For such a triple $(T,\mathcal F,\omega)$, we denote by $K_s$ the completion of the fraction field of $T$ with respect to the valuation ${\rm ord}_s\colon k(\eta)^\times \to \mathbb Z$. The valuation ${\rm ord}_s$ defines a map ${\rm ord}_s\colon \Omega_{K_s}^1\setminus\{0\}\to\mathbb Z$ which is characterized by the formula $\ord_s(a\cdot db)=\ord_s(a)$ for any $a\in K_s^\times$ and  any uniformaizer $b$ of $K_s$. The reciprocity map
$\rho_s\colon K_s^\times \to {\rm Gal}(\overline K_s/K_s)^{\rm ab}$ is normalized by sending  uniformizers of $K$ to liftings of the geometric Frobenius of $k(s)$. By local abelian class field theory (cf. \cite{Tate2}),
$\rho_s$ is injective with dense image.
\begin{theorem}[{\cite[Th\'eor\` eme 3.1.5.4]{Laumon}}]\label{thm:laumonepsilon}
Let $k$ be a finite field of characteristic $p$ and $\ell\neq p$ a prime number.
Let $\psi \colon \mathbb F_p\to {\overline{\mathbb Q}}_\ell^\times$ be a fixed non-trivial additive character of $\mathbb F_p$. There exists a unique map $\varepsilon\coloneqq\varepsilon_{\psi}$ sending any triple $(T, \F, \omega)$ over $k$ 
 to $\varepsilon(T,\F,\omega)\in \overline{\mathbb{Q}}_{\ell}^{\times}$, which satisfies the following conditions:
 \begin{itemize}
 \item[(i)] the number $\varepsilon(T,\F,\omega)$ depends only on the isomorphism class of the triple $(T,\F,\omega)$.
 \item[(ii)] For any distinguished triangle in $D^b_c(T,\overline{\mathbb Q}_{\ell})$:
 \[ \F'\to \F\to \F'' \to \F'[1] \]
 we have   
 \[ \varepsilon(T,\F,\omega)=\varepsilon(T,\F',\omega)\cdot\varepsilon(T,\F'',\omega) \]
 \item[(iii)] If $\F$ is supported on the closed point $s$ of $T$ (i.e., $\F_{\bar{\eta}}=0$), then 
 \[ \varepsilon(T,\F,\omega)=\det(-\mathrm{Frob}_s;\F_{\bar s})^{-1}, \]
 where $\mathrm{Frob}_s\in{\rm Gal}(\overline{k(s)}/k(s))$ is the geometric Frobenius at $s$. In the case where $\mathcal F_{\bar s}=0$, we put $\det(-\mathrm{Frob}_s;\F_{\bar s})^{-1}=1$.
 \item[(iv)] Let $\eta_1$ be a finite separable extension of $\eta$,  $f\colon T_1\to T$  the normalization of $T$ in $\eta_1$, and $\F_1\in D^b_c(T_1,\overline{\mathbb Q}_{\ell})$. If $\mathrm{rank}((\F_{1})_{\bar{\eta}})=0$, then \[ \varepsilon(T,f_*\F_1,\omega)=\varepsilon(T_1,\F_1,f^*\omega).\]
 \item[(v)] Let $\G$ be a smooth sheaf of rank $1$ on $\eta$, which induces a character $\chi\colon K^{\times}_{s}\to \overline{\mathbb Q}_{\ell}^{\times}$ via the reciprocity homomorphism $\rho_s\colon K^{\times}_{s} \to \mathrm{Gal}(\overline{K}_s/K_s)^{\mathrm{ab}}$. 
 Let $j\colon \eta\to T$ be the immersion.
 Then we have 
 \[ \varepsilon(T,j_*\G,\omega)=\varepsilon(\chi,\Psi_\omega), \]
 where $\Psi_\omega\colon K_s\to \bQl^\times$ is the character defined by $\Psi_\omega(a)=\psi\circ {\rm Tr}_{k(s)/\mathbb F_p}({\rm Res}(a\cdot \omega))$
 and $\varepsilon(\chi,\Psi_\omega)$ is the Tate local constant \cite{Tate} (cf. \cite[3.1.3.2]{Laumon}).
 \end{itemize}
\end{theorem}
\subsection{}The formula \eqref{eqYZ:ep} can be viewed as a globalization of the following result \eqref{eq:localtwist}.
\begin{lemma}[Local twist formula {\cite[3.1.5.6]{Laumon}}]
Let $(T,\mathcal F,\omega)$ be a triple (cf. \ref{sub:defineTriple}) over  $k$.
Let $\G$ be a smooth $\overline{\mathbb Q}_{\ell}$-sheaf on $T$. Then we have 
\begin{equation}\label{eq:localtwist}
\varepsilon(T,\F\otimes\G, \omega)=\varepsilon(T,\F,\omega)^{\mathrm{rank}\G}\cdot \det(\mathrm{Frob}_s;\G_{\bar s})^{a(T,\F,\omega)},
\end{equation}
where  $a(T,\F,\omega)$ is the local Artin conductor given by the formula
\begin{align}\label{eq:localartintwisted}
 a(T,\F,\omega)\coloneqq \mathrm{rank} \F_{\bar{\eta}}-\mathrm{rank} \F_{\bar{s}}+\mathrm{Sw}\F_{\bar{\eta}}+\mathrm{rank}\F_{\bar{\eta}}\cdot \mathrm{ord}_s(\omega),
\end{align}
and $\mathrm{Sw}\F_{\bar{\eta}}$ is the Swan conductor of $\F_{\bar{\eta}}$ $($cf. \cite[19.3]{Serre97}$)$.
\end{lemma}

\subsection{}Let $T$ be a henselian trait of equicharacteristic $p$ with generic point $\eta$ and closed point $s$, such that the residue field $k(s)$ of $s$ is a finite extension of the fixed  finite field  $k$. Let $V$ be a $\bQl$-sheaf on $\eta$ and $j\colon \eta\to T$ be the immersion. We define 
\begin{equation}\label{eq:e0forlaumon}
\varepsilon_0(T,V,\omega)\coloneqq \varepsilon(T,j_!V,\omega).
\end{equation}
Let $V^\prime\to V\to V^{\prime\prime}\to V^\prime[1]$ be a distinguished triangle in $D_c^b(\eta,\bQl)$.
Since $j_!$ is an exact functor, by Theorem \ref{thm:laumonepsilon} (ii), we have
\begin{equation}\label{eq:e00forlaumon}
\varepsilon_0(T,V,\omega)=\varepsilon_0(T,V^\prime,\omega)\cdot \varepsilon_0(T,V^{\prime\prime},\omega).
\end{equation}
For any $\mathcal F\in D_c^b(T,\bQl)$, we have a distinguished triangle 
\begin{align}
j_!j^\ast\mathcal F\to \mathcal F\to i_\ast i^\ast\mathcal F\to j_!j^\ast\mathcal F[1],
\end{align}
where $i\colon s\to T$ is the closed immersion.
By Theorem \ref{thm:laumonepsilon} (ii), we have 
\begin{align*}\varepsilon(T,\mathcal F,\omega)=\varepsilon(T,j_!j^\ast\mathcal F,\omega)\cdot \varepsilon(T,i_\ast i^\ast \mathcal F).
\end{align*} 
By \eqref{eq:e0forlaumon} and Theorem \ref{thm:laumonepsilon} (iii), we obtain the following equality
\begin{align}\label{eq:e01forlaumon}
\varepsilon(T,\mathcal F,\omega)=\det(-{\rm Frob_s};\mathcal F_{\bar s})^{-1}\cdot \varepsilon_0(T,\mathcal F|_{\eta},\omega).
\end{align}
If $W$ is the restriction to $\eta$ of an smooth $\bQl$-sheaf on $T$, then by \eqref{eq:localtwist} and \eqref{eq:e0forlaumon} we have
\begin{equation}\label{eq:e00forlaumontwist}
\varepsilon_0(T,V\otimes W,\omega)=\varepsilon_0(T,V,\omega)^{{\rank}W}\cdot\det({\rm Frob}_s;W)^{{\rm rank}V+\Sw V+{\rm rank}(V)\cdot {\rm ord}_s(\omega)}
\end{equation}
for any $\bQl$-sheaf $V$ on $\eta$.

\subsection{}\label{sub:localconsforfinite}
Now, we turn to recall the results for local constants with finite coefficients. Let $\Lambda $ be a finite field of characteristic $\ell$.
Let $K$ be a complete discrete valuation field with finite residue field $k$ of characteristic $p(p\neq \ell)$.
Let $\mathcal O_K$ be the  ring of integers of $K$ and  $S=\mathrm{Spec} {\mathcal O_K}$. Let $\overline{K}$ be a separable closure
of $K$ and $W_K=W(\overline{K}/K)$ the Weil group of $\overline{K}$ over $K$ (cf.   \cite[2.2.4]{Deligne} and \cite{Tate2}).
Let $s\in S$ be the closed point and  $\eta\in S$  the generic point.

We fix a $\Lambda$-valued non-trivial additive character $\psi$ of $K$, and the $\Lambda$-valued Haar measure $dx$ on $K$ such that $\int_{\mathcal O_K}dx=1$.
Let $n(\psi)$ be the maximal integer $n$ such that $\psi$ is trivial on $\pi_K^{-n}\mathcal{O}_K$, where $\pi_K$ is a uniformizer of $K$ (cf. \cite[3.4]{Deligne}).

We define the \emph{total dimension} $\dimtot(V)$ of a representation $V$ of $W_K$ as $\rank(V)+\Sw(V)$, where $\Sw(V)$ is the Swan conductor of $V$ (cf. \cite[19.3]{Serre97}).
For $\F\in D^b_c(S,\Lambda)$, $\dimtot(\F)$ is defined to be 
\[\dimtot(\F_{\bar{\eta}})=\sum_{i}(-1)^i\cdot \dimtot (\mathcal H^i(\mathcal F_{\bar\eta})).\]

We denote by $K(S,\Lambda)$ the Grothendieck group of $D_c^b(S,\Lambda)$, and elements of $K(S,\Lambda)$ will be called {\it virtual $\Lambda$-sheaves}.
We denote by $K(W_K,\Lambda)$ the Grothendieck group of the category of continuous $\Lambda$-representations of $W_K$. Elements of  $K(W_K,\Lambda)$ will be called {\it virtual $\Lambda$-representations}.

\subsection{}
Deligne and Langlands attached to every  continuous $\Lambda$-representation $V$ of $W_K$ a local constant $\varepsilon_0(V, \psi, dx)\in \Lambda^\times$, which is uniquely characterized by the  conditions of \cite[Th\'eor\`eme 6.5]{Deligne}.
For $\F\in D_c^b(S,\Lambda)$, we define (cf. \cite[(7.6.3)]{Deligne})
\begin{equation}\label{eq:deligneepsilondef}
\varepsilon(\F,\psi, dx):=\det(-\mathrm{Frob}_s;\F_{\bar{s}})^{-1}\cdot \varepsilon_0(\F_{\bar{\eta}},\psi, dx)  \qquad{\rm in}~\Lambda^\times.
\end{equation}
In \eqref{eq:deligneepsilondef}, when $\mathcal F_{\bar s}=0$, we put $\det(-{\rm Frob}_s;\mathcal F_{\bar s})^{-1}=1$.
The local constant $\varepsilon_0(V, \psi, dx)$ is additive with respect to $V$ by  \cite[Th\'eor\`eme 6.5]{Deligne}.
By \eqref{eq:deligneepsilondef}, the $\varepsilon(\F,\psi, dx)$ is also additive with respect to $\mathcal F$.
Therefore, they induce group homomorphisms $K(W_K,\Lambda)\xrightarrow{\varepsilon_0}\Lambda^\times$ and $K(S,\Lambda)\xrightarrow{\varepsilon}\Lambda^\times$ respectively.
\subsection{}\label{sub:notationsforepsilon}
Let $\Lambda$ be a finite field of characteristic $\ell$ or $\Lambda=\bQl$.
In order to give a uniform proof of Theorem \ref{main} in both cases, we introduce a uniform notation in \ref{sub:uniform}. In the rest part of this section, 
let $X$ be a smooth projective geometrically connected curve over a finite field $k$ of characteristic $p(p\neq \ell)$,  $K$  the function field of $X$ and  $\psi_0\colon \mathbb F_p\to {\Lambda}^\times$  a non-trivial additive character.
Let $\omega$ be 
a non-zero rational $1$-form  on $X$. It is well-known that $\omega$ defines a non-trivial additive character $\psi_\omega$ on $\mathbb A_K/K$.
Its local component at $v\in |X|$ is the ${\Lambda}^\times$-valued function defined by $(\psi_\omega)_v(a)=\psi_0({\rm Tr}_{k(v)/\mathbb F_p}(\mathrm{res}_v(a\cdot\omega)))$ for any $a\in K_v$.
The number $n((\psi_\omega)_v)$ (cf. \ref{sub:localconsforfinite}) equals to the order $\ord_v(\omega)$ of the differential form.
Let $dx=\otimes_{v\in\abs{X}}dx_v$ be the Haar measure on $\mathbb A_K$ with values in 
$\Lambda$ such that $\int_{O_{K_v}}dx_v=1$ for every place $v$.

\subsection{}\label{sub:uniform}
Let $v\in X$ be a closed point and  $X_{(v)}$ the henselization of $X$ at  $v$. 
Let $\eta_v$ be the generic point of $X_{(v)}$.
We define two group homomorphisms
\begin{align}\label{eq:defep3}
\varepsilon_v(\bullet, \omega)&\colon K(X_{(v)},\Lambda)\to \Lambda^\times\\
\label{eq:defep3-1}\varepsilon_{0v}(\bullet, \omega)&\colon K(\eta_{v},\Lambda)\to \Lambda^\times
\end{align}
as follows: let $\mathcal F\in K(X_{(v)},\Lambda)$ and $V\in K(\eta_v,\Lambda)$. 
If $\Lambda=\bQl$,  then we put
\begin{align}\label{eq:defep0}
\varepsilon_v(\mathcal F, \omega)\coloneqq \varepsilon(X_{(v)},\mathcal F, \omega|_{X_{(v)}})\qquad{\rm and}\qquad
\varepsilon_{0v}(V, \omega)\coloneqq\varepsilon_0(X_{(v)},V, \omega|_{X_{(v)}})
\end{align}
by using the local constant defined in Theorem \ref{thm:laumonepsilon} and \eqref{eq:e0forlaumon}.
If $\Lambda$ is a finite field, 
we define 
\begin{align}\label{eq:defep1}
\varepsilon_v(\mathcal F, \omega)\coloneqq\varepsilon(\F_{(v)},(\psi_\omega)_v, dx_v)\qquad{\rm and}\qquad
\varepsilon_{0v}(V, \omega)\coloneqq\varepsilon_0(V,(\psi_\omega)_v, dx_v).
\end{align}
by using \eqref{eq:deligneepsilondef}. 

For any $\mathcal F\in K(X_{(v)},\Lambda)$, we have (cf. \eqref{eq:e01forlaumon} and \eqref{eq:deligneepsilondef})
\begin{align}\label{eq:eq:defep2}
\varepsilon_v(\mathcal F,\omega)=\det(-{\rm Frob}_v;\mathcal F_{\bar v})^{-1}\cdot \varepsilon_{0v}(\mathcal F|_{\eta_v},\omega).
\end{align}
In \eqref{eq:eq:defep2}, when $\mathcal F_{\bar v}=0$, we put $\det(-{\rm Frob}_v;\mathcal F_{\bar v})^{-1}=1$.
\subsection{}In the following lemma, we collect some  formulas which are true  in both cases. These formulas  will be very useful in the proof of Theorem \ref{main}.
\begin{lemma}[{\cite{Deligne,Laumon}}]\label{twist} 
Under the assumptions in \ref{sub:notationsforepsilon}. Let $v\in X$ be a closed point, $\eta_v$  the generic point of $X_{(v)}$ and $\mathcal F\in K(X_{(v)},\Lambda)$. 
Then we have
\begin{enumerate}
\item If $\mathcal F_{\bar\eta_v}=0$, then 
\begin{align}\label{eq:twist-NEW-1}
\varepsilon_v(\mathcal F,\omega)=\det(-{\rm Frob}_v;\mathcal F_{\bar v})^{-1}.
\end{align}
\item If $\mathcal G$ is a smooth $\Lambda $-sheaf on $X_{(v)}$, then we have 
\begin{align}\label{eqYZ:t1}
\varepsilon_v(\mathcal F\otimes\mathcal G,\omega)&=\varepsilon_v(\F,\omega)^{\mathrm{rank}\G}\cdot \det(\mathrm{Frob}_v;\G_{\bar v})^{ \mathrm{dimtot} \F_{\bar{\eta}_v}-\mathrm{rank} \F_{\bar{v}}+\mathrm{rank}\F_{\bar{\eta}_v}\cdot \mathrm{ord}_v(\omega)}\\
\nonumber\varepsilon_{0v}((\mathcal F\otimes\mathcal G)|_{\eta_v},\omega)&=\varepsilon_{0v}(\F|_{\eta_v},\omega)^{\mathrm{rank}\G}\cdot \det(\mathrm{Frob}_v;\G_{\bar v})^{ \mathrm{dimtot} \F_{\bar{\eta}_v}+\mathrm{rank}\F_{\bar{\eta}_v}\cdot \mathrm{ord}_v(\omega)}
\end{align}
\item If $\F$ is smooth of rank $0$,
then we have
\begin{equation}\label{eqYZ:t2}
\varepsilon_v(\F,\omega)=\det(\mathrm{Frob}_v;\F_{\bar{v}})^{\ord_v(\omega)}.
\end{equation}
\item 
If $\F_{\bar{v}}$ is of rank $0$ and
if ${\rm ord}_v(\omega)=0$,
then we have
\begin{equation}\label{eqYZ:t3}
\varepsilon_v(\F,\omega)=\varepsilon_{0v}(\F_{\bar{v}},\omega)^{-1}\cdot\varepsilon_{0v}(\F|_{{\eta_v}},\omega).
\end{equation}
Here, we regard $\F_{\bar{v}}$ as the generic fiber of the smooth virtual $\Lambda$-sheaf on $X_{(v)}$ defined by the action of $\pi_1(X_{(v)},\bar{v})$ on the stalk $\F_{\bar{v}}$.
\end{enumerate}
\end{lemma}
\begin{proof}
(1) If $\Lambda=\bQl$, \eqref{eq:twist-NEW-1} follows from Theorem \ref{thm:laumonepsilon} (iii).
If $\Lambda$ is a finite field, \eqref{eq:twist-NEW-1} follows from \eqref{eq:deligneepsilondef}.

(2) If $\Lambda=\bQl$, the assertion follows from \eqref{eq:e00forlaumontwist} and  \eqref{eq:localtwist}. If $\Lambda$ is a finite field,
by \cite[Th\'eor\`eme 6.5]{Deligne} and the content above of \cite[Remark 6.6, p.557]{Deligne}, the constant $\varepsilon_0$ also satisfies the formula \cite[(5.5.3)]{Deligne}.

(3) By \eqref{eqYZ:t1}, we have
$\varepsilon_v(\mathcal F,\omega)=\varepsilon_v(\Lambda,\omega)^{\mathrm{rank}\F}\cdot \det(\mathrm{Frob}_v;\F_{\bar v})^{\mathrm{ord}_v(\omega)}=\det(\mathrm{Frob}_v; \F_{\bar{v}})^{\mathrm{ord}_v(\omega)}$.


(4)By the  formula   (\ref{eqYZ:t2}), we have
\[
\varepsilon_v(\F_{\bar{v}}, \omega)=1.
\]
Taking the ratio of 
\[
\varepsilon_v(\F,\omega)=\det(-\mathrm{Frob}_v;\F_{\bar{v}})^{-1}\cdot\varepsilon_{0v}(\F|_{{\eta}_v},\omega)
\]
and
\[
\varepsilon_v(\F_{\bar{v}},\omega)=\det(-\mathrm{Frob}_v;\F_{\bar{v}})^{-1}\cdot\varepsilon_{0v}(\F_{\bar{v}},\omega),
\]
we finish the proof.
\end{proof}

\subsection{}Under the assumptions in \ref{sub:notationsforepsilon},
we recall Deligne and Laumon's product formula, which gives a relation between global $\varepsilon$-factors and local $\varepsilon$-factors.
We define a group homomorphism $K(X,\Lambda)\to\Lambda^\times$ by
\begin{equation}
\F\mapsto\varepsilon(X,\F)\coloneqq\det(-\mathrm{Frob}_k;R\Gamma(X_{\bar{k}},\F))^{-1}.
\end{equation}
We have the following product formula, which is conjectured by Deligne  \cite{De2} and proved by Laumon in the outstanding paper \cite[3.2.1.1]{Laumon}.

\begin{theorem}[{\cite[3.2.1.1]{Laumon}}]\label{product}
Let $k$ be a finite field of characteristic $p$ with $q$ elements. Let  $\Lambda$ be a finite field of characteristic $\ell(\ell\neq p)$ or $\Lambda=\bQl$.
Let $X$ be a smooth projective geometrically connected curve over $k$ of genus $g$, let $\omega$ be a non-zero rational $1$-form on $X$ and $\F\in K(X,\Lambda)$.
Then we have
\begin{equation}\label{eqYZ:product}
\varepsilon(X,\F)=q^{(1-g)\rank (\F)}\prod_{v\in\abs{X}}\varepsilon_v(\F|_{X_{(v)}},\omega) \qquad{\rm in}~\Lambda^\times.
\end{equation}
\end{theorem}

\begin{remark}~

\begin{enumerate}
\item Deligne  proves \eqref{eqYZ:product} under the  condition: 
$\mathcal F$ is a smooth $\ell$-adic sheaf on an open dense subscheme $U$ of $X$ and has finite geometric monodromy, i.e., there exists a finite \'etale covering $U^\prime$ of $U$ such that $\mathcal F|_{U^\prime}$ is geometrically constant $($cf. \cite[Expos\'e III]{De2}, \cite[Remarques (3.2.1.8)]{Laumon} and \cite[p.120]{Katz87}$)$. 

When $\mathcal F$ has  finite geometric monodromy, one uses Brauer induction theorem to reduce to the abelian case, where the existence of local constants and the product formula is classical \cite{Tate}. Deligne's argument also works for mod $\ell$ representations and gives a mod $\ell$ product formula \cite[Th\'eor\`eme 7.11]{Deligne}. From this Deligne deduces that the product formula holds for any $\ell$-adic sheaf $\mathcal F$ if moreover $\mathcal F$ is a part of a compatible system of $\ell$-adic representations for an infinite set of primes $\ell$ \cite[Th\'eor\`eme 9.3]{Deligne}.

\item If $\Lambda$ is a finite field, 
the product formula \eqref{eqYZ:product} can be obtained by using Brauer induction theorem $($cf. the proofs of \cite[Proposition 3.2.1.7]{Laumon} and \cite[Th\'eor\`eme 7.11]{Deligne}$)$.

\item The factor $q^{(1-g){\rm rank}\mathcal F}$ will disappear if we choose a Haar measure $dx$ on $\mathbb A_K$ such that $\mathbb A_K/K$ is of volume 1 $($cf. \cite[Expos\'e IV, 2.1]{De2}$)$.
\end{enumerate}

\end{remark}

\section{Preliminaries on cycles}\label{sec:poc}
\subsection{}
In this section, we prove a blow up formula (cf. Lemma \ref{blowup}) and a fibration formula (cf. Lemma \ref{cc}) for the Gysin pull-back of cycles by the zero section of a cotangent bundle. These two formulas will be used in the proof of Proposition \ref{prop:cc2} and Theorem \ref{main}.

\begin{definition}Let $X$, $Y$ and $W$ be  smooth schemes  over a field $k$. 
We denote by $T^\ast_XX\subseteq T^\ast X$ the zero section of the cotangent bundle $T^\ast X$ of $X$.
Let $C$ be a conical closed subset of $T^*X$, i.e., a closed subset which is stable under the action of the multiplicative group $\mathbb{G}_m$. 

\begin{enumerate}
\item $($\cite[1.2]{Beilinson}$)$ Let $h\colon W\to X$ be a morphism over $k$. 
We say that $h$  is $C$-transversal at $w\in W$ if the fiber 
$\left((C\times_X W)\cap dh^{-1}(T^\ast_WW)\right)\times_W w$ is contained in the zero-section $ T^\ast_XX\times_X W\subseteq T^\ast X\times_X W$,
where $dh\colon T^*X\times_XW\to T^*W$ is the canonical map.
We say that $h$ is $C$-transversal if $h$ is $C$-transversal at any point of $W$.

If $h$ is $C$-transversal,
we define $h^\circ C$ to be the image of $C\times_XW$ under the map $dh\colon T^*X\times_XW\to T^*W$.
By  \cite[Lemma 3.1]{Saito}, $h^\circ C$ is a closed conical subset of $T^\ast W$.
\item $($\cite[Definition 7.1]{Saito}$)$ Assume that $X$ and $C$ are purely of dimension $d$ and 
that $W$ is purely of dimension $m$.
We say that a $C$-transversal map $h\colon W\to X$ is properly $C$-transversal if every irreducible component of $C\times_XW$ is of dimension $m$. 
\item $($\cite[1.2]{Beilinson} and \cite[Definition 5.3]{Saito}$)$ We say that a morphism $f\colon X\to Y$ over $k$ is $C$-transversal at $x\in X$ if the inverse image $df^{-1}(C)\times_Xx$ is contained in the zero-section $T^\ast _YY\times_YX\subseteq T^\ast Y\times_Y X$, where $df\colon T^\ast Y\times_Y X\rightarrow T^\ast X$ is the canonical map.
We say that $f$ is $C$-transversal if $f$ is $C$-transversal at any point of $X$.

We say that a closed point $x\in X$ is at most an isolated characteristic point of $f$ with respect to $C$ if there is an open neighborhood $U$ of $x$ such that the restriction  $f|_{ U-\{x\}}\colon U-\{x\}\rightarrow Y$ is $C\times_X(U-\{x\})$-transversal.
\end{enumerate}
\end{definition}
\subsection{}Let $X$ be a smooth scheme   purely of dimension $d$ over a field $k$. Let $W$  be a smooth scheme purely of dimension $m$ over  $k$.
Assume that  $C\subseteq T^*X$ is a conical closed subset purely of dimension $d$.
Let $Z$ be a $d$-cycle supported on $C$ and $h\colon W\to X$  a  $C$-transversal morphism. Let $\pr_h\colon T^\ast X\times_XW\to T^*X$ be the first projection map.
Since $\pr_h$ is a morphism between smooth schemes, the refined Gysin pull-back $\pr_h^! Z$ is well-defined in the sense of intersection theory \cite[6.6]{Fulton}.
We define  $h^\ast Z\in CH_m(h^\circ C)$ \cite[Definition 7.1.2]{Saito} to be 
\begin{equation}\label{eq:pullbackCCdef}
h^\ast Z:=dh_*(\pr_h^!Z).
\end{equation} 
Notice that the push-forward is well-defined since $dh\colon T^\ast X\times_XW\to T^\ast W$ is finite on $C\times_XW$ by \cite[Lemma 1.2 (ii)]{Beilinson}. 
We also put
\begin{equation}\label{eq:pullbackCCdef-1}
h^! Z:=(-1)^{m-d}\cdot h^\ast Z.
\end{equation} 

If moreover $h$ is properly $C$-transversal, then every irreducible component of $h^\circ C$ is of dimension $m$. Thus $CH_m(h^\circ C)=Z_m(h^\circ C)$. Hence we may regard $h^\ast Z$ as a $m$-cycle on $T^\ast W$, which is supported on $h^\circ C$.

\subsection{}We  use refined intersections in the sense of \cite[Definition 8.1.1]{Fulton}. Let us recall a special case. 
Let $f\colon X\to Y$ be a morphism of finite type to   a smooth scheme $Y$ purely of dimension $n$ over a field $k$. Let $p\colon Y^\prime\to Y$ be any morphism, $x\in CH_r(X)$ and $y\in CH_s(Y^\prime)$. Consider the following Cartesian diagrams
\begin{equation*}
\xymatrix{
X\times_Y Y^\prime\ar[r]\ar[d]\ar@{}|\Box[rd]&Y^\prime\ar[d]^-p&X\times_Y Y^\prime\ar[d]\ar[r]\ar@{}|\Box[rd]& X\times_k Y^\prime\ar[d]^-{{\rm id}\times p}\\
X\ar[r]^-f&Y&X\ar[r]^{\Gamma_f}&X\times_k Y
}
\end{equation*}
where $\Gamma_f$ is the graph of $f$. Then we define
\begin{align}\label{eq:def:intersectionproduct}
x{\cdot}_{f} y\coloneqq \Gamma_f^!(x\times y)\in CH_{r+s-n}(X\times_YY^\prime),
\end{align}
where $x\times y\in CH_{r+s}(X\times_k Y^\prime)$ is the exterior product (cf. \cite[1.10]{Fulton}) and $\Gamma_f^!$ is the refined Gysin homomorphism (cf. \cite[6.2]{Fulton}).
By \cite[Corolllary 8.1.3]{Fulton}, for any cycle $x$ on $X$, we have
\begin{align}\label{eq:trivialontarget}
x\cdot_f [Y]=x.
\end{align}
If $X$ is purely of dimension $d$, then the refined Gysin homomorphism 
\begin{align}\label{eq:refinedGysing00}
f^!\colon CH_k(Y^\prime)\to CH_{k+d-n}(X\times_YY^\prime)
\end{align}
is defined by the formula (cf. \cite[Definition 8.1.2]{Fulton})
\begin{align}\label{eq:refinedGysing01}
f^!(y)=[X]\cdot_fy.
\end{align}
When $f$ is flat, then $f^!=f^\ast$ is the flat pull-back \cite[Proposition 8.1.2]{Fulton}. When $f$ is a regular immersion, we have (cf. \cite[Corollary 8.1.1]{Fulton})
\begin{align}\label{eq:refinedGysing02}
f^!(y)=[X]\cdot_{{\rm id}_Y}y.
\end{align}

\subsection{}\label{sub:gysinzero}
Let $X$ be a smooth   scheme purely of dimension $d$  over a field $k$. We denote by   $0_X\colon X\to T^\ast X $ the zero section of the cotangent bundle $T^\ast X$. We denote by $0_X^!\in CH^d(X\to T^\ast X)$ the (refined) Gysin map \cite[6.2]{Fulton}, where $CH^d(X\to T^\ast X)$ is the bivariant Chow group \cite[Definition 17.1]{Fulton}). 

We prove the following blow up formula for the Gysin map $0_X^!$.
\begin{lemma}\label{blowup}
Let $X$ and $Y$ be smooth and connected schemes over a field $k$ and let $i\colon Y\hookrightarrow X$ be a closed immersion of codimension $c$.
Let $\pi\colon \widetilde{X}\to X$ be the blow up of $X$ along $Y$.
Let $C\subseteq T^*X$ be a conical closed subset purely  of dimension $d=\dim X$ and let $Z$ be a $d$-cycle supported on $C$.
Suppose $\pi$ and $i$ are properly $C$-transversal.
Then we have an equality in $CH_0(X)$:
\begin{align}\label{eq:blowupNEW}
\pi_*(0^!_{\widetilde X}(\pi^*Z))=0_X^!(Z)+(-1)^c\cdot (c-1)\cdot i_* (0_Y^!(i^*Z)),
\end{align}
where $\pi^\ast Z$ and $i^\ast Z$ are defined by \eqref{eq:pullbackCCdef}. For the definition of the Gysin maps $0^!_{\bullet}$, see Subsection \ref{sub:gysinzero}. 
\end{lemma}
\begin{proof}We only prove the case $c=2$ which is enough for the proof of main Theorem \ref{main}. 
A different proof for the general case is given in \cite[Lemma 3.3]{NYZ}.
Consider the following diagram
\begin{align}\label{eq:saito:10}
\begin{gathered}
\xymatrix{
&V\ar[d]\ar[r]^e\ar@{}|\Box[rd]&\widetilde X\ar[d]^-{0_{\widetilde X}}\\
T^\ast X&T^\ast X\times_X\widetilde X\ar[l]_-{{\rm pr}_\pi}\ar[r]^-{d\pi}&T^\ast\widetilde X.
}
\end{gathered}
\end{align}
where $d\pi$ is the map induced by $\pi$, $\pr_\pi$ is the first projection and $V=d\pi^{-1}(\widetilde X)$.
We have
\begin{align}
\label{eq:saito:20} 0^!_{\widetilde X}(\pi^*Z)&\overset{\eqref{eq:pullbackCCdef}}{=}0^!_{\widetilde X}d\pi_\ast{\rm pr}_\pi^!Z
\overset{\eqref{eq:refinedGysing02}}{=}d\pi_\ast{\rm pr}_\pi^!Z\cdot_{{\rm id}_{T^\ast\widetilde X}}\widetilde X\overset{(a)}{=}e_\ast({\rm pr}_\pi^!Z\cdot_{d\pi}\widetilde X)\\\nonumber&\overset{\eqref{eq:trivialontarget}}{=}e_\ast(({\rm pr}_\pi^!Z\cdot_{{\rm id}_{T^\ast X\times_X\widetilde X}}(T^\ast X\times_X\widetilde X))\cdot_{d\pi}\widetilde X)\\
\nonumber&\overset{(b)}{=}e_\ast({\rm pr}_\pi^!Z\cdot_{{\rm id}_{T^\ast X\times_X\widetilde X}}((T^\ast X\times_X\widetilde X)\cdot_{d\pi}\widetilde X))\\
\nonumber&\overset{\eqref{eq:refinedGysing01}}{=}e_\ast({\rm pr}_\pi^!Z\cdot_{{\rm id}_{T^\ast X\times_X\widetilde X}}d\pi^!(\widetilde X))
\end{align}
where (a) follows from the projection formula \cite[Example 8.1.7 ]{Fulton},
 (b) follows from the associativity property  \cite[Proposition 8.1.1(a)]{Fulton}.

Let $\widetilde{Y}$ be the exceptional divisor of $\pi\colon \widetilde{X}\to X$ with projection map $\tilde \pi\colon \widetilde{Y}\to Y$. Let $\tilde{i}\colon \widetilde{Y}\hookrightarrow \widetilde{X}$ be the closed immersion.
Since $\pi\colon \widetilde X\to X$ is a blow up, $\widetilde Y=\mathbb P((N_{Y/X})^\vee)$ is a projective space of relative dimension $c-1$ over $Y$ and we have exact sequences:
\begin{align}
\label{eq:saito:34}&0\to \Omega^1_{\widetilde Y/Y}\to \tilde\pi^\ast N_{Y/X}\otimes\mathcal O_{\widetilde Y}(-1)\to \mathcal O_{\widetilde Y}\to 0,\\
\label{eq:saito:35}&0\to \Omega^1_{\widetilde Y/Y}(1)\to \tilde\pi^\ast N_{Y/X}\to N_{\widetilde Y/\widetilde X}\to 0,\\
\label{eq:saito:36}&N_{\widetilde Y/\widetilde X}\simeq \mathcal O_{\widetilde Y}(1)\simeq \tilde i^\ast \mathcal O_{\widetilde X}(-\widetilde Y),\\
\label{eq:saito:37}&\Omega^1_{\widetilde X/X}\simeq \tilde i_\ast \Omega^1_{\widetilde Y/Y}.
\end{align}
For the proofs of \eqref{eq:saito:34}, \eqref{eq:saito:35} and \eqref{eq:saito:36}, we refer to \cite[B.6.3, B. 5.8]{Fulton}. Note that we use $N_{Y/X}$ to mean the conormal sheaf associated to the immersion $Y\to X$, which is different with \cite{Fulton}. In the following, we will also use $N_{Y/X}$ to mean the conormal bundle. Let $K$ be the kernel of $d\pi\colon T^\ast X\times_X\widetilde Y\to T^\ast \widetilde X\times_{\widetilde X}\widetilde Y$, i.e., $K=d\pi^{-1}(\widetilde Y)$. Then we have a commutative diagram with  exact rows and exact columns as vector bundles on $\widetilde Y$
\begin{align}\label{eq:saito:38}
\begin{gathered}
\xymatrix@C=1.5pc@R=1.5pc{
&&0\ar[d]&0\ar[d]\\
0\ar[r] &K\ar[r]^-{i_K}\ar@{=}[d]&N_{Y/X}\times_Y\widetilde Y\ar[d]\ar[r]&N_{\widetilde Y/\widetilde X}\ar[r]\ar[d]&0\\
0\ar[r]&K\ar[r]&T^\ast X\times_X\widetilde Y\ar[r]\ar[d]&T^\ast\widetilde X\times_{\widetilde X}\widetilde Y\ar[d]\ar[r]&T^\ast(\widetilde Y/Y)\ar[r]\ar@{=}[d]&0\\
&0\ar[r]&T^\ast Y\times_Y\widetilde Y\ar[r]^-{di}\ar[d]&T^\ast\widetilde Y\ar[r]\ar[d]& T^\ast(\widetilde Y/Y)\ar[r]&0\\
&&0&0
}
\end{gathered}
\end{align}
where $K$ is also equal to the kernel of $N_{Y/X}\times_Y\widetilde Y\to N_{\widetilde Y/\widetilde X}$ by the snake lemma.
Then we can write $V=d\pi^{-1}(\widetilde X)=\widetilde X\cup K$ (cf. \eqref{eq:saito:10}). By \eqref{eq:saito:35} and the first row of \eqref{eq:saito:38}, we have an isomorphism
\begin{align}
\label{eq:saito:40}K\simeq {\rm Spec}({\rm Sym}^\bullet_{\mathcal O_{\widetilde Y}}\Omega^\vee_{\widetilde Y/Y}(-1))
\end{align}
and $K$ is a vector bundle of rank $c-1$ over $\widetilde Y$. 

When $c=2$, we have $\dim K=\dim \widetilde X=\dim V=d$. Thus $d\pi^!(\widetilde X)\in CH_d(V)=Z_d(V)=Z_d(\widetilde X\cup K)=\mathbb Z\cdot \widetilde X+\mathbb Z\cdot K$.
Let $\Gamma\colon T^\ast X\times_X\widetilde X\to (T^\ast X\times_X\widetilde X)\times T^\ast\widetilde X$ be the graph of $d\pi$.
Since $T^\ast X\times_X\widetilde X$ and $ (T^\ast X\times_X\widetilde X)\times \widetilde X$ meet properly in $(T^\ast X\times_X\widetilde X)\times T^\ast\widetilde X$ at every irreducible component of $V=(T^\ast X\times_X\widetilde X)\cap ((T^\ast X\times_X\widetilde X)\times \widetilde X)$, we have (cf. \cite[Proposition 8.2(b)]{Fulton})
\begin{equation}\label{eq:saito:50}
d\pi^!(\widetilde X)=
[d\pi^{-1}(\widetilde X)]=[V]=a\cdot \widetilde X+b\cdot K\qquad a,b\in\mathbb Z.
\end{equation}
Since $V\setminus\widetilde Y$ is an open dense subscheme of $V$, we have $[V\setminus \widetilde Y]=a\cdot [\widetilde X\setminus \widetilde Y]+b\cdot [K\setminus \widetilde Y]$ (cf. \cite[Example 11.4.4 (i)]{Fulton}). Since $V\setminus \widetilde Y=(\widetilde X\setminus \widetilde Y)\coprod  (K\setminus \widetilde Y)$ is a disjoint union of smooth schemes, 
we get $a=b=1$.

Note that the cycle class $d\pi^!(\widetilde X)\cdot_{{\rm id}_{T^\ast X\times_X\widetilde X}}{\rm pr}_\pi^!Z$ is supported on $\widetilde X$, and $e_\ast$ is the identity on the support of this cycle class.
We have 
\begin{align}\label{eq:saito:70}
\pi_{\ast}0^!_{\widetilde X}(\pi^*Z)
&\overset{\eqref{eq:saito:20}}{=}\pi_\ast e_\ast({\rm pr}_\pi^!Z\cdot_{{\rm id}_{T^\ast X\times_X\widetilde X}}d\pi^!(\widetilde X))\\
\label{eq:saito:71} &\overset{\eqref{eq:saito:50}}{=}\pi_\ast ({\rm pr}_\pi^!Z\cdot_{{\rm id}_{T^\ast X\times_X\widetilde X}}\widetilde X)+\pi_\ast ({\rm pr}_\pi^!Z\cdot_{{\rm id}_{T^\ast X\times_X\widetilde X}}K)
\end{align}
Now, we  calculate the first term of \eqref{eq:saito:71}. Consider the following commutative diagram 
\begin{align}\label{eq:saito:80}
\begin{gathered}
\xymatrix{
X\ar[d]_-{0_X}\ar@{}|\Box[rd]&\widetilde{X}\ar[l]_-{\pi}\ar[d]^-{0_{X}^\prime}\\
T^\ast X&T^\ast X\times_X\widetilde X\ar[l]_-{~{\rm pr}_\pi~}
}
\end{gathered}
\end{align}
where $0^\prime_X\colon \widetilde X\to T^\ast X\times_X\widetilde X$ is the zero section. Since $0_X$ and $0^\prime_X$ are regular immersions of same codimension, by  the commutativity of refined Gysin homomorphisms \cite[Theorem 6.4, Theorem 6.2 (c)]{Fulton}, we get
\begin{equation}\label{eq:saito:81}
\pi_\ast 0_{X}^{\prime!}({\rm pr}_{\pi}^! Z)\simeq \pi_\ast\pi ^\ast 0_{ X}^!(Z)\simeq  0^!_X(Z),
\end{equation}
where  the last equality   follows from \cite[Proposition 6.7 (b)]{Fulton} ($\pi$ is a blow up).
Thus we have
\begin{align}\label{eq:saito:90}\pi_\ast ({\rm pr}_\pi^!Z\cdot_{{\rm id}_{T^\ast X\times_X\widetilde X}}\widetilde X)
\overset{\eqref{eq:refinedGysing02}}{=}\pi_\ast 0_{X}^{\prime!}({\rm pr}_{\pi}^! Z)\overset{\eqref{eq:saito:81}}{=}0^!_X(Z).
\end{align}

Now we  calculate the second term of \eqref{eq:saito:71}. Consider the following commutative diagram

\begin{align}\label{eq:saito:91}
\begin{gathered}
\xymatrix{
&&&T^\ast X\times_X\widetilde X\ar[r]^{{\rm pr}_{\pi}}&T^\ast X\\
\widetilde Y\ar@{=}[rrd]\ar@{^(->}[r]&K\ar[r]^-{i_K}\ar[rd]^-q&
N_{Y/X}\times_Y\widetilde Y\ar@{}|{\blacksquare}[rd]\ar[d]^-{p}\ar[r]^-{i_N}&T^\ast X\times_X\widetilde Y\ar@{}|\Box[rd]\ar[d]^-{di^\prime}\ar[r]^-{{\rm pr}^\prime_{\tilde \pi}}\ar[u]^-{{\rm pr}_i^\prime}\ar@{}|\Box[ur]& T^\ast X\times_XY\ar[d]^-{di}\ar[u]_-{{\rm pr}_i}\\
&&\tilde Y\ar[d]_-{\tilde \pi}\ar[r]^-{0^\prime_Y}&T^\ast Y\times_Y\widetilde Y\ar[r]^-{{\rm pr}_{\tilde \pi}}\ar[d]^-{\pr_{\tilde\pi}}&T^\ast Y\\
&&Y\ar[r]^-{0_Y}\ar@{}|\Box[ur]&T^\ast Y\ar@{=}[ur]
}
\end{gathered}
\end{align}
where the  square indicated by ``$\blacksquare$'' is Cartesian by the second column of \eqref{eq:saito:38},  $0^\ast_\bullet$ means the zero section of the respectively vector bundles, 
the morphisms indicated by $\pr_\bullet^\ast$ are the first projections, $i_K$ is defined in \eqref{eq:saito:38}, $di$ is induced by $i$,
all other maps are canonical morphisms.
Note that $i_K\colon K\to N_{Y/X}\times_Y\widetilde Y$ is of codimension $1$ and the associated conormal sheaf equals to $q^\ast\mathcal O_{\widetilde Y}(-1)$ by \eqref{eq:saito:35} and \eqref{eq:saito:36}. We have
\begin{align}\nonumber
&\pi_\ast ({\rm pr}_\pi^!Z\cdot_{{\rm id}_{T^\ast X\times_X\widetilde X}}K)
\overset{\eqref{eq:refinedGysing02}}{=}(i_\ast \tilde\pi_\ast)( i_K^!i_N^!\pr^{\prime!}_{i}\pr_{\pi}^!Z)=i_\ast\tilde\pi_\ast i_K^!i_N^!\pr^{\prime!}_{\tilde \pi}\pr_i^!Z\\
\nonumber&\overset{(1)}{=}i_\ast \tilde\pi_\ast q_\ast i_K^!i_N^!\pr^{\prime!}_{\tilde \pi}\pr_i^!Z=i_\ast\tilde\pi_\ast p_\ast i_{K\ast}i_K^! i_N^!\pr^{\prime!}_{\tilde \pi}\pr_i^!Z\\
\label{eq:saito:101} &\overset{(2)}{=}i_\ast \tilde\pi_\ast p_\ast(c_1(p^{\ast}\mathcal O_{\widetilde Y}(1))\cap  i_N^!\pr^{\prime!}_{\tilde \pi}\pr_i^!Z)\overset{(3)}{=}i_\ast\tilde\pi_\ast(c_1(\mathcal O_{\widetilde Y}(1))\cap p_\ast (i_N^!\pr^{\prime!}_{\tilde \pi}\pr_i^!Z))\\
\nonumber&\overset{(4)}{=}i_\ast\tilde\pi_\ast(c_1(\mathcal O_{\widetilde Y}(1))\cap 0^{\prime!}_Ydi^\prime_\ast(\pr^{\prime!}_{\tilde \pi}\pr_i^!Z))\overset{(5)}{=}i_\ast\tilde\pi_\ast(c_1(\mathcal O_{\widetilde Y}(1))\cap 0^{\prime!}_Y \pr_{\tilde\pi}^!(di_\ast\pr_i^!Z))\\
\nonumber&\overset{(6)}{=}i_\ast\tilde\pi_\ast(c_1(\mathcal O_{\widetilde Y}(1))\cap \tilde \pi^\ast 0^{!}_Y (di_\ast\pr_i^!Z))\overset{(7)}{=}i_\ast 0^{!}_Y (di_\ast\pr_i^!Z)\overset{\eqref{eq:pullbackCCdef}}{=}i_\ast 0^{!}_Y (i^\ast Z)
\end{align}
where
\begin{itemize}
\item the step (1) follows from the fact that $i_K^!i_N^!\pr^{\prime!}_{\tilde \pi}\pr_i^!Z$ is supported on the zero section $\widetilde Y\subseteq K$, and $q_\ast$ is an identity map on the support of $i_K^!i_N^!\pr^{\prime!}_{\tilde \pi}\pr_i^!Z$.
\item the step (2) follows from the definition of Gysin map for divisors \cite[Proposition 2.6 (b)]{Fulton}.
\item the step (3) follows from the projection formula \cite[Theorem 3.2 (c)]{Fulton}.
\item the step (4), (5) and (6) follow from the push-forward formula and the compatibility property  of Gysin maps \cite[Theorem 6.2  (a) and (c)]{Fulton}.
\item the step (7) follows from \cite[Proposition 3.1 (a) (ii)]{Fulton} since $\widetilde Y\to Y$ is a $\mathbb P^1$-bundle over $Y$ (when $c=2$).
\end{itemize}
Finally, \eqref{eq:blowupNEW} follows from \eqref{eq:saito:70}, \eqref{eq:saito:90} and \eqref{eq:saito:101}. We finish the proof.
\end{proof}

\subsection{}Now, we turn to the fibration formula for Gysin maps. We first introduce a definition.
\begin{definition}\label{def:gp}
Let $X$ be a smooth scheme purely of dimension $d$ over a field $k$ and let $Y$ be a smooth  connected curve over $k$.
Let $C\subseteq T^*X$ be a conical closed subset purely of dimension $d$.
We say that $f\colon X\to Y$ is a good fibration  with respect to $C$ if it satisfies the following properties:
\begin{enumerate}
\item There exist finitely many closed points $u_1,\ldots, u_m$ of $X$ such that $f$ is 
$C$-transversal on $X\setminus\{u_1,\ldots, u_m\}$. 
We call such closed points $u_1,\ldots, u_m$ isolated characteristic points of $f$ with respect to $C$.
\item If $u_i\neq u_j$, then $f(u_i)\neq f(u_j)$.
\item The point $u_i$ is purely inseparable over $f(u_i)$ for  $i=1,\cdots,m$.
\end{enumerate}

If $f$ only satisfies the conditions {\rm (1)} and {\rm (2)}, then we say $f$ is a pre-good fibration with respect to $C$.
\end{definition}

If $f\colon X\rightarrow Y$ is a pre-good fibration with respect to $C$ and if $k$ is perfect, then
by Lemma \ref{lem:Xcurvepropertrans} below, there exists a finite set  $\{v_1,\ldots,v_n\}$ of closed points of $Y$ such that for every closed point $v$ of $Y\setminus\{f(u_1),\ldots,f(u_m), v_1\ldots,v_n\}$,   $X_{v}$ is smooth and  the closed immersion $X_{v} \to X$  is properly $C$-transversal.
\begin{lemma}\label{lem:Xcurvepropertrans}
Let $X$ be a smooth scheme purely of dimension $d$ over a perfect  field $k$ and let $Y$ be a smooth  connected curve over $k$.
Let $C\subseteq T^*X$ be a conical closed subset purely of dimension $d$.
Let $f\colon X\to Y$ be a smooth and $C$-transversal morphism. Then we have
\begin{enumerate}
\item For any closed point $v\in Y$, $X_v\to X$ is $C$-transversal.
\item There exists an open dense  subscheme $V\subseteq Y$ such that 
 $X_v\to X$ is properly $C$-transversal for any $v\in V$.

\end{enumerate}
\end{lemma}
\begin{proof}
(1)Let $v\in Y$ be a closed point. Since $k$ is a perfect field, the closed point $v$ is a smooth scheme over $k$.
We apply \cite[Lemma 3.9.2]{Saito} to the following Cartesian diagram 
\begin{equation}
\begin{gathered}
\xymatrix{
X\ar[d]&X_{v}\ar[d]\ar[l]\\
Y&\ar[l]v
}
\end{gathered}
\end{equation}
then the  immersion $X_{v}\to X$ is $C$-transversal.

(2)
We consider the reduced scheme structure on every irreducible component $C_a$ of $C$.
By generic flatness, there exists an open dense subset $V\subseteq Y$ such that  $C_a\times_YV\to V$ is flat. Then $C_a\times_YV\to V$ is purely of relative dimension $d-1$ (cf. \cite[Corollary 9.6]{Hartshorne}). Since $C_a\times_Yv=C_a\times_XX_v$ is purely of dimension $d-1$ for any  irreducible component $C_a$ of $C$, then $X_v\to X$ is properly $C$-transversal for any $v\in V$. 
\end{proof}
\subsection{}
Let $X$ be a smooth scheme purely of dimension $d$ over any field $k$ and $C\subseteq T^\ast X$  a closed conical subset purely of dimension $d$.
Let $f\colon X\to Y$ be a morphism to a smooth connected curve $Y$ over $k$. 
Let $u\in X$ be  at most an isolated characteristic point of $f$ with respect to $C$ and $\omega$ a basis of $T^\ast Y$ on a neighborhood of $f(u)\in Y$. Then $\omega$ defines a section $Y\to T^\ast Y$ on a neighborhood of $f(u)$. By base change, it defines a section $X\to T^\ast Y\times_Y X$. Let $f^\ast \omega$ be the composition $X\to T^\ast Y\times_Y X\xrightarrow{df}T^\ast X$, which is well-defined on an open neighborhood of $u$.
The intersection of an irreducible component $C_a$ of $C$ and $f^*\omega$ consists of at most one isolated point $f^\ast\omega(u)\in T^*_uX$ 
in the fiber of $u$. Hence  the intersection numbers  $(C_a, f^*\omega)_{T^*X, u}$ are well-defined and they are independent of the choice of the basis $\omega$ (cf. \cite[Section 5.2]{Saito}).
\begin{lemma}\label{lem:fibration100}
Let $X$ be a smooth scheme over a field $k$ and $g\colon L\to E$ be a morphism of vector bundles on $X$ such that $L$ is a line bundle and $E$ is a vector bundle of rank $d$. Let $s_1\colon X\to E$ and $s_2\colon X\to L$ be the zero sections. 
Let
\begin{equation}\label{lem:eq:c1}
c_1(L)=\sum_{V\in J}{\rm ord}_{V}\cdot [V]
\end{equation}
be the associated divisor of the first Chern class (cf. \cite[2.5]{Fulton}), where $J$ is a finite set of integral sub-schemes of $X$ of codimension 1, and ${\rm ord}_V\in\mathbb Z$ is the coefficient of the component $[V]$.
Let $C\subseteq E$ be a conical closed subset purely of dimension $d$. Assume that $g^{-1}(C)$ is contained  in
$s_2(X)\cup\bigcup_{u\in I}L_{u}$
for a finite set $I$ of closed points of $X\setminus \left\{\bigcup_{V\in J}V\right\}$ and assume that $V$ is smooth for all $V\in J$. We choose a basis $\omega$ of $L$ at each point $u\in I$.
For each $V\in J$, we have a commutative diagram
\begin{align}\label{eq:lem:eq:c2}
\begin{gathered}
\xymatrix{
L\times_XV\ar[r]\ar[d]\ar@{}|\Box[rd]&L\ar[d]^g&X\ar@{_(->}[l]_{s_2}\ar@{=}[d]\\
E\times_XV\ar[r]\ar[d]\ar@{}|\Box[rd]&E\ar[d]&X\ar@{_(->}[l]_{s_1}\\
V\ar[r]^-{i_V}&X
}
\end{gathered}
\end{align}
where $i_V\colon V\to X$ is the closed immersion.
Then for any $d$-cycle $Z$ supported on $C$, we have the following equality
\begin{equation}\label{lem:eq:fib}
s_1^!(Z)=\sum_{u\in I}(Z,g(\omega))_{E,u}\cdot [u]+\sum_{V\in J}{\rm ord}_V\cdot g^!(i_V^!(Z)) \quad {\rm in}~ CH_0(X),
\end{equation}
where $i_V^!$ and $g^!$ are the refined Gysin maps $($cf. \cite[6.2]{Fulton}$)$, and $g(\omega)$ is the composition $X\xrightarrow{\omega}L\xrightarrow{g}E$ which is well-defined on an open neighborhood of $I$.
\end{lemma}
Note that, since we use refined Gysin maps, by the assumption on $g^{-1}(C)$, the support of $g^!(i_V^!(Z))$ is contained in the zero section $V\subseteq L\times_XV$.
\begin{proof}By the assumption, we may write the Gysin pullback $g^!(Z)$ in the following form
\begin{equation}\label{eq:fibration100:101}
g^!(Z)=A+\sum_{u\in I} b_u\cdot [L_u]
\end{equation}
where $A\in CH_1(s_2(X))$ and $b_u\in\mathbb Z$ for $u\in I$. 
From (\ref{eq:fibration100:101}), we can calculate $b_u$ in the following way
\begin{equation}\label{eq:fibration100:101a}
b_u=(g^!(Z),\omega)_{L,u}=(Z, g(\omega))_{E,u}
\end{equation}
by the projection formula \cite[Example 8.1.7]{Fulton}.
By (\ref{eq:fibration100:101}) and (\ref{eq:fibration100:101a}), we have
\begin{equation}\label{eq:fibration100:102}
s_1^{!}(Z)=s_2^!(g^!(Z))=s_2^{!}(A)+\sum_{u\in I} (Z, g(\omega))_{E,u}\cdot [u].
\end{equation}
Now we only need to calculate $s_2^{!}(A)$.
Since $A$ is supported on the zero section $X\subseteq L$, by the self-intersection formula \cite[Proposition 2.6 (c)]{Fulton}, we have $s_2^{!}(A)=c_1(L)\cap A$.
By \eqref{eq:lem:eq:c2}, we have
\begin{align}\label{eq:fibration100:104}
s_2^{!}(A)&=c_1(L)\cap A\overset{\eqref{lem:eq:c1}}{=}\sum_{V\in J}{\rm ord}_V\cdot (V\cdot_{{\rm id}_L} A)\\
\nonumber&\overset{\eqref{eq:refinedGysing01}}{=}\sum_{V\in J}{\rm ord}_V\cdot i_V^!A\overset{\eqref{eq:fibration100:101}}{=}\sum_{V\in J}{\rm ord}_V\cdot i_V^! g^!(Z)\\
\nonumber&=\sum_{V\in J}{\rm ord}_V\cdot g^!(i_V^!(Z))
\end{align}
where the last step follows from the commutativity property of the refined Gysin homomorphisms (cf. \cite[Theorem 6.4]{Fulton}).
By (\ref{eq:fibration100:102}) and (\ref{eq:fibration100:104}),
we get (\ref{lem:eq:fib}). This finishes the proof.
\end{proof}
\begin{lemma}\label{cc}
Let $X$ be a smooth scheme purely of dimension $d$ over a field $k$,
 let $C\subseteq T^*X$ be a conical closed subset purely of dimension $d$ and  $Z$  a $d$-cycle supported on $C$.
Let $f\colon X\to Y$ be a pre-good fibration to a smooth  connected curve $Y$ over $k$ with respect to $C$ and  $u_1,\ldots, u_m\in X$  the isolated characteristic points of $f$ with respect to $C$.
Let $\omega$ be a non-zero rational 1-form on $Y$ which has neither poles nor zeros at $f(u_1),\ldots, f(u_m)$ such that for all closed point $v\in Y$ with ${\rm ord}_v(\omega)\neq 0$, $X_v$ is smooth and $i_v\colon X_v\rightarrow X$ is $C$-transversal.

Then we have an equality in $CH_0(X)$:
\begin{align}\label{fiber1}
0_X^!(Z)=\sum_{i=1}^m(Z, f^*\omega)_{T^*X, u_i}\cdot[u_i]+\sum_{v\in |Y|}\ord_v(\omega)\cdot i_{v\ast}(0^!_{X_v}(i_v^\ast Z)).
\end{align}
where  $i_{v\ast}\colon CH_0(X_v)\rightarrow CH_0(X)$ is the canonical homomorphism. For the definition of the Gysin map $0^!_{\bullet}$, see Subsection \ref{sub:gysinzero}.
\end{lemma}
\begin{proof}
We apply Lemma \ref{lem:fibration100} to  $L=T^\ast Y\times_YX$,  $E=T^\ast X$, $J=\left\{X_v\mid v\in|Y| {\rm ~such~that~}{\rm ord}_v(\omega)\neq 0\right\}$ and $g=df\colon T^\ast Y\times_YX\to T^\ast X$, where $df$ is the morphism induced by $f$. We obtain the following equality in $CH_0(X)$
\begin{align}\label{eq:cc:100}
0_X^!(Z)=\sum_{i=1}^m(Z, f^*\omega)_{T^*X, u_i}\cdot[u_i]+\sum_{v\in |Y|}\ord_v(\omega)\cdot df^! \pr_{i_v}^!(Z)
\end{align}
where $\pr_{i_v}\colon T^\ast X\times_XX_v\to T^\ast X$ is the first projection (base change of the map $X_v\to X$).
Let $v\in |Y|$ with ${\rm ord}_v(\omega)\neq 0$, then $X_v$ is smooth and $T^\ast Y\times_YX_v\simeq T^\ast_{X_v}X$. Moreover, we have a commutative diagram 
\begin{equation}\label{eq:cc:105}
\begin{gathered}
\xymatrix{
T^\ast X&T^\ast X\times_XX_v\ar[d]_-{di_v}\ar[l]_-{{\rm pr}_{i_v}}&T^\ast Y\times_YX_v\simeq T^\ast_{X_v}X\ar[d]_-{\rm pr_2}\ar[l]_-{df}\ar@{}[ld]|-{\Box}&X_v\ar@{_(->}[l]\ar@{=}[ld]\\
&T^\ast X_v&X_v\ar[l]^-{0_{X_v}}
}
\end{gathered}
\end{equation}
where $\pr_2$ is the projection.
Since $X_v\rightarrow X$ is $C$-transversal, $(C\times_XX_v)\cap {\rm ker}(di_v)=(C\times_XX_v)\cap T^\ast_{X_v}X$ is contained in the zero section $X_v$ of $T^\ast X\times_XX_v$. We have
\begin{align}
\label{eq:cc:106}
df^!{\rm pr}^!_{i_v}(Z)&\overset{(a)}{=}{\rm pr}_{2\ast} df^!{\rm pr}^!_{i_v}(Z)\overset{(b)}{=}0_{X_v}^!di_{v\ast}{\rm pr}_{i_v}^!(Z)
\overset{\eqref{eq:pullbackCCdef}}{=}0_{X_v}^!(i_v^\ast Z),
\end{align}
where in (a) we use the fact that $df^!{\rm pr}^!_{i_v}(Z)$ is supported on the zero section $X_v\subseteq T^\ast_{X_v}X$ and $\pr_2$ is the identity map on $X_v$, 
and (b) follows from the push-forward formula \cite[Theorem 6.2 (a) and (c)]{Fulton} and the fact that $df$ and $0_{X_v}$ are regular immersions of same codimension.
By (\ref{eq:cc:100}) and (\ref{eq:cc:106}), we get (\ref{fiber1}). This finishes the proof.
\end{proof}

\section{Existence of good pencil}\label{sec:eogp}
\subsection{}The existence of a pre-good pencil is proved by Saito and Yatagawa \cite{Saito_Yatagawa}. 
We first recall their result and we follow the same notation as in  \cite[Section 2]{Saito_Yatagawa} and \cite[Section 3]{Saito}.
Let $X$ be a smooth projective scheme purely of dimension $d$ over a  field $k$.
Let $\mathcal{L}$ be a very ample invertible $\mathcal O_{X}$-module
and 
let $E\subseteq \Gamma (X,\mathcal{L})$ be a sub $k$-vector space defining
a closed immersion 
$i\colon X\rightarrow \mathbb{P}=
\mathbb{P}(E^\vee)={\rm Proj}_{k}({\rm Sym}^{\bullet}E)$. 
The dual projective space $\mathbb{P}^{\vee}
=\mathbb\mathbb{P}(E)$ parametrizes hyperplanes in $\mathbb{P}$.
The universal hyperplane
$Q=\{(x,H)\mid  x\in H\}\subseteq \mathbb{P}\times_{k}\mathbb{P}^{\vee}$ can be identified
with the covariant projective space bundle 
$\mathbb{P}(T^{\ast}\mathbb{P})$, cf. \cite[Subsection 3.2]{Saito}.
We also identify the universal family of
hyperplane sections
$X\times_{\mathbb{P}}Q$
with
$\mathbb{P}(X\times_{\mathbb{P}}T^{\ast}\mathbb{P})$.

\subsection{}For a line $L \subseteq \mathbb{P}^{\vee}$, let $A_{L}$ be the axis $\bigcap_{t\in L}H_{t}\subseteq \mathbb P$ of $L$.
We define $p_L\colon 
X_{L}=\{(x,H_{t})\mid x\in X\cap H_{t},\
t\in L\}\to L$ by the Cartesian
diagram
\begin{equation}
\begin{CD}
X_L@>>>X\times_{\mathbb{P}}Q\\
@V{p_L}VV@VVV\\
L@>>>{\mathbb P}^\vee.
\end{CD}
\label{eqpL}
\end{equation}
The projection $\pi_L\colon X_{L}\rightarrow X$ 
is an isomorphism
on the complement
$X_{L}^{\circ}=X- X\cap A_{L}$ (cf. \cite[Expos\'e XVIII, 3]{SGA7II} and \cite[Expos\'e XVII, 2]{SGA7II}).
Let $p_{L}^{\circ}\colon
X_{L}^{\circ}\rightarrow L$ be the restriction of $p_{L}$.
If moreover $A_{L}$ meets $X$ transversally, then
$\pi_L\colon X_{L}\rightarrow X$ is the blow-up of $X$ along $X\cap A_{L}$ and hence $X_{L}$ is smooth over $k$.

\subsection{}
Let $C\subseteq T^{\ast}X$ be a closed conical subset purely of dimension $d$.
Let $\widetilde C$ be the inverse image
by the surjection
$di \colon X\times_{\mathbb{P}}T^{\ast}\mathbb{P}\rightarrow T^{\ast}X$.
We consider the following conditions.

\begin{itemize}
\item[(E)] For every pair $(u,v)$ of distinct closed points of the base change $X_{\bar k}$ to an algebraic closure $\bar k$ of $k$,
the restriction
\begin{equation}
E\otimes_k\bar k\subseteq \Gamma(X,\mathcal{L})\otimes_k\bar k\rightarrow \mathcal{L}_{u}/\mathfrak{m}_{u}^{2}\mathcal{L}_{u}\oplus \mathcal{L}_{v}/\mathfrak{m}_{v}^{2}\mathcal{L}_{v} \notag
\end{equation}
is surjective.
\item[(C)] 
For every irreducible component $C_{a}$ of $C$,
the inverse image $\widetilde C_a
\subseteq \widetilde C
$ of $C_{a}$ by the surjection
$di \colon X\times_{\mathbb{P}}T^{\ast}\mathbb{P}\rightarrow T^{\ast}X$
is not contained in the $0$-section
$X\times_{\mathbb{P}}T^{\ast}_{\mathbb{P}}\mathbb{P}$.
\end{itemize}
By \cite[Lemma 3.19]{Saito},  after replacing ${\mathcal L}$ and $E$ by ${\mathcal L}^{\otimes n}$
and the image $E^{(n)}$ of $E^{\otimes n}\to \Gamma(X,{\mathcal L}^{\otimes n})$ for $n\geqq 3$ if necessary,
the condition (E) and (C) are satisfied.
For each irreducible component $C_{a}$ of $C$, we regard the projectivization
$\mathbb{P}(\widetilde C_{a})$ of  $\widetilde C_a\subseteq \widetilde C$
as a closed subset of
$\mathbb{P}(\widetilde C) \subseteq \mathbb{P}(X\times_{\mathbb{P}}T^*\mathbb{P}) =X\times_{\mathbb{P}}Q$.

The existence of a pre-good pencil is proved in \cite[Lemma 2.3]{Saito_Yatagawa}.
\begin{lemma}[{\cite[Lemma 2.3]{Saito_Yatagawa}}]\label{lem:pregoodpencil}
Let $X$ be a smooth projective scheme purely
of dimension $n$ over an algebraically closed field $k$.
Let $C\subseteq T^{\ast}X$ be a closed conical subset purely
of dimension $n$.
Let $\mathcal{L}$ be a very ample invertible $\mathcal O_{X}$-module and $E\subseteq \Gamma(X,\mathcal{L})$ a sub $k$-vector space 
satisfying the conditions {\rm (E)} and {\rm (C)}. 
Let $\mathbb{P}=
\mathbb{P}(E^\vee)$ and $\mathbf{G}=\mathrm{Gr}(1,\mathbb{P}^{\vee})$ be the Grassmannian variety
parameterizing lines in $\mathbb{P}^{\vee}$.
Then, there exists  a dense open subset $U\subseteq \mathbf{G}$ consisting of lines $L\subseteq \mathbb{P}^{\vee}$
satisfying the following conditions {\rm (1)}--{\rm (6):}
\begin{enumerate}
\item The axis $A_{L}$ meets $X$ transversally and the immersion $X\cap A_{L}\rightarrow X$ is $C$-transversal.
\item The blow-up $\pi_L\colon X_L\to X$ is $C$-transversal.
\item The morphism $p_{L}\colon X_{L}\rightarrow L$ is a pre-good fibration with respect to $\pi^\circ_LC$.

\item No isolated characteristic point of $p_{L}$ is contained in  the inverse image by $\pi_L\colon X_{L}\rightarrow X$ 
of the intersection $X\cap A_{L}$.

\item For every irreducible component $C_a$ of $C$,
the intersection $X_L\cap {\mathbb P}(\widetilde C_a)$ is non-empty.

\item For every pair of irreducible components $C_a\neq C_b$ of $C$,
the intersection $X_L\cap {\mathbb P}(\widetilde C_a)\cap {\mathbb P}(\widetilde C_b)$ is empty.
\end{enumerate}
\end{lemma}
\subsection{}In general, if $k$ is not necessarily algebraically closed, we say that a line $L\subseteq \mathbb P^\vee$ is a {\it pre-good pencil} with respect to $C$ if $L$ satisfies the conditions {\rm (1)}--{\rm (6)} in the above Lemma. 
If $L$ is a pre-good pencil with respect to $C$, then $X_L$ is a smooth scheme and $p_L\colon X_L\rightarrow L$ is a pre-good fibration with respect to $\pi^\circ_L C$.

\begin{lemma}\label{EGAIV:dim0}
Let $X$ be a smooth  scheme purely of dimension $d$ over  a field $k$, $i\colon Y\to X$  a closed immersion of codimension $c$ such that $Y$ is smooth and purely of dimension $d-c$.  Let $C$ be a conical closed subset of $T^\ast X$ such that $i$ is properly $C$-transversal.
Let $\pi\colon W\to X$ be the blow up of $X$ along $Y$. Then
$\pi$ is properly $C$-transversal and $Z:=W\times_X Y\hookrightarrow W$ is properly  $\pi^\circ C$-transversal. 
\end{lemma}
\begin{proof}
Let $p\colon Z\to Y$ be the projection and $U=X\setminus Y$. 
Since $i$ is properly $C$-transversal and  the smooth morphism $p$ is properly $i^\circ C$-transversal, the composition $Z\to X$ is also properly $C$-transversal by \cite[Lemma 7.2.2]{Saito}. 
Applying \cite[Lemma 7.2.2]{Saito} again to the composition $Z\to W\xrightarrow{\pi} X$, we get that $\pi$ is properly $C$-transversal on a neighborhood of $Z\subseteq W$
and $Z\to W$ is properly $\pi^\circ C$-transversal. Since moreover $\pi$ is an isomorphism over $U$, hence $\pi$ must be properly $C$-transversal.
\end{proof}
\subsection{}\label{sub:defmeetproper}
Let $f\colon C\to Z$ and $h\colon W\to Z$ be morphisms of schemes of finite type over a field $k$. Assume that every irreducible component of $C$ is of dimension $n$ and that $h$ is locally of complete intersection of relative virtual dimension $d$. We say that {\it $h\colon W\to Z$ meets $f\colon C\to Z$ properly} if every irreducible component of $C\times_ZW$ is 
of dimension $n+d$.
\begin{lemma}[{\cite[Lemma 1.2.3]{Saito17b}}]\label{lem:openforpropertran}Let $\mathbb P$ be a projective space over a field $k$ and $X\subseteq \mathbb P$ a smooth closed subscheme purely of dimension $d$ over $k$. Let $C\subseteq T^\ast X$ be a conical closed subset such that   $C$ is purely of dimension $d$. Let $\mathbb{P}^{\vee}$ be the dual projective space and $\mathbf{G}=\mathrm{Gr}(1,\mathbb{P}^{\vee})$ the Grassmannian variety
parameterizing lines in $\mathbb{P}^{\vee}$. Then the line $L\in \mathbf{G}$ such that the immersion $ A_L\cap X\to X$ meets $C\to X$ properly form a dense open subset of the Grassmannian variety $\mathbf{G}$.
\end{lemma}
\begin{proof}This follows directly from \cite[Lemma 1.2.3]{Saito17b}. For completeness, we include the proof. Let $\mathbf A\subseteq \mathbb P\times\mathbf G$ be the universal family of linear subspaces of codimension $2$, i.e., 
\begin{equation}\mathbf A=\left\{(x,L)\in \mathbb P\times\mathbf G\mid x\in A_L\right\}.
\end{equation}
We consider the Cartesian diagram 
\begin{align}
\begin{gathered}
\xymatrix{
&C_{\mathbf A}\ar@{}|\Box[rd]\ar[ddl]\ar[r]\ar[d]&C\ar[d]\\
&X_{\mathbf A}\ar[d]\ar[r]\ar@{}|\Box[rd]&X\ar[d]\\
\mathbf G&\mathbf A\ar[l]\ar[r]&\mathbb P.
}
\end{gathered}
\end{align}
Since the projection $\mathbf A\to\mathbb P$ is smooth of relative dimension $\dim \mathbf G-2$, we have $\dim C_{\mathbf A}=\dim\mathbf G+d-2$. Hence the open subset $U$ of $\mathbf G$ consisting of $L$ such that $\dim (C\times_{\mathbb P}A_L)=\dim (C\times_X(X\cap A_L))\leq d-2$ is dense. For $L\in U$, by \cite[Lemma 7.2.1]{Saito}, $\dim (C_a\times_X(X\cap A_L))\geq d-2$ for any irreducible component $C_a$ of $C$. Thus we must have  $ \dim (C_a\times_X(X\cap A_L))= d-2$ for all $a$, i.e., $ A_L\cap X\to X$ meets $C\to X$ properly.
\end{proof}
\begin{lemma}\label{pencil}
Let $k$ be an infinite field. Let $X$ be a smooth projective scheme purely of dimension $n$ over  $k$.
Let $C\subseteq T^{\ast}X$ be a closed conical subset purely of dimension $n$.
Let $\mathcal{L}$ be a very ample invertible $\mathcal O_{X}$-module and $E\subseteq \Gamma(X,\mathcal{L})$ a sub $k$-vector space 
satisfying the conditions {\rm (E)} and {\rm (C)}. 
Let $\mathbb{P}=
\mathbb{P}(E^\vee)$ and  $\mathbf{G}=\mathrm{Gr}(1,\mathbb{P}^{\vee})$ be the Grassmannian variety parameterizing lines in $\mathbb{P}^{\vee}$.
Then, there exists a dense open subset $U\subseteq \mathbf{G}$ consisting of lines $L\subseteq \mathbb{P}^{\vee}$ satisfying the following conditions 
{\rm (1)}--{\rm (4):}
\begin{enumerate}
\item $L_{\bar k}$ is a pre-good pencil with respect to $C_{\bar k}$. Let $u_1,\ldots, u_m$ be the isolated characteristic points of $p_L\colon X_L\to X$ with respect to $\pi_L^\circ C$.
\item The point $u_i$ is purely inseparable over $p_L(u_i)$ for all $i=1,\cdots,m$.
\item The immersion $i\colon A_L\cap X\to X$ is properly $C$-transversal.
\item The blow up $\pi_L\colon X_L\to X$ along $A_L$ is properly $C$-transversal.
\end{enumerate}

\end{lemma}
We say that a line $L\subseteq \mathbb P^\vee$ is a {\it good pencil} with respect to $C$ if $L$ satisfies  the conditions {\rm (1)}-{\rm (4)} in the above Lemma. 
If $L$ is a good pencil, then  $X_L$ is a smooth scheme and $p_L\colon X_L\rightarrow L$ is a good fibration.
\begin{proof}
We first consider the case where $k$ is algebraically closed. 
In this case, the condition (2) is trivial.
Let $U\subseteq \mathbf{G}$ be the dense open subset defined in Lemma \ref{lem:pregoodpencil}.
Let $U^\prime\subseteq \mathbf{G}$ be the dense open subset defined by Lemma \ref{lem:openforpropertran}.
Let $U_0=U\cap U^\prime$. Then for any $L\in U_0$, the line $L$ satisfies the condition (1) by  Lemma \ref{lem:pregoodpencil}. In particular $A_L\cap X$ is smooth and $A_L\cap X\to X$ is $C$-transversal.
Since $L\in U^\prime$, $A_L\cap X\to X$ meets $C\to X$ properly. Thus every irreducible component 
of $C\times_X(A_L\cap X)$ is of dimension $d-2$ (cf. Subsection \ref{sub:defmeetproper}), i.e.,  $A_L\cap X\to X$ is properly $C$-transversal. 
Since $A_L$ meets $X$ transversally, $X_L\to X$ is the blow up of $X$ along $A_L\cap X$ (cf. \cite[Expos\'e XVIII, 3]{SGA7II}).
By Lemma \ref{EGAIV:dim0}, $X_L\to X$ is also 
properly $C$-transversal. Thus any  $L\in U_0$ satisfies the conditions (1)-(4) of Lemma \ref{pencil}. This proves the lemma in the case where $k$ is algebraically closed.

 For a general field $k$, let $U'\subseteq \mathbf{G}_{\bar{k}}$ be an open subset such that $L\in U'$ satisfies the conditions (1)-(4) of Lemma \ref{pencil} for $X_{\bar{k}}$ and let $Z'=\mathbf{G}_{\bar{k}}\setminus U'$ be its complement.
Let $Z$ be the image of $Z'$ by $\mathbf{G}_{\bar{k}}\to \mathbf{G}$ and let $U$ be its complement $U=\mathbf{G}\setminus Z$. Then the extension $L_{\bar{k}}$ of a $k$-rational point $L\in U(k)$ satisfies the conditions (1)-(4).

For any morphism $f\colon X\rightarrow Y$ (respectively $g\colon Y\rightarrow X$), 
we note that if $f_{\bar k}$ (respectively $g_{\bar k}$) is $C_{\bar{k}}$-transversal, then  $f$ (respectively $g$) is also $C$-transversal. 
Let  $L\subseteq \mathbb P^\vee$ be a line such that $L_{\bar k}$ satisfies the conditions (1)-(4) of Lemma \ref{pencil} with respect to $C_{\bar k}$. Let $u$ be an isolated characteristic point of $p_L$ and put $v=p_L(u)$.
Since there exists only one isolated characteristic point in $X_L\otimes_k{\overline{k(v)}}$, the extension $k(u)/k(v)$  must be purely inseparable.
Hence $L\in U(k)$ satisfies the conditions (1)-(4).

When $k$ is an infinite field, there exists a $k$-rational point in $U$ since the Grassmannian varieties are rational.
This proves the lemma.
\end{proof}

\section{Characteristic cycle}\label{sec:cc}
\subsection{}We prove Theorem \ref{main} in this section.
First, we recall the definitions of singular support and characteristic cycle for any \'etale sheaf with finite coefficients. Then we introduce these notions for $\ell$-adic sheaves. In order to 
define the characteristic cycles for $\bQl$-sheaves, following a suggestion of T. Saito, we construct a decomposition homomorphism $d_X\colon K(X,\bQl)\to K(X,\bFl)$ for certain   schemes $X$ in \ref{sub:defidecomp}, which can be viewed as a generalization of the decomposition homomorphisms defined in \cite[15.2]{Serre97} (See also \cite[Th\'eor\`eme 2]{Serre68}).
\begin{definition}\label{Milnor}
Let $k$ be a perfect field of characteristic $p$. Let $\Lambda$ be  a finite field of characteristic $\ell(\ell\neq p)$ or $\Lambda=\bQl$.
Let $X$ be a smooth scheme over $k$  and  $\F\in D_c^b(X,\Lambda)$. Let $C$ be a conical closed subset of $T^\ast X$. 
\begin{enumerate}
\item $($\cite[1.3]{Beilinson}$)$ We say that $\mathcal F$ is micro-supported on $C$ if the following condition holds:

Let $(h, f)$ be any pair of morphisms $h\colon W\to X$ and $f\colon W\to Y$ between smooth schemes over $k$ such that $h$ is $C$-transversal and $f$ is $h^\circ C$-transversal.
Then $f$ is $($universally$)$ locally acyclic with respect to $h^*\F$.

\item Assume that $C$ and $X$ are purely of dimension $d$. 
Assume that $\mathcal F$ is micro-supported on $C$.
Let $Z$ be a $d$-cycle supported on $C$.
We say that $Z$ satisfies the Milnor formula for $\mathcal F$ if the following condition holds:

Let  $(j, f)$ be any pair of an \'etale morphism $j:U\to X$ and a morphism $f:U\to Y$ to a smooth curve $Y$ over $k$.
Let $u\in U$ be a closed point such that $u$ is at most an isolated characteristic point of $f$ with respect to $j^\circ C$.
Then we have
\begin{equation}\label{eqYZ:milnor}
-\dimtot\phi_u(j^*\F, f)=(j^*Z, df)_{T^*U, u},
\end{equation}
which is called the Milnor formula.  Here $\phi$ denotes the vanishing cycle functor and ${\dimtot}$ is the total dimension (cf. Subsection \ref{sub:localconsforfinite}).
\end{enumerate}
\end{definition}
\subsection{}We have the following existence theorem for the singular support and characteristic cycle, which is proved by Beilinson and T. Saito respectively.
\begin{theorem}\label{thm:beilinsonsaito} 
Let $k$ be a perfect field of characteristic $p$ and  $\Lambda$ be a finite local ring such that the characteristic $\ell$ of the residue field of $\Lambda$ is invertible in $k$.
Let $X$ be a smooth scheme over $k$ and $\mathcal F\in D_{\rm ctf}^b(X,\Lambda)$. Here $D_{\rm ctf}^b(X,\Lambda)$ is the full subcategory of $D^b(X,\Lambda)$ consisting of objects $\mathcal G$ with finite Tor-dimension  such that $\mathcal H^i(\mathcal G)$ are constructible for all $i$.

$(1)$ {\rm (Beilinson, \cite[1.3]{Beilinson})} There is a unique minimal conical closed subset $C$ of $T^*X$ such that $\mathcal F$ is micro-supported on $C$. We  denote this unique element by $SS\F$, and we call it the singular support of $\mathcal F$.  If $X$ is purely of dimension $d$,  then every irreducible component of $SS(\mathcal F)$ is of dimension $d$.

$(2)$ {\rm (Saito, \cite[Theorem 5.9, Theorem 5.19]{Saito})} Assume that $X$ is purely of dimension $d$. Then there is a unique $d$-cycle $CC\mathcal F$ supported on $SS\mathcal F$ with $\mathbb Z$-coefficients such that $CC\mathcal F$ satisfies the Milnor formula for $\mathcal F$.

\end{theorem}

\subsection{}We recall the following result in \cite{Saito}, which can be used to define the singular support and characteristic cycle for any $\ell$-adic sheaf.
\begin{lemma}[{\cite{Saito}}]\label{reductionCC}
Let $X$ be a smooth  scheme purely of dimension $d$ over a perfect field $k$ of characteristic $p$. 
Let $\Lambda$ be a finite local ring with residue field $\Lambda_0$ of characteristic $\ell(\ell\neq p)$. For any $\mathcal F\in D_{ctf}^{b}(X,\Lambda)$, we have
\begin{align}
\label{bql-1}SS(\mathcal F)&=SS(\mathcal F\otimes_\Lambda^L \Lambda_0),\\
\label{bql-2}CC(\mathcal F)&=CC(\mathcal F\otimes_\Lambda^L \Lambda_0).
\end{align}
\end{lemma}
\begin{proof}Formula \eqref{bql-1} is proved in \cite[Corollary 4.5.4]{Saito}. By the proof of \cite[Theorem 5.9]{Saito}, for every pair $(j,f)$ of an \'etale morphism
$j\colon U\to X$ and a morphism $f\colon U\to Y$ to a smooth curve over $k$, we have (cf. \cite[Expos\'e XIII, (2.1.13.1)]{SGA7II} or \cite[Corollary D.8, p.349]{KW})
\begin{equation}
{\dimtot}\phi_u(j^\ast\mathcal F,f)={\dimtot}\phi_u(j^\ast(\mathcal F\otimes_\Lambda^L\Lambda_0),f),
\end{equation} 
where $u$ is at most an isolated characteristic point of $f$ with respect to $j^\circ SS(\mathcal F)$. This implies that \eqref{bql-2} by Theorem \ref{thm:beilinsonsaito}(2).
\end{proof}
\subsection{}\label{sub:recallladicsheaf} Now we start to define the characteristic cycle for $\ell$-adic sheaves. 
Let $k$ be a perfect field of characteristic $p$ and $\ell\neq p$ a prime number.
Let $E$ be a finite extension of $\mathbb Q_\ell$, $\mathcal O_E$ the ring of integers of $E$. Let $\mathfrak m_E$ be the maximal ideal of $\mathcal O_E$ with a uniformizer $\pi_E$ and residue field $k_E=\mathcal O_E/\mathfrak m_E$.
Let $X$ be a smooth scheme purely of dimension $d$ over $k$ and  $\mathcal F\in D_c^b(X,\mathcal O_E)$.
Recall that (cf. \cite[(1.1.2)]{Weil2} and \cite[Expos\'e VI, D\'efinition 1.1.1]{SGA5}) $\mathcal F$ is a family 
$(\mathcal F_n)_{n\geq 0}$ of objects $\mathcal F_n\in D_{ctf}^b(X,\mathcal O_E/\mathfrak m_E^{n+1})$ with ismorphisms for all $n\geq 0$
\begin{equation}\label{eq:repforladic}
\mathcal F_{n+1}\otimes_{\mathcal O_E/\mathfrak m_E^{n+2}}^L\mathcal O_E/\mathfrak m_E^{n+1}\xrightarrow{~\simeq~}\mathcal F_{n}.
\end{equation}
Recall that the reduction map
\begin{align}
-\otimes^L_{\mathcal O_E}\mathcal O_E/\mathfrak m_E\colon D_c^b(X,\mathcal O_E)\to D_c^b(X,\mathcal O_E/\mathfrak m_E)
\end{align}
is conservative \cite[Theorem 3.6 (ii)]{Eke90}. 
By Proposition \ref{reductionCC}, for any integer $n\geq 0$, we have  $SS(\mathcal F_n)=SS(\mathcal F_0)=SS(\mathcal F\otimes_{\mathcal O_E}^L\mathcal O_E/\mathfrak m_E)$ and $CC(\mathcal F_n)=CC(\mathcal F_0)=CC(\mathcal F\otimes_{\mathcal O_E}^L\mathcal O_E/\mathfrak m_E)$.
\begin{definition}\label{def:ccforladic}
We define the fake singular support $SS(\mathcal F)$ and the characteristic cycle $CC(\mathcal F)$ of $\mathcal F$ to be
\begin{align}
\label{eq:def:ccforladic1}&SS(\mathcal F):=SS(\mathcal F\otimes_{\mathcal O_E}^L\mathcal O_E/\mathfrak m_E),\\
\label{eq:def:ccforladic2}&CC(\mathcal F):=CC(\mathcal F\otimes_{\mathcal O_E}^L\mathcal O_E/\mathfrak m_E).
\end{align}
\end{definition}
\begin{proposition}\label{thm:milnorforladic}
Under the assumptions in Subsection \ref{sub:recallladicsheaf}.
Then $\mathcal F$ is micro-supported on $SS(\mathcal F)$ $($cf. Definition \ref{Milnor}$)$ and $CC(\mathcal F)$
satisfies the Minor formula \eqref{eqYZ:milnor} for $\mathcal F$.
\end{proposition}
\begin{proof} Let $\mathcal F=(\mathcal F_n)_{n\geq0}$ be a projective system satisfying \eqref{eq:repforladic}. By Lemma \ref{reductionCC}, for any integer $n\geq 0$, the sheaf $\mathcal F_n$ is micro-supported on $SS(\mathcal F)$ and 
$CC(\mathcal F_n)=CC(\mathcal F)$ satisfies the Minor formula \eqref{eqYZ:milnor} for $\mathcal F_n$.

Let us show that $\mathcal F$ is micro-supported on $SS(\mathcal F)$. Let $(h, f)$ be a pair of morphisms $h\colon W\to X$ and $f\colon W\to Y$ between smooth schemes over $k$ such that $h$ is $SS(\mathcal F)$-transversal and $f$ is $h^\circ SS(\mathcal F)$-transversal. We need to show that $f$ is locally acyclic relatively to $h^\ast\mathcal F$.
By \cite[Th. finitude, D\'efinition 2.12]{SGA4h}, we need to show that the canonical morphism
\begin{equation}\label{eq:milnorforladic:1}
(h^\ast\mathcal F)_{\bar w}\to R\Gamma(W_{(\bar w)}\times_{Y_{(f(\bar w))}}\bar t,h^\ast\mathcal F|_{W_{(\bar w)}\times_{Y_{(f(\bar w))}}\bar t})
\end{equation}
is an isomorphism for all geometric points $\bar w\to W$ and $\bar t\to Y_{(f(\bar w))}$, where $W_{(\bar w)}$ is the strict localization of $W$ at $\bar w$. 
Since $f$ is locally acyclic relatively to $\mathcal F_n$ for all $n$, thus we have an isomorphism
\begin{equation}\label{eq:milnorforladic:2}
(h^\ast\mathcal F_n)_{\bar w}\to R\Gamma(W_{(\bar w)}\times_{Y_{(f(\bar w))}}\bar t,h^\ast\mathcal F_n|_{W_{(\bar w)}\times_{Y_{(f(\bar w))}}\bar t}).
\end{equation}
By definition, we have $(h^\ast\mathcal F)_{\bar w}\simeq \varprojlim_n (h^\ast\mathcal F_n)_{\bar w}$ and 
\begin{align*}R\Gamma(W_{(\bar w)}\times_{Y_{(f(\bar w))}}\bar t,h^\ast\mathcal F|_{W_{(\bar w)}\times_{Y_{(f(\bar y))}}\bar t})=\varprojlim_n R\Gamma(W_{(\bar w)}\times_{Y_{(f(\bar w))}}\bar t,h^\ast\mathcal F_n|_{W_{(\bar w)}\times_{Y_{(f(\bar y))}}\bar t}). 
\end{align*}
Applying $\varprojlim_n$ to \eqref{eq:milnorforladic:2}, we prove that the morphism \eqref{eq:milnorforladic:1} is also an isomorphism.

Now let us show that $CC(\mathcal F)$
satisfies the Minor formula \eqref{eqYZ:milnor} for $\mathcal F$. Let $(j, f)$ be any pair of an \'etale morphism $j:U\to X$ and a morphism $f:U\to Y$ to a smooth curve $Y$ over $k$.
Let $u\in U$ be a closed point such that $u$ is at most an isolated characteristic point of $f$ with respect to $j^\circ SS\F$. By the formula \eqref{eq:dimtotophireduction} of Subsection \ref{sub:dimtotphi} below, we have
\begin{align}
\label{eq:milnorforladic:3}
\dimtot (\phi_u(j^\ast\mathcal F,f))=\dimtot (\phi_u(j^\ast\mathcal F_0,f)).
\end{align}
Since we have Milnor formula for $\mathcal F_0\in D_c^b(X,\mathcal O_E/\mathfrak m_K)$ 
\begin{equation}
-\dimtot(\phi_u(j^\ast\mathcal F_0,f))= (CC(\mathcal F_0),df)_{T^\ast U,u},
\end{equation}
 by \eqref{eq:milnorforladic:3} we get
\begin{equation}
-\dimtot\phi_u(j^\ast\mathcal F,f)= (CC(\mathcal F_0),df)_{T^\ast U,u}\overset{\eqref{eq:def:ccforladic2}}{=}(CC(\mathcal F),df)_{T^\ast U,u}.
\end{equation}
This finishes the proof.
\end{proof}

\subsection{}\label{sub:dimtotphi}
Let $Y$ be a smooth scheme and $X$ a smooth curve over a perfect field $k$ of characteristic $p$.
Let $f\colon Y\to X$ be a morphism over $k$ and $\mathcal F\in D_c^b(Y,\bQl)$. We choose a finite extension $E$ of $\mathbb Q_\ell$ and an object  $\mathcal (\mathcal F_n)_{n\geq 0}\in D_c^b(Y,\mathcal O_E)$ satisfying \eqref{eq:repforladic}   such that $\mathcal F=(\mathcal F_n)$ in $D_c^b(Y,\bQl)$ (cf. \cite{Weil2}).
Then for any closed point $u\in Y$, we have
\begin{equation}\label{eq:dimtotophireduction}
\dimtot \phi_u(\mathcal F,f)=\dimtot \phi_u(\mathcal F_n,f) \qquad\forall n\geq 0.
\end{equation} 
The above result follows from the proof of the Grothendieck-Ogg-Shafarevich formula \cite[Expos\'e X, Th\'eor\`eme 7.1]{SGA5}. See also \cite[Corollary 10.2.7]{Leifu}.
Indeed, let $ K=\phi_u(\mathcal F,f)$ and $  K_n=\phi_u(\mathcal F_n,f)$. By \cite[Proposition 10.1.17 (iv)]{Leifu}, $( K_n)_{n\geq 0}$ is a projective system satisfying
\begin{align}
K_{n+1}\otimes^L_{\mathcal O_E/\mathfrak m^{n+2}}\mathcal O_E/\mathfrak m^{n+1}\simeq K_{n}
\end{align}
and $K_n$ is quasi-isomorphic to a bounded complex  of flat
$\mathcal O_E/\mathfrak m^{n+1}$-modules for all $n\geq 0$.
Then by \cite[Lemma 10.1.14]{Leifu},
we may represent  each $K_n$ by a bounded complex $L_n$ of free $\mathcal O_E/\mathfrak m^{n+1}$-modules of finite rank, and we have quasi-isomorphisms
\begin{equation}
L_{n+1}\otimes_{\mathcal O_E/\mathfrak m^{n+2}}\mathcal O_E/\mathfrak m^{n+1}\simeq L_n.
\end{equation}
By \cite[Proposition 10.1.15]{Leifu}, we may find a bounded complex $L$ of free $\mathcal O_E$-modules of finite ranks and quasi-isomorphisms $L/\pi_E^{n+1}L\to L_n$ such that the diagrams  ($n\geq 0$)
\begin{equation}
\begin{gathered}
\xymatrix{
L/\pi_E^{n+2}L\ar[d]\ar[r]&L_{n+1}\ar[d]\\
L/\pi_E^{n+1}L\ar[r]&L_n
}
\end{gathered}
\end{equation}
commute up to homotopy. 
Recall that for any local ring $A$ and  any bounded complex $G$ of projective $A$-modules, we have (applying the identity function to \cite[Rapport, (4.3.1)]{SGA4h}, see also \cite[Proposition 10.2.3]{Leifu})
\begin{align}
{\rm rank}G=\sum_{i}(-1)^i\cdot{\rm rank}G^i, \qquad[G]=\sum_{i}(-1)^i\cdot [G^i]  \quad{\rm in} ~K(D^b_{\rm perf}(\Lambda))\simeq K(\Lambda).
\end{align}
Then we have
\begin{align}
{\rm rank}K=\sum_i(-1)^i\cdot{\rm rank} (L^i\otimes_{\mathcal O_E}E)=\sum_{i}(-1)^i\cdot{\rm rank}(L^i/\pi_E^{n+1}L^i)={\rm rank} K_n.
\end{align}
Since $\dimtot=\Sw+\rank$ and the Swan conductor of a module is determined by its reduction(cf. \cite[1.10]{Katz88}  for a  nice explanation of this fact), we have
\begin{align}
{\rm dimtot}K=\sum_i(-1)^i\cdot{\rm dimtot} (L^i\otimes_{\mathcal O_E}E)=\sum_{i}(-1)^i\cdot{\rm dimtot}(L^i/\pi_E^{n+1}L^i)={\rm dimtot} K_n.
\end{align}
This finishes the proof of \eqref{eq:dimtotophireduction}.

\subsection{}\label{sub:FiniteAssum}
Let $k$ be a field of characteristic $p>0$ and $\bar k$ an algebraic closure of $k$.
We assume that $k$ satisfies the following condition:
\begin{itemize}
\item[(F)] For any finite extension $k^\prime\subseteq \bar k$ of $k$, the groups $H^i({\rm Gal}(\bar k/k^\prime),\mathbb Z/\ell\mathbb Z)$
are finite for all integers $i\geq 0$.
\end{itemize}
Then for any separated scheme $X$ of finite type over $k$, $D_c^b(X,\bQl)$ is a triangulated category with a $t$-structure (cf. \cite[2.2.15 and 2.2.16]{BBD}). Note that finite fields and algebraically closed fields satisfy the  assumption (F).

Similar to the definition of $D_c^b(X,\bQl)$, 
the triangulated  category $D_c^b(X,\overline{\mathbb F}_\ell)$ (respectively $D_c^b(X,\overline{\mathbb Z}_\ell)$) is defined as the direct limit 
\begin{align}\label{eq:dcbdef}
D_c^b(X,\overline{\mathbb F}_\ell)=``2{\rm-}\varinjlim" D_c^b(X,\Lambda)\qquad{\rm respectively}~
D_c^b(X,\overline{\mathbb Z}_\ell)=``2{\rm-}\varinjlim" D_c^b(X, \mathcal O_E)
\end{align}
where $\Lambda\subseteq\overline{\mathbb F}_\ell$ (respectively $E\subseteq \bQl$) ranges over all finite extensions of $\mathbb F_\ell$ (respectively $\mathbb Q_\ell$).  
For a finite extension $\Lambda\to \Lambda^\prime$ of finite fields over $\mathbb F_\ell$, the transition map $D_c^b(X,\Lambda)\to D_c^b(X,\Lambda^\prime)$ in  \eqref{eq:dcbdef}  is defined by $\mathcal F\mapsto \mathcal F\otimes_\Lambda\Lambda^\prime$. For the case $D_c^b(X, \mathcal O_E)$, we refer to \cite[p.331]{KW}.
For an object $\mathcal F$ in $D_c^b(X,\Lambda)$ (respectively $D_c^b(X,\mathcal  O_E)$),
we write $\mathcal F\otimes \bFl$ (respectively $\mathcal F\otimes \bZl$) when viewed as an object in $D_c^b(X,\overline{\mathbb F}_\ell)$ (respectively $D_c^b(X,\overline{\mathbb Z}_\ell)$).
Furthermore, 
\begin{align}
&Hom(\mathcal F\otimes\bFl,\mathcal G\otimes\bFl)=Hom(\mathcal F,\mathcal G)\otimes\bFl \\
&\nonumber\qquad\qquad{\rm respectively}~
Hom(\mathcal F\otimes\bZl,\mathcal G\otimes\bZl)=Hom(\mathcal F,\mathcal G)\otimes\bZl.
\end{align}
For $\bullet\in\{\bQl,\bZl,\bFl\}$,
we denote by $K(X,\bullet)$  the Grothendieck group of $D_c^b(X,\bullet)$ (cf. \cite[Expos\'e VIII]{SGA5} and \cite[Expos\'e IV]{SGA6}). Elements of $K(X,\bullet)$ will be called virtual $\bullet$-sheaves. Since there is an essentially surjective and exact functor $D_c^b(X,\bZl)\to D_c^b(X,\bQl)$, it induces a surjective map
\begin{equation}
K(X,\bZl)\to K(X,\bQl).
\end{equation}
\subsection{}\label{sub:Ggroupofbqlsheaves} 
For $\bullet\in\{\bQl,\bZl,\bFl\}$,
the group $K(X,\bullet)$  equals to the Grothendieck group of the abelian category of constructible $\bullet$-sheaves and  also equals to that of the abelian category 
${\rm Perv}(X,\bullet)$ of perverse $\bullet$-sheaves \cite[2.2.16-2.2.18]{BBD}. For any $K\in D_c^b(X,\bullet)$, we have
\[
[K]=\sum_{i}(-1)^i\cdot [\mathcal H^i(K)]=\sum_{j}(-1)^j\cdot [~^p\mathcal H^j(K)]\qquad{\rm in}~K(X,\bullet).
\]
The operation $\otimes^L$ induces a product ``$\cdot$'' on $K(X,\bullet)$. The functors $Rf_\ast, f^\ast, Rf_!$ and $ Rf^!$ induce  group homomorphisms between corresponding Grothendieck groups, which will also  be denoted by $Rf_\ast, f^\ast, Rf_!$ and $Rf^!$ respectively.

We briefly recall the definition of constructible $\bullet$-sheaves for $\bullet\in \{\bQl,\bZl\}$.
Let $E$ be a finite extension of $\mathbb Q_\ell$.
Recall that a constructible $\pi_E$-adic sheaf $\mathcal F$ is a projective system $\mathcal F=(\mathcal F_n)_{n\geq  0}$ with the following conditions:
\begin{itemize}
\item[(a)] For any $n\geq 0$, $\mathcal F_n$ is a constructible sheaf of $\mathcal O_E/\mathfrak m^{n+1}$-modules.
\item[(b)] The transition morphism  $\mathcal F_{n+1}\to \mathcal F_n$ induces an isomorphism 
\begin{equation}
\mathcal F_{n+1}\otimes_{\mathcal O_E/\mathfrak m^{n+1}}\mathcal O_E/\mathfrak m^{n}\xrightarrow{~\simeq~}\mathcal F_n, \qquad{\rm for~any}~n\geq0.
\end{equation}
\end{itemize}
A constructible $\pi_E$-adic sheaf $\mathcal F$ is called torsion-free if the map $ \mathcal F\xrightarrow{\pi_E} \mathcal F$ is A-R injective (cf. \cite[Corollary 10.1.9]{Leifu}). This implies that  the stalks of the sheaves $\mathcal F_n$ are flat $\mathcal O_E/\mathfrak m^{n+1}$-modules and the stalks $\mathcal F_{\bar x}$ of $\mathcal F$ are finitely generated torsion-free $\mathcal O_E$-modules.
Any constructible $\bQl$-sheaf is isomorphic to $\mathcal F\otimes\bQl$ for some torsion-free constructible $\pi_E$-adic sheaves $\mathcal F$ (cf. \cite[p. 328]{KW}). 

Let $E$ be a finite extension  of $\bQl$.
Any torsion-free constructible $\pi_E$-adic sheaf $\mathcal F=(\mathcal F_n)_{n\geq 0}$ can be viewed as an element in $D_c^b(X,\mathcal O_E)$. Indeed,
since  $\mathcal F_n$ is $\mathcal O_E/\mathfrak m_E^{n+1}$ flat, $\mathcal F_n$ can be considered as a complex
\begin{align*}
\to 0\to\cdots\to 0\to\mathcal F_n\to 0\to\cdots\to
\end{align*}
concentrated in degree 0. The projective system associated to these complexes is contained in $D_c^b(X,\mathcal O_E)$. But this construction fails if $\mathcal F$ is not torsion-free (cf. \cite[II.6, Lemma 6.2]{KW}).
\subsection{}\label{sub:reductionsheaf}
Let $E$ be a finite extension  of $\mathbb Q_\ell$,  the functor $-\otimes^L_{\mathcal O_E}\mathcal O_E/\mathfrak m_E$ defines a group homomorphism
\begin{align}\label{eq:cdethm:00}
 K(X,\mathcal O_E)\to K(X,\mathcal O_E/\mathfrak m_E)
\end{align}
The above group homomorphism is compatible with field extensions and it induces a group homomorphism (called the reduction map)
\begin{align}\label{eq:cdethm:01}
 K(X,\bZl)\to K(X,\bFl).
\end{align}
For any $\mathcal F\in  K(X,\bZl)$, we denote by $\overline{\mathcal F}$ the image of $\mathcal F$ under the reduction map \eqref{eq:cdethm:01}.
By (\cite[Theorem 6.3 iii]{Eke90} and \cite[1.1.2 c]{Weil2}), the operations $Rf_\ast$,$f^\ast$, $Rf_!$ and $Rf^!$  commute with reduction, i.e., for any morphism $f\colon X\to Y$ between separated schemes of finite type over $k$, we have commutative diagrams 
\begin{align}
\begin{gathered}\label{eq:cdethm:1}
\xymatrix{
K(X,\bZl)\ar[r] \ar[d]_-{Rf_\ast}&K(X,\bFl)\ar[d]^-{Rf_\ast} & K(X,\bZl)\ar[r] \ar[d]_-{Rf_!}&K(X,\bFl)\ar[d]^-{Rf_!}\\
K(Y,\bZl)\ar[r]& K(Y,\bFl)&K(Y,\bZl)\ar[r]& K(Y,\bFl)\\
K(Y,\bZl)\ar[r] \ar[d]_-{f^\ast}&K(Y,\bFl)\ar[d]^-{f^\ast}&K(Y,\bZl)\ar[r] \ar[d]_-{Rf^!}&K(Y,\bFl)\ar[d]^-{Rf^!}\\
K(X,\bZl)\ar[r]& K(X,\bFl)&K(X,\bZl)\ar[r]& K(X,\bFl)
}
\end{gathered}
\end{align}

\subsection{} Let $X$ be a smooth scheme purely of dimension $d$ over a perfect field $k$ satisfying the assumption (F) of Subsection \ref{sub:FiniteAssum}. 
By the additivity of characteristic cycle \cite[Lemma  5.13]{Saito}, the map $\mathcal F\mapsto CC\mathcal F$ defines a  group homomorphism
$CC\colon K(X,\bFl)\to Z_d(T^\ast X)$, where $Z_d(T^\ast X)$ is the group of $d$-cycles on $T^\ast X$.
By \eqref{eq:def:ccforladic2},  we also have a well-defined group homomorphism $CC\colon K(X,\bZl)\to Z_d(T^\ast X)$ with a commutative diagram 
\begin{align}
\begin{gathered}
\xymatrix{
K(X,\bZl)\ar[rd]_-{\rm CC}\ar[r]& K(X,\bFl)\ar[d]^-{\rm CC}\\
&Z_d(T^\ast X)
}
\end{gathered}
\end{align}

\subsection{}\label{sub:defidecomp}
Let $X$ be a separated scheme of finite type  over a field $k$ satisfying the condition (F)  of Subsection \ref{sub:FiniteAssum}.
Following a suggestion of  T. Saito, we construct a group homomorphism
\begin{align}d_X\colon K(X,\bQl)\to K(X,\bFl),
\end{align}
which is similar to the decomposition homomorphism  defined in \cite[15.2]{Serre97} (cf. Subsection \ref{sub:serredecomporecall} for a briefly recall of the definition).
By \ref{sub:Ggroupofbqlsheaves}, 
the group $K(X,\bQl)$   equals to the Grothendieck group of the abelian category of constructible $\bQl$-sheaves.
So we only need to define $d_X$ for  constructible $\bQl$-adic sheaves. Now let  $\mathcal F$ be a constructible $\bQl$-adic sheaf. 
We choose a finite extension  $E$ of $\mathbb Q_\ell$ and a (torsion-free) $\pi_E$-adic sheaf $\mathcal G$ on $X$ such that $\mathcal F\simeq \mathcal G\otimes_{\mathcal O_E}\bQl$ (cf. \cite[p. 328]{KW}). Then we define $d_X(\mathcal F)\coloneqq\overline{\mathcal G}$ (cf. \ref{sub:reductionsheaf}), which is independent of the choice  of $E$ and $\mathcal G$ by  Lemma \ref{lem:welldefineddecom} below. By linearity, this defines a group homomorphism $d_X\colon K(X,\bQl)\to K(X,\bFl)$, which will be called the decomposition homomorphism on $X$.

In general, for $\mathcal F\in D_c^b(X,\bQl)$, 
we choose  a finite extension $E$ of $\mathbb Q_\ell$  and $\mathcal G\in D_c^b(X,\mathcal O_E)$ such that $\mathcal F=\mathcal G\otimes\bQl$. We write $\mathcal G=\sum_{i}a_i\cdot \mathcal G_i\in K(X,\mathcal O_E)  (a_i\in\mathbb Z)$ with $\mathcal G_i$   constructible $\pi_E$-adic sheaves.
Then we have
\begin{align}\label{eq:defofdecomphomoofschemes}
d_X(\mathcal F)=\sum_{i}a_i\cdot \overline{\mathcal G_i}=\sum_{i}a_i\cdot [\mathcal G_i\otimes_{\mathcal O_E}^L\mathcal O_E/\mathfrak m_E]= \mathcal G\otimes_{\mathcal O_E}^L\mathcal O_E/\mathfrak m_E=\overline{\mathcal G}\in K(X,\bFl).
\end{align}
\begin{lemma}\label{lem:welldefineddecom}
Let $E$ and $E^\prime$ be  finite extensions of $\mathbb Q_\ell$.  Let $\mathcal F$ (respectively $\mathcal G$) be   a constructible $\pi_E$-adic (respectively $\pi_{E^\prime}$-adic)
 sheaf on $X$. If $\mathcal F\otimes_{\mathcal O_{E}}\bQl\simeq \mathcal G\otimes_{\mathcal O_{E^\prime}}\bQl$ as $\bQl$-sheaves, then we have
\begin{equation}\label{eq:serrepf0-1}
\overline{\mathcal F}=\overline{\mathcal G}\qquad{\rm in}~K(X,\bFl).
\end{equation}
\end{lemma}
\begin{proof}Let $E^{\prime\prime}$ be a finite extension of $\mathbb Q_\ell$ such that $E\subseteq E^{\prime\prime}$ and 
$E^{\prime}\subseteq E^{\prime\prime}$. 
Let $\mathcal F^\prime$ and $\mathcal G^\prime$ be the images in $D_c^b(X,\mathcal O_{E^{\prime\prime}})$ of $\mathcal F$ and $\mathcal G$  respectively. 
Then
$\overline{\mathcal F}=\overline{\mathcal G}\in K(X,\bFl)$ if and only if $\overline{\mathcal F^\prime}=\overline{\mathcal G^\prime}\in K(X,\bFl)$.
Thus we may assume that $E=E^\prime$.  Let $\mathcal O_E$ be the ring of integers of $E$ and let $k_E$ be the residue field of $\mathcal O_E$.

By devissage, we can reduce  to the case where $\mathcal F=j_!\mathcal K$ and $\mathcal G=j_!\mathcal L$ for a locally closed immersion $j\colon U\to X$ and smooth $\pi_E$-adic sheaves $\mathcal K$ and $\mathcal L$ on $U$.  We may assume $U$ is connected.
Let $x\in U$ be a closed point.
By (\cite[Expos\'e VI, Proposition 1.2.5]{SGA5} or \cite[Proposition 10.1.23]{Leifu}), the fiber functor $\mathcal H\mapsto \mathcal H_{\bar x}$ defines an equivalence between the category of smooth  $\pi_E$-adic sheaves  on $U$ and the category of finitely generated $\mathcal O_E$-modules with continuous $\pi_1(U,\bar x)$-actions. 

Now we use a similar  argument with \cite[Theorem 32]{Serre97}. First, we assume that $\mathcal K_{\bar x}$ and $\mathcal L_{\bar x}$ are torsion-free $\mathcal O_E$-modules.
Since $\mathcal K_{\bar x}\otimes\bQl\simeq \mathcal L_{\bar x}\otimes\bQl$, replacing $\mathcal K_{\bar x}$ by a scalar multiple (which does not effect the reduction) and extending the field $E$ if necessarly, 
we can assume that 
$\mathcal K_{\bar x}$ is contained in $\mathcal L_{\bar x}$. Thus there exists an integer $n\geq 0$ such that
\begin{equation}\label{eq:serrepf0}
\pi_E^n \mathcal L_{\bar x}\subseteq \mathcal K_{\bar x}\subseteq \mathcal L_{\bar x}.
\end{equation}
Now we prove \eqref{eq:serrepf0-1} by induction on $n$.

If $n=1$. Let $M$ be the $k_E$-module $\mathcal L_{\bar x}/\mathcal K_{\bar x}$ with continuous  $\pi_1(U,\bar x)$-actions. 
By \cite[Example 3.1.7, p.68]{weibel}, we have ${\rm Tor}_1^{\mathcal O_E}(k_E,M)=M$ and  ${\rm Tor}_1^{\mathcal O_E}(k_E,\mathcal L_{\bar x})=0$ since $\mathcal L_{\bar x}$ is a torsion-free $\mathcal O_E$-module. So the short exact sequence $0\to \mathcal K_{\bar x}\to \mathcal L_{\bar x}\to M\to 0$ induces an exact sequence
\begin{equation}
0\to M\to \mathcal K_{\bar x}\otimes_{\mathcal O_E}k_E\to \mathcal L_{\bar x}\otimes_{\mathcal O_E}k_E\to M\to 0.
\end{equation}
Let $\mathcal M$ be the locally constant and constructible sheaf of $k_E$-modules on $U$ which corresponds to the representation $M$ of $\pi_1(U)$.
Then we get an exact sequence of smooth sheaves of $k_E$-modules on $U$:
\begin{equation}\label{eq:serrepf1}
0\to \mathcal M\to \mathcal K\otimes_{\mathcal O_E}k_E\to \mathcal L\otimes_{\mathcal O_E}k_E\to \mathcal M\to 0.
\end{equation}
Passing to $K(X,\bFl)$ we have
\begin{equation}
{j_!\mathcal M}-\overline{j_! \mathcal K}+ \overline{j_!\mathcal L}-{j_!\mathcal M}=0.
\end{equation}
Thus $\overline{j_!\mathcal K}=\overline{j_!\mathcal L}$ in $K(X,\bFl)$ which proves the result in the case $n=1$.

For $n\geq 2$, we put $T=\pi_E^{n-1}\cdot \mathcal L_{\bar x}+\mathcal K_{\bar x}$. By \eqref{eq:serrepf0}, we then have 
\begin{equation}
\pi_E^{n-1}\cdot \mathcal L_{\bar x}\subseteq T\subseteq \mathcal L_{\bar x} \qquad{\rm and}\qquad 
\pi_E\cdot T\subseteq \mathcal K_{\bar x}\subseteq T.
\end{equation}
Let $\mathcal T$ be the  smooth and torsion-free $\pi_E$-adic sheaf on $U$ corresponding to $T$.
By the case $n=1$ and  induction, we get
\begin{equation}
\overline{j_!\mathcal L}=\overline{j_!\mathcal T}=\overline{j_!\mathcal K}.
\end{equation}

If $\mathcal K_{\bar x}$ or $\mathcal L_{\bar x}$ is not torsion-free, without loss of generality, we assume  $\mathcal K_{\bar x}$ is not torsion-free. Then we write $\mathcal K_{\bar x}=F\oplus T$, where $F$ is a  torsion-free $\mathcal O_E$-module with continuous $\pi_1(U,\bar x)$-action and $T$ is a torsion $\mathcal O_E$-module such that $\pi_E^n T=0$ for some integer $n> 0$. Then we only need to show that $T\otimes^L_{\mathcal O_E}k_E=0$ in $K(\pi_1(U,\bar x), k_E)$.

Let $P$ be a torsion-free and finitely generated $\mathcal O_E$-module with continuous $\pi_1(U,\bar x)$-action such that
we have a surjection $P\to T$. Let $Q$ be the kernel of $P\to T$. Then $Q$ is also torsion-free. Since $T=P/Q$ is killed by $\pi_E^n$, we have 
\begin{equation}
\pi_E^n\cdot P\subseteq Q\subseteq P.
\end{equation}
This condition is the same as \eqref{eq:serrepf0}.
By the same proof above, we have $[P\otimes_{\mathcal O_E}k_E]=[Q\otimes_{\mathcal O_E}k_E]$ in $K(\pi_1(U,\bar x), k_E)$, thus $[T\otimes^L_{\mathcal O_E}k_E]=[P\otimes_{\mathcal O_E}k_E]-[Q\otimes_{\mathcal O_E}k_E]=0$ in $K(\pi_1(U,\bar x), k_E)$. 
We finish the proof.
\end{proof}

\begin{theorem}\label{thm:cdeforetalesheaf}
Let $k$ be a perfect field  satisfying the condition (F)  of Subsection \ref{sub:FiniteAssum}.

$(1)$ 
For any  morphism $f\colon X\to Y$ between separated schemes of finite type over  $k$, we have  commutative diagrams
\begin{align}
\begin{gathered}\label{eq:cdethm:1-1}
\xymatrix{
K(X,\bQl)\ar[r]^-{d_X} \ar[d]_-{Rf_\ast}&K(X,\bFl)\ar[d]^-{Rf_\ast} & K(X,\bQl)\ar[r]^-{d_X} \ar[d]_-{Rf_!}&K(X,\bFl)\ar[d]^-{Rf_!}\\
K(Y,\bQl)\ar[r]^-{d_Y}& K(Y,\bFl)&K(Y,\bQl)\ar[r]^-{d_Y}& K(Y,\bFl)\\
K(Y,\bQl)\ar[r]^-{d_Y} \ar[d]_-{f^\ast}&K(Y,\bFl)\ar[d]^-{f^\ast}&K(Y,\bQl)\ar[r]^-{d_Y} \ar[d]_-{Rf^!}&K(Y,\bFl)\ar[d]^-{Rf^!}\\
K(X,\bQl)\ar[r]^-{d_X}& K(X,\bFl)&K(X,\bQl)\ar[r]^-{d_X}& K(X,\bFl)
}
\end{gathered}
\end{align}

$(2)$ If $X$ is a smooth scheme over  $k$ and if $\mathcal F\in D_c^b(X,\bQl)$, then $\mathcal F$ is micro-supported on $SS(d_X(\mathcal F))$.

$(3)$  If $X$ is a smooth scheme purely of dimension $d$ over $k$ and if $\mathcal F\in D_c^b(X,\bQl)$, then $CC(d_X(\mathcal F))$
satisfies the Minor formula \eqref{eqYZ:milnor} for $\mathcal F$.
\end{theorem}
\begin{proof}
(1) 
follows from the fact that the operations $Rf_\ast$, $f^\ast$, $Rf_!$ and $Rf^!$  commute with reduction (cf.  \cite[Theorem 6.3 iii]{Eke90} and \cite[1.1.2 c]{Weil2}). Indeed, for $\mathcal G\in D_c^b(X,\bQl)$, 
we choose  a finite extension $E$ of $\mathbb Q_\ell$  and $\mathcal F=(\mathcal F_n)_{n\geq 0}\in D_c^b(X,\mathcal O_E)$ such that $\mathcal G=\mathcal F\otimes\bQl$.  By \cite[Theorem 6.3]{Eke90}, we have $Rf_\ast\mathcal F\in D_c^b(Y,\mathcal O_E)$ and $(Rf_\ast\mathcal F)\otimes^L_{\mathcal O_E}\mathcal O_E/\mathfrak m_E\simeq Rf_\ast(\mathcal F\otimes^L_{\mathcal O_E}\mathcal O_E/\mathfrak m_E)$. Thus
\begin{align}
d_Y(Rf_\ast(\mathcal G))\overset{\eqref{eq:defofdecomphomoofschemes}}{=}(Rf_\ast\mathcal F)\otimes^L_{\mathcal O_E}\mathcal O_E/\mathfrak m_E= Rf_\ast(\mathcal F\otimes^L_{\mathcal O_E}\mathcal O_E/\mathfrak m_E)\overset{\eqref{eq:defofdecomphomoofschemes}}{=}Rf_\ast (d_X(G)).
\end{align}
This proves the commutativity for $Rf_\ast$. Other diagrams can be checked similarly.

(2) follows from Proposition \ref{thm:milnorforladic}, and (3) follows from \eqref{eq:dimtotophireduction} and Proposition \ref{thm:milnorforladic}.
\end{proof}

\begin{definition}\label{def:ccforbql}
Let $k$ be a perfect field of characteristic $p$ satisfying the assumption (F) of Subsection \ref{sub:FiniteAssum}.
Let $X$ be a smooth scheme purely of dimension $d$ over $k$ and $\mathcal F\in D_c^b(X,\bQl)$. Then we define the characteristic cycle $CC(\mathcal F)$ $($respectively fake singular support $SS(\mathcal F))$ of $\mathcal F$ to be
\begin{equation}\label{eq:def:ccforbql}
CC(\mathcal F)=CC(d_X(\mathcal F)), \qquad{\rm(respectively} ~SS(\mathcal F)=SS(d_X(\mathcal F)){\rm)}.
\end{equation}
By the additivity of characteristic cycle \cite[Lemma  5.13]{Saito}, the map $\mathcal F\mapsto CC\mathcal F$ defines a  group homomorphism
$CC\colon K(X,\bQl)\to Z_d(T^\ast X)$.
\end{definition}
The characteristic cycle is compatible with pull-back by any properly $SS\F$-transversal morphism.
\begin{lemma}[{\cite[Theorem 7.6]{Saito}}]\label{pull-back}
Let $X$ be a smooth scheme purely of dimension $d$ over a perfect field $k$ of characteristic $p$ and $\F\in D_c^b(X,\Lambda)$.
If $\Lambda=\bQl$, we further assume $k$ satisfies the condition (F) of Subsection \ref{sub:FiniteAssum}.
Let $W$ be a smooth scheme purely of dimension $m$ over $k$.
Let $h\colon W\to X$ be a properly $SS\F$-transversal morphism. 
Then we have
\begin{equation}\label{eqYZ:ccpb}
CCh^*\F=h^!CC\F,
\end{equation}
where $h^!CC\F$ is defined to be $(-1)^{m-d}\cdot h^* CC\F=(-1)^{m-d}\cdot dh_*(\pr_h^!CC\F)$ $($cf. \eqref{eq:pullbackCCdef}$)$ and $\pr_h \colon T^\ast X\times_XW\to T^\ast X$ is the first projection.
\end{lemma}
\begin{proof}
If $\Lambda$ is finite, this is \cite[Theorem 7.6]{Saito}. Now we consider the case where $\Lambda=\bQl$.
We choose  a finite extension $E$ of $\mathbb Q_\ell$  and $\mathcal G=(\mathcal G_n)_{n\geq 0}\in D_c^b(X,\mathcal O_E)$ such that $\mathcal F=\mathcal G\otimes\bQl$. 
Then $SS\F=SS\G_0$ and $h$ is $SS\G_0$-transversal. Thus  we have
\begin{equation}
CCh^\ast\F\overset{\eqref{eq:def:ccforbql}}{=}CCh^*\G_0\overset{(a)}{=}h^!CC\G_0\overset{\eqref{eq:def:ccforbql}}{=}h^!CC\F,
\end{equation}
where (a) follows from  \cite[Theorem 7.6]{Saito}.
\end{proof}

\begin{definition}[{\cite[Definition 5.7]{Saito}}] Let $k$ be a perfect field of characteristic $p$ and $\Lambda$  a finite field of characteristic $\ell(\ell\neq p)$ or $\Lambda=\bQl$.
If $\Lambda=\bQl$, we further assume $k$ satisfies the condition (F) of Subsection \ref{sub:FiniteAssum}.
Let $X$ be a smooth scheme purely of dimension $d$ over  $k$ and $\F\in D_c^b(X,\Lambda)$.
We define the characteristic class of $\F$ by
\begin{equation}\label{eqYZ:ccdef}
cc_X\F=(CC\F, T^*_XX)_{T^*X}\in CH_0(X).
\end{equation}
\end{definition}
If $X$ is moreover projective and $\Lambda$ is finite, then $\chi(X_{\bar k},\F)=\deg(cc_X\F)$ by \cite[Theorem 7.13]{Saito}. The case $\Lambda=\bQl$ can
be deduced from \cite[Theorem 7.13]{Saito} by applying the decomposition homomorphism $d_X\colon K(X,\bQl)\to K(X,\bFl)$, 
or one can prove it by induction on the dimension of $X$ and by using \eqref{eq:fiber2}. 

The following  identity (\ref{eq:fiber2}) for $cc_X\mathcal F$ is key to the proof of our main Theorem \ref{main}.
\begin{proposition}\label{prop:cc2}
Let $X$ be a smooth scheme purely of dimension $d$ over  a perfect field $k$
and  $\mathcal F\in D_c^b(X,\Lambda)$.
If $\Lambda=\bQl$, we further assume $k$ satisfies the condition (F) of Subsection \ref{sub:FiniteAssum}.
Let $f\colon X\to Y$ be a pre-good fibration to a smooth connected curve $Y$ over $k$ with respect to $SS\mathcal F$ and let $u_1,\ldots, u_m\in X$ be the isolated characteristic points of $f$ with respect to $SS\mathcal F$.
Let $\omega$ be a non-zero rational $1$-form on $Y$ which has neither poles nor zeros at $f(u_1),\ldots, f(u_m)$ such that for all closed point $v\in Y$ with ${\rm ord}_v(\omega)\neq 0$, $X_v$ is smooth and $i_v\colon X_v\rightarrow X$ is properly $SS\mathcal F$-transversal.

Then we have an equality in $CH_0(X)$:
\begin{align}\label{eq:fiber2}
cc_X\mathcal F=-\sum_{i=1}^m{\rm dimtot}\phi_{u_i}(\mathcal F,f )\cdot [u_i]-\sum_{v\in |Y|}\ord_v(\omega)\cdot i_{v\ast}(cc_{X_v}(\mathcal F|_{X_v})),
\end{align}
where $i_{v\ast}\colon CH_0(X_v)\rightarrow CH_0(X)$ is the natural homomorphism.
\end{proposition}
\begin{proof}We apply Lemma \ref{cc} to $Z=CC\mathcal F$. For a closed point $v\in |Y|$, if ${\rm ord}_v(\omega)\neq 0$, 
then $i_v\colon X_v\rightarrow X$ is properly $SS\mathcal F$-transversal and is of codimension $1$. Hence by Lemma \ref{pull-back}
we have 
\[
cc_{X_v}(\F\vert_{X_v})=(i_v^!CC\F, T_{X_v}^*X_v)\overset{\eqref{eq:pullbackCCdef-1}}{=}-(i_v^*CC\F,T_{X_v}^*X_v).
\]
By the Milnor formula (\ref{eqYZ:milnor}), $-\dimtot\phi_{u_i}(\F, f)=(CC\F, f^*\omega)_{T^*X, u_i}$.
Now (\ref{eq:fiber2}) follows from (\ref{fiber1}).
\end{proof}
\begin{lemma}\label{lem:blowupsheaf}
Let $X$ and $Y$ be smooth projective connected schemes over a  field $k$ and let $i\colon Y\hookrightarrow X$ be a closed immersion of codimension $r\geq 1$.
Let $\pi\colon \widetilde{X}\to X$ be the blow up of $X$ along $Y$ and $\mathcal F\in D_c^b(X,\Lambda)$. Then we have a distinguished triangle in $D_c^b(X,\Lambda)$:
\begin{equation}\label{eq:blowupsheaf:0}
\mathcal F\to R\pi_\ast\pi^\ast\mathcal F\to \bigoplus_{t=1}^{r-1}(i_\ast i^\ast\mathcal F)(-t)[-2t]\to\mathcal F[1].
\end{equation}
In particular, if $k=\mathbb F_q$ is a finite field, then we have $($cf. \eqref{eq:introep00}$)$
\begin{equation}\label{eq:blowupsheaf:00}
\varepsilon(\widetilde{X},\pi^\ast \mathcal F)=\varepsilon(X,\mathcal F)\cdot \varepsilon(Y,i^\ast\mathcal F)^{r-1}\cdot q^{-\frac{r(r-1)}{2}\chi(Y_{\bar k},i^\ast\mathcal F)} \quad{\rm in}~\Lambda^\times.
\end{equation}
\end{lemma}
\begin{proof}By \cite[Expos\'e XVI, Proposition 2.2.2.1]{ILO14}, we have a distinguished triangle in $D_c^b(X,\Lambda)$:
\begin{equation}\label{eq:blowupsheaf:1}
\Lambda\to R\pi_\ast \Lambda\to \bigoplus_{t=1}^{r-1}(i_\ast \Lambda)(-t)[-2t]\to \Lambda[1],
\end{equation}
which is also valid for $\Lambda=\bQl$.
Applying the functor $-\otimes^L\mathcal F$ to the above distinguished triangle, we get another
distinguished triangle in $D_c^b(X,\Lambda)$:
\begin{equation}\label{eq:blowupsheaf:2}
\mathcal F\to R\pi_\ast \Lambda\otimes^L\mathcal F\to \bigoplus_{t=1}^{r-1}(i_\ast \Lambda)\otimes^L\mathcal F(-t)[-2t]\to \mathcal F[1].
\end{equation}
By the projection formula \cite[Expos\'e XVII, Proposition 5.2.9]{SGA4T3}, we have 
\begin{equation}\label{eq:blowupsheaf:3}
R\pi_\ast \Lambda\otimes^L\mathcal F\simeq R\pi_\ast(\Lambda\otimes^L\pi^\ast\mathcal F)=R\pi_\ast\pi^\ast\mathcal F \qquad {\rm and}\qquad i_\ast\Lambda\otimes^L\mathcal F\simeq i_\ast i^\ast\mathcal F.
\end{equation}
Now \eqref{eq:blowupsheaf:0} follows from \eqref{eq:blowupsheaf:2} and \eqref{eq:blowupsheaf:3}.

Now let $k=\mathbb F_q$ be a finite field.
Since $\varepsilon(X,\mathcal F)=\det(-{\rm Frob}_k, R\Gamma(X_{\bar k};\mathcal F))^{-1}$ is additive with respect to $\mathcal F$, by \eqref{eq:blowupsheaf:0} we have
 \begin{equation}\label{eq:blowupsheaf:4}
 \varepsilon(X,R\pi_\ast\pi^\ast\mathcal F)=\varepsilon(X,\mathcal F)\prod_{t=1}^{r-1}\varepsilon(X,i_\ast i^\ast\mathcal F(-t)[-2t]).
 \end{equation}
Recall that for any $\mathcal K\in D_c^b(X,\Lambda)$ and any integer $n$, we have (cf. \cite[(3.1.1.7)]{Laumon} for $\bQl$-sheaves)
 \begin{align}\label{eq:blowupsheaf:5}
 \varepsilon(X, \mathcal K(n))
 &=\prod_{j}\det(-{\rm Frob}_k;  H^j(X_{\bar k},\mathcal K(n)))^{-1-j}\\
 \nonumber&=\prod_{j}\left(q^{-n\cdot (-1)^{-1-j}\cdot \dim_\Lambda  H^j(X_{\bar k},\mathcal K)}\cdot\det(-{\rm Frob}_k;  H^j(X_{\bar k},\mathcal K))^{-1-j}\right)\\
\nonumber &=q^{n\cdot\chi(X_{\bar k},\mathcal K)}\cdot \varepsilon(X,\mathcal K).
 \end{align}
Since $ \varepsilon(X,R\pi_\ast\pi^\ast\mathcal F)= \varepsilon(\widetilde{X},\pi^\ast\mathcal F)$ and $\varepsilon(X,i_\ast i^\ast\mathcal F(-t)[-2t])=\varepsilon(Y, i^\ast\mathcal F(-t))$, combining \eqref{eq:blowupsheaf:4} and \eqref{eq:blowupsheaf:5}, we get \eqref{eq:blowupsheaf:00}.
\end{proof}
Now, we turn to prove our main theorem.
\begin{theorem}\label{thm:2}
Let $X$ be a smooth connected and projective  scheme  of dimension $d$ over a finite field $k$ of characteristic $p$. 
Let $\Lambda$ be a finite field of characteristic $\ell(\ell\neq p)$ or $\Lambda=\bQl$.
Let $\F\in D_c^b(X,\Lambda)$  and let $\G$ be a smooth sheaf of $\Lambda$-modules on $X$. Let $\det\G\colon CH_0(X)\to\Lambda^\times$ also denote the composition of the reciprocity map $CH_0(X)\to\pi_1(X)^{\rm ab}$ and the character $\pi_1(X)^{\rm ab}\rightarrow \Lambda^\times$ corresponding to the sheaf $\det\G$.
Then we have
\begin{align}\label{epsilon}
\det\G(-cc_X\F)=\frac{\varepsilon(X,\F\otimes\G)}{\qquad\varepsilon(X,\F)^{{\rm rank}\mathcal G}~~}\qquad{\rm in}~\Lambda^\times.
\end{align}
\end{theorem}
\begin{remark}\label{pf:rthpower}
After taking a finite extension $k_r$ of $k$ of degree $r$, the left and right hand sides of the formula \eqref{epsilon} become $r$th power. 
Indeed, since 
$\mathrm{Frob}_{k_r}=(\mathrm{Frob}_{k})^r$, we have
\begin{align*}
\varepsilon (X_{k_r},\mathcal F|_{X_{k_r}})&=\det(-\mathrm{Frob}_{k_r};R\Gamma(X_{\bar{k}},\F))^{-1}=\det(-(\mathrm{Frob}_{k})^r;R\Gamma(X_{\bar{k}},\F))^{-1}\\
&=(-1)^{- \chi(X_{\bar k},\mathcal F)}\cdot \left(\det(\mathrm{Frob}_{k};R\Gamma(X_{\bar{k}},\F))^{-1}\right)^r\\
&=(-1)^{(r-1)\cdot \chi(X_{\bar k},\mathcal F)}\cdot\left(\det(-\mathrm{Frob}_{k};R\Gamma(X_{\bar{k}},\F))^{-1}\right)^r\\
&=(-1)^{(r-1)\cdot \chi(X_{\bar k},\mathcal F)}\cdot\varepsilon (X,\mathcal F)^r. 
\end{align*}
By  \cite[Corollaire 2.10]{Illusie}, we have $\chi(X_{\bar k},\mathcal F\otimes \mathcal G)={\rm rank}\mathcal G\cdot \chi(X_{\bar k},\mathcal F)$ and
\begin{align*}
\frac{\varepsilon(X_{k_r},(\F\otimes\G)|_{X_{k_r}})}{\varepsilon(X_{k_r},\F|_{X_{k_r}})^{{\rm rank}\mathcal G}}=
\frac{(-1)^{(r-1)\cdot \chi(X_{\bar k},\mathcal F\otimes\mathcal G)}\cdot\varepsilon (X,\mathcal F\otimes\mathcal G)^r}{(-1)^{(r-1)\cdot \chi(X_{\bar k},\mathcal F)\cdot {\rm rank}\mathcal G}\cdot\varepsilon (X,\mathcal F)^{{\rm rank}\mathcal G\cdot r}}=\left(\frac{\varepsilon(X,\F\otimes\G)}{\varepsilon(X,\F)^{{\rm rank}\mathcal G}}\right)^r.
\end{align*}
Thus the right hand side of \eqref{epsilon} becomes  $r$th power after base change to $k_r$.
Since the projection $\pi: X_{k_r}\to X$ is finite \'etale, $\pi: X_{k_r}\to X$ is properly $SS(\mathcal F)$-transversal and we have $\pi^\ast(cc_X(\mathcal F))=cc_{X_{k_r}}(\pi^\ast\mathcal F)$ by Lemma \ref{pull-back}. 
By $($\cite[Theorem 1]{Kato_Saito:1983}, or \cite[Lemma 5.1 (1)]{Raskind}$)$, we have a commutative diagram
\begin{align}
\begin{gathered}
\xymatrix{
CH_0(X_{k_r})\ar[d]_-{\pi_\ast}\ar[r]&\pi_1^{\rm ab}(X_{k_r})\ar[d]^-{\pi_\ast}\ar[rrd]^-{{\det}(\pi^\ast\mathcal G)}\\
CH_0(X)\ar[r]&\pi_1^{\rm ab}(X)\ar[rr]^-{{\det}\mathcal G}&&\Lambda^\times
}
\end{gathered}
\end{align}
Then $\det(\pi^\ast\G)(-cc_{X_{k_r}}\pi^\ast\F)=\det\G(-\pi_\ast cc_X\pi^\ast\F)=\det\G(-\pi_\ast \pi^\ast cc_X\F)=\det\G(-r\cdot cc_X\F)=\det\G(- cc_X\F)^r$. Thus the left hand side of \eqref{epsilon} also becomes  $r$th power after base change to $k_r$.
\end{remark}
\begin{proof}[Proof of Theorem \ref{thm:2}]
We prove the theorem by induction on the dimension $d=\dim X$. We denote by $\G_0$ the virtual sheaf $[\G]-\rank\G\cdot [\Lambda]\in K(X,\Lambda)$.

If $d=0$, any constructible sheaf on $X$ corresponds to a continuous representation of $\pi_1(X)$.
Since the rank of $\F\otimes\G_0$ equals to $0$, we have $\det(-\mathrm{Frob};\F\otimes\G_0)=\det(\mathrm{Frob};\F\otimes\G_0)$.
Since the characteristic class $cc_X\F$ of $\F$ equals to the $0$-cycle $\rank\F \cdot [X]\in CH_0(X)$,
we have
\begin{align}
\varepsilon(X,\F\otimes\G_0)&=\frac{\det(\mathrm{Frob};\F)^{\rank\G}}{\det(\mathrm{Frob};\G)^{\rank\F}\det(\mathrm{Frob};\F)^{\rank\G}}\\
\nonumber&=\det(\mathrm{Frob};\G)^{-\rank\F}=\det\G(-cc_X\F).
\end{align}
Hence Theorem \ref{thm:2} is proved for $d=0$.

Suppose that $d\geq1$ and that $X$ has a good fibration $f\colon X\to \mathbb{P}^1$ with respect to $C=SS\F$.
Let $u_1,\cdots,u_m\in X$ be the isolated characteristic points of $f$ with respect to $C$.
Let $\Sigma$ be the finite set consisting of the closed points $f(u_1),\cdots, f(u_m)$ of $\mathbb {P}^1$.
Since the left and right hand sides of the formula \eqref{epsilon} become $r$th power after taking an extension of $k$ of degree $r$ (cf. Remark \ref{pf:rthpower}),
it is enough to show the theorem after taking two extensions of $k$ whose degrees are coprime.
Hence we may assume that there exists a coordinate $t$ of $\mathbb{P}^1$ such that $\infty\notin\Sigma$ and $X_\infty\to X$ is properly $C$-transversal and we fix such a coordinate $t$.
Then the rational $1$-form $\omega=dt$ has neither  zeros nor poles on $\Sigma$.
Let $K$ be the function field of $\mathbb{P}^1$.
Let $dx=\otimes_{v\in\abs{X}}dx_v$ be the Haar measure on $\mathbb A_K$ with values in $\Lambda$ such that $\int_{O_{K_v}}dx=1$ for every place $v$.

Since $\G$ is smooth and $f$ is proper, the generic rank of $Rf_*(\F\otimes\G_0)$ equals to $0$ by \cite[Corollaire 2.10]{Illusie}.
Since the generic rank of $Rf_*(\F\otimes\G_0)$ is zero, the right hand side of (\ref{epsilon}), which equals to $\varepsilon(\mathbb{P}^1, Rf_*(\F\otimes\G_0))$, is the product of local constants by Theorem \ref{product}, i.e.,
\begin{equation}\label{eq:thm2:pro1}
\varepsilon(\mathbb{P}^1, Rf_*(\F\otimes\G_0))=\prod_{v\in\abs{\mathbb{P}^1}}\varepsilon_v(Rf_*(\F\otimes\G_0)|_{\mathbb P^1_{(v)}},\omega),
\end{equation}
where ${\mathbb P^1_{(v)}}$ is the henselization of $\mathbb P^1$ at $v$.
Suppose $v\notin\Sigma$.
Since $f$ is $C$-transversal at the fiber of $v$, the morphism $f$ is universally locally acyclic over $v$ by Definition \ref{Milnor}.
Since $f$ is proper, the  complex $Rf_\ast(\F)$ and $Rf_\ast(\mathcal F\otimes\mathcal G)$ are smooth at $v$ by \cite[Appendice \`a Th. finitude, 2.4]{SGA4h}, i.e., the cohomology groups of $Rf_\ast(\F)$ and $Rf_\ast(\mathcal F\otimes\mathcal G)$ are locally constant on a neighborhood of $v$.
By Lemma \ref{twist} (3), we have
\begin{align}
\varepsilon_v(Rf_*(\F\otimes\G_0)|_{\mathbb P^1_{(v)}},\omega)&=\det(\mathrm{Frob}_v; Rf_*(\F\otimes\G_0)_{\bar{v}})^{\ord_v(\omega)}\\
\nonumber&=\det(-\mathrm{Frob}_v; Rf_*(\F\otimes\G_0)_{\bar{v}})^{\ord_v(\omega)}
\end{align}
because $\rank(Rf_*(\F\otimes\G_0)) =0$.
By the proper base change theorem \cite[Expose XII, Th\'eor\`em 5.1]{SGA4T3}, we have
\begin{align}
\det(-\mathrm{Frob}_v; Rf_*(\F\otimes\G_0)_{\bar{v}})^{\ord_v(\omega)}&=\det(-\mathrm{Frob}_v;R\Gamma(X_{\bar{v}},\F\otimes\G_0))^{\ord_v(\omega)}\\
\nonumber&=\varepsilon(X_{v}, (\F\otimes\G_0)\vert_{X_v})^{-\ord_v(\omega)}.
\end{align}
Hence, for $v\notin\Sigma$, we have 
\begin{equation}\label{eq:thm2:ind}
\varepsilon_v(Rf_*(\F\otimes\G_0)|_{\mathbb P^1_{(v)}},\omega)=\det\G(\ord_v(\omega)\cdot cc(\F\vert_{X_v}))
\end{equation}
by the induction hypothesis for $X_v$ over $k(v)$.

Suppose $v\in\Sigma$. Since $f$
has only one isolated characteristic point $u$ in the fibre $X_v$, we have the following distinguished triangle \cite[Expos\'e XIII, (2.4.6.3)]{SGA7II}
\begin{align}\label{eq:revisedmainproof1}
R\Gamma(X_{\bar{v}},\F\otimes\G_0)\to R\Gamma(X_{\bar{\eta}_v},\F\otimes\G_0)\to \phi_u(\F\otimes\G_0,f)\to R\Gamma(X_{\bar{v}},\F\otimes\G_0)[1],
\end{align}
where $\bar{\eta}_v$ is a geometric generic point of $\mathbb P^1_{(\bar v)}$ of the strict localization of $\mathbb P^1$ at $v$.
Using  Lemma \ref{twist} (4) and  proper base change theorem, we have
\begin{align}
\label{eq:thm2:vin}\varepsilon_v(Rf_*(\F\otimes\G_0)|_{\mathbb P^1_{(v)}},\omega)&=\varepsilon_{0v}(Rf_*(\F\otimes\G_0)_{\bar{v}}, \omega)^{-1}\cdot \varepsilon_{0v}(Rf_*(\F\otimes\G_0)|_{{\eta}_v}, \omega)\\
\nonumber &\overset{\eqref{eq:revisedmainproof1}}{=}\varepsilon_{0v}(\phi_u(\F\otimes\G_0,f),\omega).
\end{align}
Since $\G_0$ is smooth, $f$ is a good fibration and $u$ is $k(v)$-rational, we have 
\begin{equation}\label{eq:vanishingcycleG0}
\phi_u(\F\otimes\G_0,f)=\phi_u(\F,f)\otimes\G_0\vert_u.
\end{equation}
 Here we regard $\G_0\vert_u$ as a virtual unramified representation of ${\rm Gal}(\overline{K}_v/K_v)\to \pi_1(\mathbb{P}^1_{(\bar v)},\bar{v})\cong \pi_1(v)\cong\pi_1(u)$, where $K_v$ is the fractional field of $\mathbb{P}^1_{(\bar v)}$. Indeed, we consider the following distinguished triangle with continuous ${\rm Gal}(\overline{K}_v/K_v)$-action
 \[\mathcal F_{\bar u}\otimes\mathcal G_{\bar u}\to \psi_{u}(\mathcal F\otimes \mathcal G,f)\to\phi_u(\F\otimes\G,f)\to \mathcal F_{\bar u}[1],\]
where $\psi$ is the nearby cycle functor \cite[Expos\'e XIII]{SGA7II}. We only need to show that $\psi_u(\F\otimes\G,f)=\psi_u(\F,f)\otimes\G\vert_u$.
By \cite[Expos\'e XIII, Proposition 2.1.4]{SGA7II}, we have
\begin{align}
\psi_u(\mathcal F\otimes\mathcal G,f)\simeq R\Gamma(X_{(\bar u)}\times_{\mathbb P^1_{(\bar v)}}\bar\eta_v,\mathcal F\otimes\mathcal G ).
\end{align}
Since  the restriction $\mathcal G|_{X_{(\bar u)}}$ equals to the pull-back of $\G\vert_u$ (regarded as a smooth sheaf on $\mathbb{P}^1_{(\bar v)}$) by the map $f_{(u)}\colon X_{(\bar u)}\to \mathbb P^1_{(\bar v)}$, we have 
$\psi_u(\mathcal F\otimes\mathcal G)\simeq \psi_u(\mathcal F)\otimes\mathcal G|_{u}$ by \cite[Expos\'e XVII, (5.2.11.1)]{SGA4T3}.

Since $\ord_v(\omega)=0$, by \eqref{eq:vanishingcycleG0} and Lemma \ref{twist} (2), we have 
\[
\varepsilon_{0v}(\phi_u(\F\otimes\G_0,f),\omega)=\det(\mathrm{Frob}_v; \G|_u)^{\dimtot\phi_u(\F, f)}.
\]
Combining with formula (\ref{eq:thm2:vin}), we proved that for $v\in\Sigma$
\begin{equation}\label{eq:thm2:vin2}
\varepsilon_v(Rf_*(\F\otimes\G_0)|_{\mathbb P^1_{(v)}},\omega)=\det \mathcal G(\dimtot\phi_u(\F, f)\cdot [u]).
\end{equation}
Therefore by (\ref{eq:thm2:pro1}), (\ref{eq:thm2:ind}) and (\ref{eq:thm2:vin2}), we have
\begin{align}
\varepsilon(X,\F\otimes\G_0)
&=\prod_{v\in\abs{\mathbb{P}^1}\setminus\Sigma}\det\G(\ord_v(\omega)\cdot cc(\F\vert_{X_v}))\times\prod_{i=1}^m\det\G(\dimtot\phi_{u_i}(\F, f)\cdot [u_i])\\
\nonumber&=\det\G\left(\sum_{v\in\abs{\mathbb{P}^1}}\ord_v(\omega)\cdot cc(\F\vert_{X_v})+\sum_{i=1}^m\dimtot\phi_{u_i}(\F, f)\cdot [u_i]\right)
\end{align}
Therefore, our theorem follows from Proposition \ref{prop:cc2} when $X$ has a good pencil.

Let $X$ be embedded in a projective space $\mathbb{P}$.
Suppose there exists a line $L\subseteq\mathbb{P}^{\vee}$ such that the fibration $p_L\colon X_L\to L$ defined by $L$ satisfies the conditions (1)--(4) in Lemma \ref{pencil}. We use the notation as in Lemma \ref{pencil}, i.e.,   $\pi_L\colon X_L\to X$ is the blow up of $X$ along $A_L\cap X$ and $i\colon A_L\cap X\to X$ is the closed immersion.
By induction hypothesis, $\varepsilon(X\cap A_L,\F\otimes\G_0)=\det(\G\vert_{X\cap A_L})(-cc(\F\vert_{A_L\cap X}))$.
Since $X_L$ has a good fibration, we have $\varepsilon(X_L,\pi^*_L(\F\otimes\G_0))=\det(\pi_L^*\G)(-cc(\pi^*_L\F))$.
Recall that we have the following commutative diagram by \cite[Theorem 1]{Kato_Saito:1983}
\begin{align}
\begin{gathered}\label{functorialclassfield}
\xymatrix{
 CH_0(X_{L})\ar[d]_-{\pi_{L,\ast}}\ar[r]&\pi_1^{\rm ab}(X_{L})\ar[d]^-{\pi_{L,\ast}}\ar[rrd]^-{{\det}(\pi_L^\ast\mathcal G)}\\
  CH_0(X)\ar[r]&\pi_1^{\rm ab}(X)\ar[rr]^-{{\det}\mathcal G}&&\Lambda^\times\\
 CH_0(X\cap A_L)\ar[u]\ar[r]&\pi_1^{\rm ab}(X\cap A_L)\ar[u]\ar[rru]_-{{\det}(\mathcal G|_{X\cap A_L})}
}
\end{gathered}
\end{align}
Note that $\chi((X\cap A_L)_{\bar k},\mathcal F\otimes\mathcal G_0)=0$ (cf. \cite[Corollaire 2.10]{Illusie}).
Hence by Lemma \ref{lem:blowupsheaf} and \eqref{functorialclassfield}, we have
\begin{align}
\varepsilon(X,\F\otimes\G_0)&\overset{\eqref{eq:blowupsheaf:00}}{=}\varepsilon(X_L,\pi^*_L(\F\otimes\G_0))\cdot \varepsilon(X\cap A_L,\F\otimes\G_0)^{-1}\\
\nonumber&=\det\G(-\pi_{L,*}(cc(\pi_L^*\F))+i_*cc(\F\vert_{X\cap A_L})).
\end{align}
Since $i\colon X\cap A_L\to X$ and $\pi_L: X_L\to X$ are properly $SS\F$-transversal and $i$ is of codimension $2$ ($A_L\to \mathbb P$ is of codimension 2 and  $A_L$ meets $X$ transversally), we have
\begin{equation}\label{eqYZ:pm3}
\pi_{L,*}(cc(\pi_L^*\F))-i_*cc(\F\vert_{X\cap A_L})=\pi_{L, *}(0_{X_L}^!(\pi_L^!CC\F))-i_*(0_{X\cap A_L}^! (i^!CC\F))
\end{equation}
by Lemma \ref{pull-back}, where $0_{X_L}\colon X_L\to T^\ast X_L$ and $0_{X\cap A_L}\colon X\cap A_L\to T^\ast(X\cap A_L)$ are the zero sections. By Lemma \ref{blowup}, the right hand side of the formula (\ref{eqYZ:pm3})  equals to $cc_X\F$.
Hence we have
\[
\varepsilon(X,\F\otimes\G_0)=\det\G(-cc_X\F).
\]

In general, we prove the formula (\ref{epsilon}) after extending the base field $k$ by using Lemma \ref{pencil}.
Let $k_r$ be the extension of degree $r$ over $k$ and let $X_r$ be the base change $X\otimes_kk_r$.
Since the left and right hand sides of the formula (\ref{epsilon})  become $r$th power for $X_r$  (cf. Remark \ref{pf:rthpower}),
it is enough to show that there exist coprime integers $r_1$ and $r_2$ such that the  formula (\ref{epsilon}) is true for $X_{r_1}$ and $X_{r_2}$.
Let $r \neq \ell $ be any integer prime to $\ell$.
The composition of all extensions of degree  $r^e(e\geq 0)$ over $k$ is an infinite field.
Hence by Lemma \ref{pencil}, there exists a good pencil for $X$ over an extension of $k$ of degree $r^{e}$ with $e$ large enough and the formula (\ref{epsilon}) is proved.
\end{proof}

\subsection{}\label{sub:serredecomporecall}
As a corollary of Theorem \ref{thm:2}, we prove the compatibility of characteristic classes with proper push-forward in Corollary \ref{corYZ:2}. 
Before that, let us briefly recall the definition of the decomposition map in \cite[15.2]{Serre97}.
Let $G$ be a finite group. Let $E$ be a finite extension of $\mathbb Q_\ell$, let $\mathcal O_E$ be the ring of integers of $E$ and $\Lambda$ the residue field of $\mathcal O_E$.
Let $K(G,E)$ (respectively $K(G, \Lambda)$) be the 
Grothendieck group of the category of finitely generated $E[G]$-modules
(respectively $\Lambda[G]$-modules). The decomposition homomorphism
\begin{align}\label{eq:sub:serredecomporecall00}
d\colon K(G,E)\to K(G,\Lambda)
\end{align}
is defined as follows:

For any $E$-representation $V$ of $G$, we choose any $G$-invariant $\mathcal O_E$-module $V_0$ such that $V=V_0\otimes_{\mathcal O_E}E$. Then 
\begin{align}
d(V)=V_0\otimes_{\mathcal O_E}\Lambda  \qquad{\rm in}~K(G,\Lambda).
\end{align}
By \cite[Chapter 15, Theorem 32]{Serre97}, $d(V)$ is independent of the choice of $V_0$. By \cite[Chapter 16, Theorem 33]{Serre97},  $d$ is  surjective. 

\begin{corollary}\label{corYZ:2}
Let $f:X\to Y$ be a proper map between smooth projective connected  schemes over {a finite field} $k$ and  $\F\in D_c^b(X,\Lambda)$.
Then we have an equality in $CH_0(Y)$:
\begin{equation}\label{eqYZ:cor2:1}
f_*(cc_X\F)=cc_Y Rf_*\F.
\end{equation}
\end{corollary}

\begin{proof}
First, we prove this corollary when $\Lambda=\overline{\mathbb Q}_\ell$.
Let $\chi$ be a continuous character $\pi_1(Y)^{\rm ab}\to\Lambda^\times$ and $\chi_0$ the virtual representation $\chi-[\Lambda]$.
By the projection formula \cite[Expos\'e XVII, Proposition 5.2.9]{SGA4T3}, we have 
\begin{align}\label{eq:corYZ:2:2}
R\Gamma(Y_{\bar k}, Rf_\ast\mathcal F\otimes\chi_0)&=R\Gamma(Y_{\bar k}, Rf_\ast(\mathcal F\otimes f^\ast\chi_0))\\
\nonumber &=R\Gamma(X_{\bar k},\mathcal F\otimes f^\ast\chi_0)).
\end{align}
By \cite[Theorem 1]{Kato_Saito:1983}, we have a commutative diagram
\begin{align}\label{eq:corYZ:2:3}
\begin{gathered}
\xymatrix{
CH_0(X)\ar[d]_-{f_\ast}\ar[r]&\pi_1^{\rm ab}(X)\ar[d]^-{f_\ast}\ar[rrd]^-{f^\ast\chi}\\
CH_0(Y)\ar[r]&\pi_1^{\rm ab}(Y)\ar[rr]^-{\chi}&&\Lambda^\times
}
\end{gathered}
\end{align}
Now we have
\begin{align}
\chi(-cc_Y Rf_*\F)&\overset{{\rm Thm.} \ref{thm:2}}{=}\det(-\mathrm{Frob}_k;R\Gamma(Y_{\bar{k}}, Rf_*\F\otimes\chi_0))\\
\nonumber&\overset{\eqref{eq:corYZ:2:2}}{=}\det(-\mathrm{Frob}_k;R\Gamma(X_{\bar{k}}, \F\otimes f^*\chi_0))\\
\nonumber&\overset{{\rm Thm.} \ref{thm:2}}{=}(f^\ast\chi)(-cc_X\mathcal F)\\
\nonumber&\overset{\eqref{eq:corYZ:2:3}}{=}\chi(-f_*cc_X\F),
\end{align}
where $cc_Y Rf_*\F$ and $f_*cc_X\F$ also denote their images under the reciprocity map $CH_0(Y)\to\pi_1^{\rm ab}(Y)$ respectively.

The multiplicative group $\Lambda^\times$ contains $\mathbb Q/\mathbb Z$. Since the equality $\chi(-cc_Y Rf_*\F)=\chi(-f_*(cc_X\F))$ holds for all characters of $\pi_1^{\rm ab}(Y)$,
by the injectivity of the reciprocity map $CH_0(Y)\to\pi_1^{\rm ab}(Y)$ in \cite[Theorem 1]{Kato_Saito:1983}, we have $f_*(cc_X\F)=cc_Y Rf_*\F$. This finishes the proof in this case.

When $\Lambda$ is a finite field, this is reduced to the above case in the following way.
By devissage and additivity of  characteristic cycles \cite[Lemma 5.13]{Saito}, it is reduced to the case $\F=j_!\G$ for a locally closed immersion $j\colon U\to X$ and a locally constant and constructible sheaf $\G$ of $\Lambda$-modules on $U$. Then there is a finite \'etale covering  $V\to U$ of Galois group $G$ such that $\mathcal G|_{V}$ is constant. Then $\mathcal G$ corresponds  to a $\Lambda[G]$-module, which we still denote by $\mathcal G$.
Let $E$ be a finite extension of $\mathbb Q_\ell$ such that the residue field  of the ring $\mathcal O_E$ of integers  of $E$ equals to $\Lambda$.
Since the decomposition map $d\colon K(G,E)\to K(G,\Lambda)$ is surjective (cf. \cite[Chapter 16, Theorem 33]{Serre97}),
 we may find an element $\mathcal G^\prime$ in $K(G,E)$ such that $d(\mathcal G^\prime)=\mathcal G$.
We can also view $\mathcal G^\prime$ as an element in $K(U,\overline{\mathbb Q}_\ell)$ and $\mathcal G^\prime$ is a linear combination of the classes of smooth $\overline{\mathbb Q}_\ell$-adic sheaves on $U$.
%
By Definition \ref{def:ccforbql} and \eqref{eq:cdethm:1}, 
we have $CC\F=CCj_!\G'$ and $CC Rf_*\F=CC Rf_*j_!\G'$.
Hence it is reduced to the above case where $\Lambda=\overline{\mathbb Q}_\ell$.
\end{proof}
\begin{corollary}\label{cor:blowupcc}
Let $X$ and $Y$ be smooth projective connected schemes over a finite field $k$ and let $i\colon Y\hookrightarrow X$ be a closed immersion of codimension $r\geq 1$.
Let $\pi\colon \widetilde{X}\to X$ be the blow up of $X$ along $Y$ and $\mathcal F\in D_c^b(X,\Lambda)$.
Then we have an equality in $CH_0(X)$:
\begin{align}\label{eq:blowupcc}
\pi_\ast(cc_{\widetilde{X}}\pi^\ast\mathcal F)=cc_X\mathcal F+ (r-1)\cdot i_\ast(cc_Y(i^\ast\mathcal F)).
\end{align}
\end{corollary}
\begin{proof}
By Lemma \ref{lem:blowupsheaf}, we have
\begin{align}
cc_{{X}}(R\pi_\ast\pi^\ast\mathcal F)=cc_X(\mathcal F)+(r-1)\cdot cc_X(i_\ast i^\ast\mathcal F).
\end{align}
Then \eqref{eq:blowupcc} follows from  Corollary \ref{corYZ:2}.
\end{proof}
\section{Swan class}\label{sec:sc}
\subsection{}Let $X$ be a proper smooth and connected scheme  over a perfect field $k$. Let $j\colon U\rightarrow X$ be an open dense sub-scheme of $X$ and $\mathcal F$ a locally constant and constructible sheaf of $\Lambda$-modules on $U$. 
There are two kinds of Swan classes of $\mathcal F$. 
In \cite{Kato_Saito}, Kato and T. Saito defined the Swan class $\Sw_X^{\rm ks}(\mathcal F)\in {\rm CH}_0(X-U)\otimes_{\mathbb Z}\mathbb Q$ by using logarithmic product and alteration. Later  in \cite[Definition 6.7.3]{Saito}, T. Saito defined another Swan class $\Sw_X^{\rm cc}(\mathcal F) \in {\rm CH}_0(X-U)$ by using characteristic cycle
\begin{equation}\label{eq:sc:cc}
\Sw_X^{\rm cc}(\mathcal F):=(T^\ast_XX, CC(j_!\mathcal F)-{\rm rank}\mathcal F\cdot CC(j_! \Lambda))_{T^\ast X} .
\end{equation}
 Note that $\Sw_X^{\rm cc}(\mathcal F)$ has integral coefficients since $CC(j_!\mathcal F)$ has integral coefficients by \cite[Theorem 5.18]{Saito}.
T.~Saito formulated the following conjecture:
\begin{conjecture}[{\cite[Conjecture 6.8.2]{Saito}}]\label{conj:sc:ss}Let $X$ be a proper smooth and connected scheme  over a perfect field $k$. Let $U$ be an open dense sub-scheme of $X$ and $\mathcal F$ a locally constant  and constructible  sheaf of $\Lambda$-modules on $U$. Then
we have an equality in $CH_0(X-U)$:
\begin{equation}\label{eq:sc:ss}
 \Sw_X^{\rm ks}(\mathcal F)=- \Sw_X^{\rm cc}(\mathcal F).
\end{equation}
\end{conjecture}
\subsection{}In this section, we show in Theorem \ref{thm:sc:surf} that the above equality (\ref{eq:sc:ss}) is true in $CH_0(X)$ assuming the resolution of singularities and the following special case of proper push-forward of characteristic class (cf. \cite[Section 6.2]{Saito}).
\begin{conjecture}\label{conj:sc:pf}
Let $X$ be a fixed  proper smooth and connected scheme over a perfect field $k$.
For any  smooth proper scheme $Y$ over  $k$, any $F\in D_c^b(Y,\Lambda)$, and any  generically finite and surjective morphism $f\colon Y\rightarrow X$,  we have an equality in $CH_0(X)$:
\begin{equation}\label{eq:sc:pf}
cc_X(Rf_\ast \mathcal F)=f_\ast cc_Y(\mathcal F).
\end{equation}
\end{conjecture}
\subsection{}In \cite[Conjecture 1]{Saito17}, T. Saito formulated a very general conjecture on the proper push-forward of characteristic cycle. 
The above Conjecture \ref{conj:sc:pf} can be deduced from his conjecture. By \cite[Lemma 2]{Saito17}, if $f$ is finite on the support of $\mathcal F$, then (\ref{eq:sc:pf}) holds.  
If $k$ is a finite field and if $X$ is a smooth and proper surface, then  (\ref{eq:sc:pf})  is a consequence of Corollary \ref{corYZ:2} since any proper smooth surface is projective. 
\subsection{}
Now, let us  recall the definition of $\Sw_X^{\rm ks}(\mathcal F)$. In the following, we assume that $\Lambda$ is a finite field. The case $\Lambda=\overline{\mathbb Q}_{\ell}$ can be reduced to the finite field case by using \cite[Lemma 4.2.7]{Kato_Saito} and Definition \ref{def:ccforbql}.
Let $X$ be a proper smooth and connected scheme of dimension $d$ over a perfect field $k$ of characteristic $p$. Let $U$ be an open dense sub-scheme of $X$ and $\mathcal F$ a locally constant and   constructible sheaf of $\Lambda$-modules on $U$. Let $f\colon V\rightarrow U$ be a finite \'etale Galois covering of Galois group $G$ trivializing $\mathcal F$. Let $M$ be the ${\Lambda}$-representation of $G$
corresponding to $\mathcal F$. 

By Nagata's compactification theorem \cite{Nagata}, we may find a proper scheme $Y$ which contains $V$ as an open dense subscheme. Replacing $Y$ by the closure of the graph of $f\colon V\to U$ in $X\times Y$, we may assume there is a commutative diagram
\begin{equation}\label{eq:sc:diag00}
\begin{gathered}
\xymatrix{
V\ar[r]\ar[d]&Y\ar[d]\\
U\ar[r]&X.
}
\end{gathered}
\end{equation}
Since $V$ is proper over $U$ and is dense in $U\times_XY$, the above diagram \eqref{eq:sc:diag00} is Cartesian.
By \cite[Lemma 3.2.1]{Kato_Saito}, we can construct a commutative diagram 
\begin{equation}\label{eq:sc:altdig}
\begin{gathered}
\xymatrix{
W\ar@/_2pc/[ddd]_{h}\ar[d]_{g}\ar[rr]^{j^{\prime\prime}}&&Z\ar[d]^{\bar g}\ar[ldd]\ar@/^2pc/[ddd]_{\bar h}\\
V\ar[rr]^{j^\prime}\ar[dd]_f&&Y\ar[dd]^{\bar f}\\
&X^\prime\ar[rd]&\\
U\ar[ru]\ar[rr]^j&&X
}
\end{gathered}
\end{equation}
of schemes over $k$ satisfying the following properties:
\begin{enumerate}
\item $U$ is the complement of a Cartier divisor $B$ of $X^\prime$ and the map $X^\prime\rightarrow X$ is an isomorphism on $U$.
\item $Z$ is smooth over $k$ and $W$ is the complement of a divisor $D$ with simple normal crossings. Let $\{D_i\}_{i\in I}$ be the irreducible components of $D$.
\item The map $\bar g\colon Z\rightarrow Y$ is proper, surjective and generically finite. 
\item The squares are Cartesian.
\end{enumerate}
The log blow up $(Z\times Z)^\prime\rightarrow Z\times Z$ is defined to be the blow up of $Z\times Z$ at $D_i\times D_i$ for all $i\in I$ (cf. \cite[Definition 1.1.1]{Kato_Saito}). For $i\in I$, let $(D_i\times Z)^\prime$ be the proper transform 
of $D_i\times Z$ and $(Z\times D_i)^\prime$ be that of $Z\times D_i$. We put $D^{(1)\prime}=\bigcup_{i\in I}(D_i\times Z)^\prime$ and $D^{(2)\prime}=\bigcup_{i\in I}(Z\times D_i)^\prime$.
We define $(Z\times Z)^\sim=(Z\times Z)^\prime\setminus \{D^{(1)\prime}\cup D^{(2)\prime}\}$, and call it the log self-product of $Z$ with respect to $D$.
By the universal property of blow up, the diagonal map $Z\rightarrow Z\times Z$ induces a map $Z=\Delta_Z^{\rm log}\rightarrow (Z\times Z)^\sim$, which is called the log diagonal map.
Similarly, we can define the log self-product $(X^\prime\times X^\prime)^\sim$ of $X^\prime $ with respect to the Cartier divisor  $B$ (cf. \cite[Definition 1.1.1]{Kato_Saito}). We have a canonical map $(Z\times Z)^\sim\rightarrow (X^\prime\times X^\prime)^\sim$. 
We put
 \[(Z\times_{X^\prime}Z)^\sim=(Z\times Z)^\sim\times_{(X^\prime\times X^\prime)^\sim}\Delta_{X^\prime}^{\rm log}.\]
For $\sigma\in G$, let $\Gamma_\sigma\subseteq V\times_UV$ be the graph of $\sigma$. We consider the Gysin map
\[ CH_d(V\times_UV\setminus \Delta_V)=\bigoplus_{\sigma\in G\setminus\{1\}}\mathbb Z\cdot \Gamma_\sigma
\xrightarrow[]{\quad (g\times g)^!\quad} CH_d(W\times_UW\setminus W\times_VW).\]
Now  the Swan class $\Sw_X^{\rm ks}(\mathcal F)\in {\rm CH}_0(X-U)\otimes_{\mathbb Z}\mathbb Q$ is defined  by (cf. \cite[Definition 4.2.6]{Kato_Saito})
\begin{align}\label{eq:sc:pft:2} 
\Sw_X^{\rm ks}(\mathcal F)=\frac{1}{[W:U]}\bar{h}_\ast\left\{\sum_{\sigma\neq 1}\left(\dim_\Lambda M-\dim_\Lambda M^{\sigma}
+\frac{\dim_\Lambda M^{\sigma^p}/M^\sigma}{p-1}\right)\cdot\left(\widetilde{\Gamma}_\sigma,\Delta_{ Z}^{\rm log}\right)_{(Z\times Z)^\sim}\right\}
\end{align}
where $\widetilde{\Gamma}_\sigma\in CH_d((Z\times_{X^\prime}Z)\setminus W\times_VW)$ is a lifting of $(g\times g)^!\Gamma_\sigma$. If the order of $\sigma$ is not a power of $p$, then $\left(\widetilde{\Gamma}_\sigma,\Delta_{ Z}^{\rm log}\right)_{(Z\times Z)^\sim}=0$ by \cite[Lemma 4.1.2.2]{Kato_Saito}.

\begin{lemma}\label{lem:sc:pfcc}
Let $X$ be a proper smooth and connected scheme  over a perfect field $k$ of characteristic $p$. Let $U$ be an open dense sub-scheme of $X$.
Consider a Cartesian diagram
\begin{align}\label{eq:sc:pfcc:1}
\begin{gathered}
\xymatrix{
V\ar[r]^{j^\prime}\ar[d]_{f}&Y\ar[d]^{\bar f}\\
U\ar[r]^{j}&X
}
\end{gathered}
\end{align}
of smooth and connected schemes over $k$, where $f$ is a finite \'etale morphism, $j$ and $j^\prime$ are open immersions. Let $\mathcal G$ be a locally constant and constructible  sheaf of $\Lambda$-modules on $V$. 
Assume Conjecture \ref{conj:sc:pf} is true for $X$,
then we have the following equalities:
\begin{align}
\label{eq:sc:pft:4}\Sw_X^{\rm ks}(f_\ast\mathcal G)&={\bar f}_\ast \Sw_Y^{\rm ks}(\mathcal G)+{\rm rank}\mathcal G\cdot \Sw_X^{\rm ks}(f_\ast \Lambda)\in CH_0(X- U),\\
\label{eq:sc:pft:5}\Sw_X^{\rm cc}(f_\ast\mathcal G)&={\bar f}_\ast \Sw_Y^{\rm cc}(\mathcal G)+{\rm rank}\mathcal G\cdot \Sw_X^{\rm cc}(f_\ast \Lambda)\in CH_0(X).
\end{align}
If moreover $D=Y\setminus V$ (respectively $B=X\setminus U$) is a simple normal crossing divisor on $X$ (respectively $Y$), then we have
\begin{align}
\label{eq:sc:pft:6}\Sw_X^{\rm ks}(f_\ast \Lambda)&=-\Sw_X^{\rm cc}(f_\ast \Lambda)=d^{\rm log}_{V/U}\in CH_0(X)
\end{align}
where $d^{\rm log}_{V/U}$ is the wild discriminant of $V$ over $U$ (cf. \cite[Definition 4.3.1]{Kato_Saito}).
\end{lemma}
\begin{proof}
(\ref{eq:sc:pft:4}) follows from \cite[Corollary 4.3.4]{Kato_Saito}. (\ref{eq:sc:pft:5}) follows from (\ref{eq:sc:pf}).
By \cite[Proposition 3.4.10]{Kato_Saito}, we have
 \begin{align}
\label{eq:sc:pft:7}\Sw_X^{\rm ks}(f_\ast \Lambda)=(-1)^{d-1}{\bar f}_\ast\left\{c_{d, D}^Y\left(\Omega_{Y/F}^1({\rm log}D)-\bar{f}^\ast\Omega_{X/F}^1({\log}B)\right)\cap [Y]\right\}.
\end{align}
Since $\Sw_X^{\rm cc}(f_\ast\Lambda)=cc_X(j_!f_\ast\Lambda)-[V:U]\cdot cc_X(j_!\Lambda)=cc_X(\bar f_\ast j^\prime_!\Lambda)-[V:U]\cdot cc_X(j_!\Lambda)$, hence by (\ref{eq:sc:pf}) we have
 \begin{align}
\label{eq:sc:pft:8}\Sw_X^{\rm cc}(f_\ast\Lambda)=\bar f_\ast \left(cc_Y(j^\prime_!\Lambda)- \bar f^\ast cc_X(j_!\Lambda)\right).
\end{align}
Since $cc_Y(j^\prime_!\Lambda)=(-1)^d c_d(\Omega^1_{Y/F}({\rm log}D))\cap [Y]$ and $cc_X(j_!\Lambda)=(-1)^d c_d(\Omega^1_{X/F}({\rm log}D))\cap [X]$, by (\ref{eq:sc:pft:7}) and (\ref{eq:sc:pft:8}), we get that $\Sw_X^{\rm ks}(f_\ast \Lambda)=-\Sw_X^{\rm cc}(f_\ast \Lambda)$ as zero classes in $CH_0(X)$.
\end{proof}

\begin{theorem}\label{thm:sc:surf}Let $X$ be a proper smooth and connected scheme  over a perfect field $k$ of characteristic $p$. 
Let $U$ be an open dense sub-scheme of $X$ and $\mathcal F$ a locally constant and constructible sheaf of $\Lambda$-modules on $U$. 
Assume Conjecture \ref{conj:sc:pf} is true for $X$ and  assume the resolution of singularities, then
we have an equality in $CH_0(X)$:
\begin{equation}\label{eq:sc:ss2}
 \Sw_X^{\rm ks}(\mathcal F)=- \Sw_X^{\rm cc}(\mathcal F).
\end{equation}
\end{theorem}
The idea of the following proof is due to T. Saito.
\begin{proof}
 By Lemma \ref{lem:sc:pfcc} and resolution of singularities, 
 we may assume that $B=X- U$ is a simple normal crossing divisor and we can find a Cartesian diagram (\ref{eq:sc:pfcc:1}) 
 such that $\mathcal F$ is trivialized by $V$ and $D=Y- V$ is a simple normal crossing divisor on $Y$.
By Brauer's theorem \cite[Theorem 19]{Serre97}, 
we may assume that $\mathcal F=h_\ast \mathcal G$ where $h\colon U^\prime\rightarrow U$ 
is an intermediate covering of $V\rightarrow U$ such that the Galois group $G$ of $V\rightarrow U^\prime$ is a 
$p$-elementary group and $\mathcal G$ is a smooth sheaf of rank 1 on $U^\prime$. 
By Lemma \ref{lem:sc:pfcc}  and resolution of singularities again, we may assume $\mathcal F$ is of rank 1 such that the Galois group $G$ of $V\rightarrow U$ is a $p$-elementary abelian group. 
Let $\chi\in H^1(U,\mathbb Q/\mathbb Z)$ be the character of $G$ corresponding to $\mathcal F$. Replacing $\chi$ by its $p$-primary part,  it does not change both sides of (\ref{eq:sc:ss2}). 
Then $\chi\in H^1(U,\mathbb Z/p^n\mathbb Z)$ for some non-negative integer  $n$.
Hence we may assume that the character associated to $\mathcal F$ is of order $p^n$ and $G$ is a cyclic group of order $p^n$.

If $n=0$, then $\mathcal F= \Lambda$ is the trivial sheaf. Hence $\Sw_X^{\rm ks}(\mathcal F)=-\Sw_X^{\rm cc}(\mathcal F)=0$.

If $n=1$, $G$ is a cyclic group of order $p$. Then $f_\ast\Lambda=\bigoplus\limits_{\chi}\mathcal F_\chi$ where $\chi$ runs over (1-dimensional) 
irreducible representations of $G$ and $\mathcal F_\chi$ is the smooth sheaf on $U$ corresponding to $\chi$. We have
\begin{align}
\label{eq:sc:surf:4}\Sw_X^{\rm ks}(f_\ast\Lambda)&=\sum_{\chi}\Sw_X^{\rm ks}(\mathcal F_\chi)=\sum_{\chi\neq 1}\Sw_X^{\rm ks}(\mathcal F_\chi),\\
\label{eq:sc:surf:3}\Sw_X^{\rm cc}(f_\ast\Lambda)&=\sum_{\chi}\Sw_X^{\rm cc}(\mathcal F_\chi)=\sum_{\chi\neq 1}\Sw_X^{\rm cc}(\mathcal F_\chi).
\end{align}
For an irreducible representation $\chi$, let $\Lambda_\chi$ be the representation space of $\chi$. 
Then for any two non-trivial (1-dimensional) irreducible representations $\chi$ and $\chi^\prime$ of $G$, they have same order $|G|=p$. 
For any $\sigma\in G$, we have
\begin{align}\label{eq:sc:surf:2}
{\rm dim}~ \Lambda_\chi^{\sigma}={\rm dim}~ \Lambda_{\chi^\prime}^{\sigma}.
\end{align}
Hence by \cite[Definition 5.1]{Saito_Yatagawa}, $\chi$ and $\chi^\prime$ have same wild ramification. 
Thus by \cite[Theorem 0.1]{Saito_Yatagawa}, $\Sw_X^{cc}(\mathcal F_\chi)=\Sw_X^{cc}(\mathcal F_{\chi^\prime})$.
By (\ref{eq:sc:surf:2}) and (\ref{eq:sc:pft:2}), we also have $\Sw_X^{ks}(\mathcal F_\chi)=\Sw_X^{ks}(\mathcal F_{\chi^\prime})$.
Since $\mathcal F$ corresponds to a non-trivial irreducible  representation of $G$,
we get
\begin{align}\label{eq:sc:surf:1}
(p-1) \Sw_X^{\rm ks}(\mathcal F)&=\sum_{\chi\neq 1}\Sw_X^{\rm ks}(\mathcal F_\chi)=\Sw_X^{\rm ks}(f_\ast\Lambda)\\
\label{eq:sc:surf:1-1}&\overset{\rm (\ref{eq:sc:pft:6})}{=}
-\Sw_X^{\rm cc}(f_\ast\Lambda)=-\sum_{\chi\neq 1}\Sw_X^{\rm cc}(\mathcal F)=-(p-1) \Sw_X^{\rm cc}(\mathcal F).
\end{align}
We obtained $ \Sw_X^{\rm ks}(\mathcal F)=- \Sw_X^{\rm cc}(\mathcal F)$.

In general, let $H\subseteq G$ be the unique subgroup of order $p$. Let $\chi$ be the character of $G$ associated to $\mathcal F$ and $\chi_H$  the restriction of $\chi$ to $H$. 
Then $\chi$ is of order $p^n$ and $\chi_H$ is of order $p$. Let $V\xrightarrow{g} U^\prime\xrightarrow{h} U$ be the etale coverings such that the Galois group of $V\xrightarrow{g} U^\prime$ is $H$.
Let $\mathcal F_{\chi_H}$ be the smooth sheaf on $U^\prime$ associated to $\chi_H$. By definition, $\mathcal F_{\chi_H}=h^\ast\mathcal F$. We have
\begin{align}\label{eq:sc:surf:6}
h_\ast \mathcal F_{\chi_H}=h_\ast h^\ast\mathcal F\simeq \mathcal F\otimes h_\ast\Lambda\simeq \mathcal F\otimes\bigoplus_{\psi}\mathcal F_{\psi}\simeq\bigoplus_{\psi}\left(\mathcal F\otimes\mathcal F_{\psi}\right)
\end{align}
where $\psi$ runs over irreducible representations of the cyclic group $G/H$ and $\mathcal F_{\psi}$ is the smooth sheaf on $U$ associated to $\psi$. 
We also regard $\psi$ as a character of $G$. Now the character $\chi\cdot \psi$ associated to $\mathcal F\otimes\mathcal F_{\psi}$ is of order $p^n$ ($\chi$ and $\chi\cdot \psi$ have  same order). Then for any $\sigma\in G$,   we have
\begin{align}\label{eq:sc:surf:zj}
\dim \Lambda_{\chi}^\sigma=\dim \Lambda_{\chi\cdot \psi}^\sigma.
\end{align} 
Hence by \cite[Definition 5.1 and Theorem 0.1]{Saito_Yatagawa}, $\chi$ and $\chi\cdot \psi$ have same wild ramification and $\Sw_X^{cc}(\mathcal F)=\Sw_X^{cc}(\mathcal F\otimes\mathcal F_\psi)$. 
By (\ref{eq:sc:surf:zj}) and (\ref{eq:sc:pft:2}), we also have $\Sw_X^{ks}(\mathcal F)=\Sw_X^{ks}(\mathcal F\otimes\mathcal F_\psi)$.
Thus by 
(\ref{eq:sc:surf:6}) and Lemma \ref{lem:sc:pfcc}, we get
\begin{align}
\label{eq:sc:surf:7}\bar h_\ast \Sw^{ks}(\mathcal F_{\chi_H})+{\rm rank}\mathcal F_{\chi_H}\cdot\Sw_X^{ks}(h_\ast\Lambda)&=\Sw_X^{ks}(h_\ast \mathcal F_{\chi_H})\\
\nonumber&=\sum_{\psi}\Sw_X^{ks}(\mathcal F\otimes\mathcal F_\psi)=p^{n-1}\Sw_X^{ks}(\mathcal F),\\
\label{eq:sc:surf:8}\bar h_\ast \Sw^{cc}(\mathcal F_{\chi_H})+{\rm rank}\mathcal F_{\chi_H}\cdot\Sw_X^{cc}(h_\ast\Lambda)&=\Sw_X^{cc}(h_\ast \mathcal F_{\chi_H})\\
\nonumber&=\sum_{\psi}\Sw_X^{cc}(\mathcal F\otimes\mathcal F_\psi)=p^{n-1}\Sw_X^{cc}(\mathcal F).
\end{align}
By the case $n=1$, $\Sw^{ks}(\mathcal F_{\chi_H})=-\Sw^{cc}(\mathcal F_{\chi_H})$. By Lemma  \ref{lem:sc:pfcc}, $\Sw_X^{ks}(h_\ast\Lambda)=-\Sw_X^{cc}(h_\ast\Lambda)$. Hence from (\ref{eq:sc:surf:7}) and (\ref{eq:sc:surf:8}), we  get that $\Sw^{ks}(\mathcal F)=-\Sw^{cc}(\mathcal F)$ in $CH_0(X)$. This finished the proof.
\end{proof}
\subsection{}
If $k$ is a finite field, then Conjecture \ref{conj:sc:pf} is true for proper smooth surface by Corollary \ref{corYZ:2}. 
Thus we have the following result.
\begin{corollary}\label{cor:2surf}
Let $X$ be a proper smooth and connected surface  over a finite field $k$. 
Let $U$ be an open dense subscheme of $X$ and $\mathcal F$ a locally constant and constructible sheaf of $\Lambda$-modules on $U$. 
Then
we have an equality in $CH_0(X)$:
\begin{equation}
 \Sw_X^{\rm ks}(\mathcal F)=- \Sw_X^{\rm cc}(\mathcal F).
\end{equation}
\end{corollary}


\begin{thebibliography}{0}

\bibitem{Pat17a} T. Abe and D. Patel,  \emph{On a localization formula of epsilon factors via microlocal geometry}, to appear in Annals of $K$-theory.

\bibitem{SGA4T3}
M. Artin, A. Grothendieck, et J.-L. Verdier, \emph{Th\'eorie des topos et cohomologie \'etale des sch\'emas}, S\'eminaire de G\'eom\'etrie Alg\'ebrique du Bois-Marie 1963-64 (SGA 4), Lecture Notes
in Mathematics, vol. 269, 270, 305, Springer-Verlag, Berlin-Heidelberg-New York, 1972-1973.
\bibitem{SGA6}
P. Berthelot, A. Grothendieck, L. Illusie. et al.,\emph{Th\'eorie
des intersections et th\'eor\`eme de Riemann-Roch}, S\'eminaire de G\'eom\'etrie Alg\'ebrique du Bois-Marie 1966-67 (SGA 6), Lecture Notes in Mathematics 225, Springer-Verlag, Berlin-Heidelberg-New York, 1971.

\bibitem{Beilinson}
A. Beilinson, \emph{Constructible sheaves are holonomic}, Selecta Mathematica,
 22, Issue 4 (2016): 1797-1819.
\bibitem{beilinson07}A. Beilinson, \emph{Topological E-factors}, Pure Appl. Math. Q., Vol. 3, No. 1 (2007): 357-39.
\bibitem{BBD}
A. Beilinson, I. N. Bernstein et P. Deligne, \emph{Faisceaux pervers}, dans Analyse et topologie sur les
espaces singuliers (I), Conf\'erence de Luminy, juillet 1981, Ast\'erisque 100 (1982).



\bibitem{DH}
P. Deligne, G. Henniart, \emph{Sur la variation, par torsion, des constantes locales
d'\'equations fonctionnelles de fonctions L}, Invent. Math. 64 (1981): 89-118

\bibitem{Weil2}
P. Deligne, \emph{La conjecture de Weil, II}, Inst. Hautes \'Etudes Sci. Publ. Math. 52 (1980): 137-252.

\bibitem{De2}
P. Deligne, \emph{Les constantes des \'equations fonctionnelles des fonctions $L$},
S\'eminaire \`a l' I.H.E.S., 1980, notes de L. Illusie.

\bibitem{SGA4h}
P. Deligne, \emph{Cohomologie \'etale}, S\'eminaire de G\'eom\'etrie Alg\'ebrique du Bois-Marie SGA 4{\tiny$1/2$}, avec la collaboration de J. F. Boutot, A. Grothendieck, L. Illusie et J.-L. Verdier. Lecture Notes in Mathematics 569, Springer-Verlag, Berlin-Heidelberg-New York, 1977.

\bibitem{SGA7II}P. Deligne et N. Katz, \emph{Groupes de monodromie en g\'eom\'etrie alg\'ebrique}, S\'eminaire de G\'eom\'etrie Alg\'ebrique du Bois-Marie 1967-1969 (SGA 7 II), Lecture Notes in Mathematics 340, Springer-Verlag, Berlin-Heidelberg-New York, 1973.


\bibitem{Deligne}
P. Deligne, \emph{Les constantes des \'equations fonctionnelles des fonctions $L$}, in Modular Functions of One Variable II, Lecture Notes in Mathematics 349, Springer-Verlag, Berlin-Heidelberg-New York, 1972.

\bibitem{Eke90}
T. Ekedahl, \emph{On the adic formalism}, in The Grothendieck Festschrift, Vol. II, volume 87 of Progr. Math., Birkh\"auser Boston, Boston, MA, (1990): 197-218.

\bibitem{Leifu}L. Fu, \emph{Etale cohomology theory}, revised edition. Nankai tracts in Math. 14 (2015).


\bibitem{Fulton}W. Fulton, \emph{Intersection theory}, 2nd edition, Springer New York (1998).
\bibitem{Ginsburg}
V. Ginsburg, \emph{Characteristic varieties and vanishing cycles}, Inventiones mathematicae, 84 (1986), 327-402.
\bibitem{SGA5}
A. Grothendieck {\it et al.}, \emph{Cohomologie $\ell$-adique et fonctions L.}
S\'eminaire de G\'eom\'etrie Alg\'ebrique du Bois-Marie 1965--1966 (SGA 5).
dirig\'e par A. Grothendieck avec la collaboration de I. Bucur, C. Houzel, L. Illusie, J.-P. Jouanolou et J-P. Serre. Lecture Notes in Mathematics  589, Springer-Verlag, Berlin-Heidelberg-New York, (1977).
\bibitem{Grothendieck}
A. Grothendieck, \emph{R\'ecoltes et Semailles, R\'eflexions et t\'emoignages sur un pass\'e de math\'ematicien}, \href{http://lipn.univ-paris13.fr/\~{}duchamp/Books\&more/Grothendieck/RS/pdf/RetS.pdf}{http://lipn.univ-paris13.fr/\~{}duchamp/Books\&more/Grothendieck/RS/pdf/RetS.pdf}
\bibitem{Hartshorne}
R. Hartshorne, \emph{Algebraic geometry}, Springer-Verlag, Berlin-Heidelberg-New York, 1977.

\bibitem{ILO14}
L. Illusie, Y. Lazslo, and F. Orgogozo, \emph{Travaux de Gabber sur l'uniformisation locale et la cohomologie \'etale 
des sch\'emas quasi-excellents,} Ast\'erisque, vol. 361-362, Soc. Math. France, 2014, S\'eminaire \`a l'\'Ecole Polytechnique 2006-2008.

\bibitem{Illusie}
L. Illusie, \textit{Th\'eorie de Brauer et caract\'eristique d'Euler-Poincar\'e, d'apr\`es Deligne}, Caract\'eristique d'Euler-Poincar\'e, Expos\'e VIII, Ast\'erisque 82-83 (1981): 161-172.

%





\bibitem{Kato_Saito:1983}
K. Kato and S. Saito, \textit{Unramified class field theory of arithmetical surfaces}, Annals of Mathematics, 118 (1983): 241-275.
\bibitem{Kato_Saito}
K. Kato and T. Saito, \textit{Ramification theory for varieties over a perfect field}, Annals of Mathematics, 168 (2008): 33-96.


\bibitem{Katz88}
N. M. Katz, \emph{Gauss sums, Kloosterman sums, and monodromy groups}, volume 116 of Annals of Mathematics Studies, Princeton University Press,
Princeton, NJ, 1988.

\bibitem{Katz87}
N. M. Katz, \emph{Travaux de Laumon}, S\'eminaire N. Bourbaki, 1987-1988, exp. n 691, p. 105-132.



\bibitem{KW}
R. Kiehl and R. Weissauer, \emph{Weil conjectures, perverse sheaves and l-adic Fourier
transform}, Ergebnisse der Math. 42, Springer-Verlag, Berlin, 2001.
\bibitem{Laumon}
G. Laumon, \textit{Transformation de Fourier, constantes d'\'equations fonctionnelles et conjecture de Weil}, Publications Math\'ematiques de l'IH\'ES 65 (1987): 131-210. 

\bibitem{Nagata}
M. Nagata, \emph{A generalization of the imbedding problem of an abstract variety in a complete variety}, J. Math. Kyoto Univ. 3 (1963): 89-102.
\bibitem{Pat12} D. Patel, \emph{De Rham $\varepsilon$-factors}, Inventiones Mathematicae,  190 (2012): 299-355.
\bibitem{Pat17b} D. Patel, \emph{K-theory of algebraic microdifferential operators}, preprint. 

\bibitem{Raskind}
W. Raskind, \emph{Abelian class field theory of arithmetic schemes}, K-Theory and Algebraic Geometry: Connections with Quadratic Forms and Division Algebras (Santa Barbara, CA, 1992), Proc. Sympos. Pure Math., vol. 58, Amer. Math. Soc., Providence, RI, (1995): 85-187. 
\bibitem{Saito:1984}
S. Saito, \textit{Functional equations of $L$-functions of varieties over finite fields}, Journal of the Faculty of Science, the University of Tokyo. Sect. IA, Mathematics, Vol.31, No.2 (1984): 287-296.



\bibitem{Saito17b}
T.~Saito, \textit{Characteristic cycles and the conductor of direct image}, \href{https://arxiv.org/abs/1704.04832}{arXiv:1704.04832}. 

\bibitem{Saito17}
T. Saito, \textit{On the proper push-forward of the characteristic cycle of a constructible sheaf}, accepted for publication at Proceedings of Symposia in Pure Mathematics, \href{https://arxiv.org/abs/1607.03156}{arXiv:1607.03156}.

\bibitem{Saito}
T. Saito, \textit{The characteristic cycle and the singular support of a constructible sheaf}, Inventiones mathematicae, 207 (2017): 597-695.


\bibitem{Saito_Yatagawa}
T. Saito and Y. Yatagawa, \textit{Wild ramification determines the characteristic cycle}, Annales Scientifiques de l'ENS 50, fascicule 4 (2017): 1065-1079.

\bibitem{Sai93ep}
T. Saito, \emph{ $\varepsilon$-factor of a tamely ramified sheaf on a variety}, Inventiones mathematicae, 113 (1993): 389-417.




\bibitem{Serre97}
J. P.  Serre, \emph{Linear representations of finite groups}, Graduate Texts in Mathematics 42, Springer-Verlag, New York-Berlin-Heidelberg, 1997.

\bibitem{Serre68}
J. P. Serre, \emph{Groupe de Grothendieck des sch\'emas en groupes r\'eductifs d\'eploy\'es},  Publications Math\'ematiques de l'IH\'ES 34 (1968): 37-52.


\bibitem{Tate2}
J. Tate, \emph{Number-theoretic background}, in Automorphic forms, representations and $L$-functions, PSPM, vol 33, part II. A. M. S., Providence (1979): 3-26.

\bibitem{Tate}
J. Tate, \emph{Fourier analysis in number fields and Hecke's zeta functions}, in \emph{Algebraic Number Theory}, London, Academic Press (1967): 305-347. 


\bibitem{NYZ}
N. Umezaki, E. Yang and Y. Zhao, \emph{A blow up formula for Gysin pull-back}, \url{https://yangenlin.files.wordpress.com/2018/03/buf.pdf}{}.


\bibitem{vi1}
I. Vidal,  \textit{Formule du conducteur pour un caract\`ere l-adique}, Compos. Math., 145  (2009): 687-717.
\bibitem{vi2}
I. Vidal, \textit{Formule de torsion pour le facteur epsilon d'un caract\`ere sur une surface}, Manuscripta math., 130 (2009): 21-44. 

\bibitem{weibel}
C. A. Weibel, \emph{An introduction to homological algebra}, Cambridge studies in advanced mathematics 38, 1994.

\end{thebibliography}
\end{document}